\documentclass[12pt]{amsart}%
\usepackage{amsfonts}
\usepackage{epsfig}
\usepackage{graphicx}
\usepackage{amsmath}
\usepackage{amsfonts}
\usepackage{amssymb}
\usepackage{latexsym}
\usepackage{rotating}
\usepackage{pstricks, pst-node, pst-text, pst-3d}
\usepackage{amsbsy}
\usepackage{bm}
\usepackage{dsfont}
\usepackage[all]{xy}

\input{amssym.def}
\newtheorem{theorem}{Theorem}
\newtheorem{proposition}{Proposition}
\newtheorem{corollary}{Corollary}
\newtheorem{lemma}{Lemma}

\newtheorem{remark}{Remark}

\newtheorem{definition}{Definition}
\theoremstyle{remark}
\newtheorem{example}{Example}

\newcommand{\Id}{\text{\rm Id}}

\newcommand{\ad}{\text{\rm ad}}
\newcommand{\spn}{\text{\rm span}\,}
\newcommand{\tr}{\text{\rm tr}}
\newcommand{\rank}{\text{\rm rank}}
\newcommand{\GL}{\text{\rm GL}}
\newcommand{\Cl}{\text{\rm Cl}}
\newcommand{\SO}{\text{\rm SO}}
\newcommand{\Orth}{\text{\rm O}}
\newcommand{\End}{\text{\rm End}}
\newcommand{\diag}{\text{\rm diag}}

\newcommand{\Aut}{\text{\rm Aut}}
\newcommand{\Hom}{\text{\rm Hom}}

\newcommand{\so}{\mathfrak{so}}
\newcommand{\g}{\mathfrak{g}}
\newcommand{\h}{\mathfrak{h}}
\newcommand{\n}{\mathfrak{n}}
\newcommand{\orth}{\mathfrak{o}}

\newcommand{\R}{\mathbb{R}}

\DeclareMathOperator{\la}{\langle}
\DeclareMathOperator{\ra}{\rangle}

\input epsf
\begin{document}
\title[Pseudo-metric Lie algebras]{Pseudo-metric 2-step nilpotent Lie algebras}
\author[C.~Autenried, K.~Furutani, I.~Markina,  and A.~Vasil'ev]{Christian Autenried$^{\dag}$, Kenro Furutani, Irina Markina$^{\dag}$, and Alexander Vasil'ev$^{\dag}$}

\thanks{The authors$^{\dag}$ have been  supported by the grants of the Norwegian Research Council 
\#213440/BG; and EU FP7 IRSES program STREVCOMS, grant  no.
PIRSES-GA-2013-612669}
 
\subjclass[2010]{Primary: 15A66 17B30, 22E25} 

\keywords{Nilpotent Lie algebra, nilmanifold,  $H$-type Lie algebra, non-degenerate scalar product, isomorphism, Lie triple system, lattice}

\address{K.~Furutani:  Department of Mathematics, Faculty of Science and Technology, Science University of Tokyo, 2641 Yamazaki, Noda, Chiba (278-8510), Japan}
\email{furutani\_kenro@ma.noda.tus.ac.jp}

\
\address{C.~Autenried, I.~Markina, and A.~Vasil'ev: Department of Mathematics, University of Bergen, P.O.~Box 7800,
Bergen N-5020, Norway}
\email{christian.autenried@math.uib.no}
\email{irina.markina@math.uib.no}
\email{alexander.vasiliev@math.uib.no}

\begin{abstract}
The metric approach to studying 2-step nilpotent Lie algebras by making use of non-degenerate scalar products is realised. We show that any 2-step nilpotent Lie algebra is isomorphic to its standard pseudo-metric form, that is a 2-step nilpotent Lie algebra endowed with some standard non-degenerate scalar product compatible with Lie brackets.
This choice of the standard pseudo-metric form allows to study the isomorphism properties. If the elements of the centre of the standard pseudo-metric form constitute a Lie triple system of the pseudo-orthogonal Lie algebra, then the original 2-step nilpotent Lie algebra admits integer structure constants. Among  particular applications we prove that  pseudo $H$-type algebras have  bases with rational structural constants, which implies that the corresponding pseudo $H$-type groups admit lattices.\end{abstract}
\maketitle



\section{Introduction and statement of main results}


The present article is inspired by two series of works devoted to the study of 2-step nilpotent Lie algebras by means of the scalar products. In 80' A.~Kaplan introduced  Lie algebras of Heisenberg type that he called $H$-type algebras~\cite{Kaplan, Kap2}, that are 2-step Lie algebras endowed with a positive-definite scalar product compatible with the Lie structure. $H$-type algebras and their groups became a fruitful source of research related to sub-elliptic operators and the geometry associated with these differential operators, which nowadays is called sub-Riemannian geometry, see e.g.,~\cite{CCFI,CDPT,CM,CDKR,Gaveau,Kor1,Riemann,Riemann1}. The metric approach was extended and generalised by P.~Eberlein~\cite{Eber04,Eber03,Eber02} to a study of arbitrary 2-step nilpotent Lie algebras and their Lie groups. He introduced a standard metric 2-step nilpotent Lie algebra, that is isomorphic to the direct sum $\mathbb R^m\oplus W$, with the centre $W\subset \so(m)$. The Lie brackets are uniquely defined by the Euclidean product on $\mathbb R^m$ and the trace product on $\so(m)$ by the equality
\begin{equation}\label{eq:intr}
(w(x),y)_{\mathbb R^m}=(w,[x,y])_{\so(m)},\quad x,y\in \mathbb R^m,\quad w\in W.
\end{equation}
One of his results states that any 2-step nilpotent Lie algebra is isomorphic to some standard metric 2-step nilpotent Lie algebra~\cite{Eber03}. The Heisenberg type Lie algebras are  standard metric Lie algebras related to the representations of the Clifford algebras $\Cl(\mathbb R^n)$. Namely, let the Euclidean space $\mathbb R^n$ generate the Clifford algebra $\Cl(\mathbb R^n)$, and let $J\colon \Cl(\mathbb R^n)\to \End(\mathbb R^m)$ be a representation of $\Cl(\mathbb R^n)$. We denote $W:=J(\mathbb R^n)\subset \so(m)\subset \End(\mathbb R^m)$. Then the Heisenberg type Lie algebra is the 2-step nilpotent Lie algebra $\mathbb R^m\oplus W$ with the Lie brackets defined by~\eqref{eq:intr}.

Later,  an analog to the Heisenberg type Lie algebra was introduced in~\cite{Ciatti,GKM}, and studied in~\cite{AFM,Ciatti1,CP,FM,JPP}. Since these type of algebras are related to the representations of Clifford algebras generated by a vector space with an indefinite scalar product, they were called pseudo $H$-type Lie algebras. The pseudo $H$-type Lie algebras naturally carry a pseudo-metric, and therefore, it would be inconvenient to consider them as  standard Lie algebras with a positive-definite scalar product. In the present work we extend the notion of a standard metric 2-step nilpotent Lie algebra, which allows us to consider  2-step nilpotent Lie algebras with an arbitrary non-degenerate scalar product. Particularly, we show results analogous to those of P.~Eberline~\cite{Eber03}, establishing an isomorphism between an arbitrary 2-step nilpotent Lie algebra and a standard (pseudo-) metric 2-step nilpotent Lie algebra. 

The structure of the work is as follows. We collect  notations and necessary definitions in Section~\ref{sec:prelim}. Section~\ref{sec:pseudo_metric} is devoted to the definition of a standard pseudo-metric form for a 2-step nilpotent Lie algebra. Here the main result states that any 2-step nilpotent Lie algebra is isomorphic to a properly chosen standard pseudo-metric 2-step nilpotent Lie algebra. In Section~\ref{sec:isom}, we formulate some properties of isomorphic Lie algebras in terms of a chosen pseudo-metric. In Section~\ref{sec:lts}, we collect useful facts about Lie triple systems of the pseudo-orthogonal Lie algebra $\so(p,q)$. We show that in the case when the Lie triple system has a trivial centre it forms a rational subspace of a specially chosen subalgebra $\mathcal L$ of $\so(p,q)$. In Section~\ref{sec:int_str}, we explain the construction of a 2-step nilpotent Lie algebra with the centre isomorphic to a Lie triple system of $\so(p,q)$. We prove that if the Lie triple system is a rational subspace of $\mathcal L$, then the constructed 2-step nilpotent Lie algebra has rational structural constants. It leads to the existence of a lattice on the corresponding Lie group.


\section{Preliminaries}\label{sec:prelim}



\subsection{Clifford algebras and their representations}


Let $V$ be a real vector space endowed with a non-degenerate quadratic form $Q(v)$, $v\in V$, which defines a symmetric bilinear form $(u,v)=\frac{1}{2}(Q(u+v)-Q(u)-Q(v))$ by polarization. The {\it Clifford algebra} $\Cl\big(V,(.\,,.)\big)$, named after the English geometer  William Kingdon Clifford~\cite{Cliff}, is an associative unital algebra with the unit $\mathds{1}$ freely generated by $V$ modulo the relations
$$
v\otimes v=-Q(v)\mathds{1}=-(v,v)\mathds{1}\quad\text{for all $v\in V$ }\qquad\text{or}
$$
\[
 u\otimes v+v\otimes u=-2(u,v)\mathds{1}\quad\text{for all $u,v\in V$}.
\]
For an introductory text, one may look at~\cite{Garling}. Every non-degenerate quadratic form $Q$ on an $n$-dimensional vector space $V$ is equivalent to the standard diagonal form
\[
Q_{r,s}(v)=v_1^2+v_2^2+\dots +v_r^2-v_{r+1}^2-\dots-v_{r+s}^2, \quad n=r+s.
\]
Using the isomorphism $(V,Q)\simeq (\mathbb R^{r,s},Q_{r,s})$ we will work with the Clifford algebra $\Cl_{r,s}=\Cl(\mathbb R^{r,s},Q_{r,s})$ that is isomorphic to $\Cl\big(V,(.\,,.)\big)$.  Starting with an orthonormal basis $\ell_1,\dots, \ell_n$ in $\mathbb R^{r,s}$, $r+s=n$, one defines 
a linear basis of $\Cl_{r,s}$ by the sequence
\[
1,\ell_1,\ldots,\ell_j,\ldots, (\ell_{k_1}\otimes \ldots\otimes  \ell_{k_j}),\,\,1\leq k_1<k_2<\dots<k_j\leq n,\,\,j=1,2,\dots, n.
\]
It follows that the dimension of $\Cl_{r,s}$ is $2^n$, $n=r+s$.

A {\it Clifford module} is a representation space of a Clifford algebra where the multiplication by elements of Clifford algebra is defined and satisfies certain axioms~\cite{Atiyah}. 
A Clifford  module for $\Cl_{r,s}$ is a finite-dimensional real space $U$ and a linear map $J\colon \Cl_{r,s}\to {\rm End}(U)$, satisfying the Clifford relation $J(v)\circ J(v)=J^2(v)=-(v,v)\Id_{U}$ or
$J(u)\circ J(v)+J(v)\circ J(u)=-2(u,v)\Id_{U}$ for $u,v\in\mathbb R^{r,s}$. The matrix representation is given by anti-commuting matrices. Here and in what follows, by saying {\it scalar product} we mean a
non-degenerate symmetric real bilinear form, and by {\it inner product}, a positive-definite one. 

Clifford modules $(U,J)$ and $(U',J')$ for $\Cl_{r,s}$ are isomorphic (or equivalent) if there is an isomorphism $\phi\colon U\to U'$, such that $\phi\circ J(v)=J'\circ\phi(v)$ for all $v\in U$. A linear subspace $W\subset U$ is a submodule if it is invariant under $J$. A Clifford module $(U,J)$ is irreducible if the only submodules are $U$ and $\{0\}$.


\subsection{Pseudo $H$-type Lie algebras}\label{subseq:pseudoH}


The Clifford modules introduced in the previous section lead to the construction of Lie algebras, that carry a scalar product, and in some sense, generalise the Heisenberg algebra. To reveal this relation, we start recalling known relations between the Clifford modules and the composition of quadratic forms.

\begin{definition}\label{def:adm_mod}
 Let $U$ be a vector space, $J\colon \Cl_{r,s} \to {\rm End}(U)$ a representation map, and let $(.\, ,.)_U$ be a scalar product on $U$. The module $(U,J)$ is called admissible with respect to $(.\, ,.)_U$ if the map $J$ is skew-symmetric  $(J(z)(u),v)_U=-(u,J(z)(v))_U$ for any $z\in \mathbb R^{r,s}$. We write $(U,J,(.\, ,.)_U)$ for an admissible $\Cl_{r,s}$-module.
\end{definition}

The next definition concerns with the {\it composition of quadratic forms}. 
Let $(W,(.\,,.)_{W})$, $(U,(.\,,.)_{U})$ be two vector spaces with corresponding quadratic forms, which we write as symmetric bilinear forms. 
\begin{definition}\label{def:comp}
A bilinear map $\mu\colon W\times U\to U$ is called a composition of $(W,(.\,,.)_{W})$ and $(U,(.\,,.)_{U})$ if  the following equality 
\begin{equation}\label{eq:composition}
\big(\mu(w,u),\mu(w,u)\big)_{U}=\big( w,w\big)_{W}\big(u,u\big)_{U},
\end{equation}
holds for any $w\in W$ and $u\in U$.
\end{definition}
Formula~\eqref{eq:composition} can be written in a non-symmetric form
\begin{equation}\label{eq:nonsym1}
\big(\mu(w,u),\mu(w',u)\big)_{U}=\big( w,w'\big)_{W}\big(u,u\big)_{U}\qquad\text{or}
\end{equation}
\begin{equation}\label{eq:nonsym2}
\big(\mu(w,u),\mu(w,u')\big)_{U}=\big( w,w\big)_{W}\big(u,u'\big)_{U}.
\end{equation}
We assume that there is $w_0\in W$ such that $(w_0,w_0)_{W}=1$ and $\mu(w_0,u)=u$. It always can be done by normalisation procedure of a quadratic form $(.\,,.)_{W}$ and redefinition of $\mu$, see~\cite{Lam}. Let us denote by $\mathcal Z$ the orthogonal complement to the space $W_0=\spn\{w_0\}$ with respect to $(.\,,.)_{W}$ and by $J$  the restriction of $\mu$ to $\mathcal Z$, thus $J\colon \mathcal Z\times U\to U$. We also write $J_z$ or $J(z)$ for a fixed value of $z\in \mathcal Z$. Applying formula~\eqref{eq:nonsym1}, we obtain
\begin{equation}\label{eq:orth_J}
\big(J_z(u),u\big)_{U}=\big(\mu(z,u),\mu(w_0,u)\big)_{U}=\big( z,w_0\big)_{W}\big(u,u\big)_{U}=0.
\end{equation}
The last equality shows that the map $J$ is skew-symmetric with respect to $(\cdot\,,\cdot)_{U}$ in the following sense
\begin{equation}\label{eq:skew_sym}
\big(J_z(u),u'\big)_{U}+\big( u,J_z(u')\big)_{U}=0
\end{equation}
for any $z\in \mathcal Z$ and $u,u'\in U$. Indeed, taking into account~\eqref{eq:orth_J}, we get
$$
0=(J_z(u+u'),u+u')=(J_z(u),u')+(J_z(u'),u).
$$
We conclude that having a normalised composition map, one can always construct a skew-symmetric map from it.

Now we state two theorems which describe relations between Definitions~\ref{def:adm_mod} and~\ref{def:comp}. For details of the proof of Theorems~\ref{CompositionToClifford} and~\ref{CliffordToComposition}, see~\cite[Theorem 5.5, Remark 5.7]{Lam}.

\begin{theorem}\label{CompositionToClifford}
Let $(W,(.\,,.)_{W})$ and $(U,(.\,,.)_{U})$ be two vector spaces with scalar products, and let $\mu\colon W\times U\to U$ be a composition of quadratic forms~\eqref{eq:composition} which is normalised by $\mu(w_0,u)=u$ with $W=\spn\{w_0\}\oplus_{\perp}\mathcal Z$. Then the Clifford algebra $\Cl(\mathcal Z,(.\,,.)_{\mathcal Z})$ admits a representation $J$ on $U$ and $(U,J,(.\, ,.)_U)$ is an admissible $\Cl(\mathcal Z,(.\,,.)_{\mathcal Z})$-module. Here $(.\,,.)_{\mathcal Z}$ is the restriction of $(.\,,.)_{W}$ to $\mathcal Z$.
\end{theorem}
The proof follows from the following observation
\begin{equation}\label{eq:J_Cliff}
\big(u,J^2_zu'\big)_U=\big(u,J_z(J_zu')\big)_U=-\big( J_zu,J_zu'\big)_U=-(z,z)_{\mathcal Z}(u,u')_U.
\end{equation}
Since the scalar product $(.\,,.)_U$ is non-degenerate we get $J_z^2=-( z,z)_{\mathcal Z}\Id_{U}$, whenever $z\in \mathcal Z$.

\begin{theorem}\label{CliffordToComposition}
Let $(U,J,(.\,,.)_U)$ be an admissible $\Cl(\mathcal Z,(.\,,.)_{\mathcal Z})$-module.  
Then there exist a vector space $W=W_0\oplus \mathcal Z$, and a scalar product $(.\,,.)_{W}$, such that they admit a composition $\mu\colon W\times U\to U$ of $(.\,,.)_{W}$ and $(.\,,.)_{U}$. Here $J$ is the restriction of $\mu$ from $W$ to the space $\mathcal Z$, and $W_0$ is one dimensional vector space orthogonal to $\mathcal Z$ with respect to $(.\,,.)_{W}$.
\end{theorem}

The skew-symmetric map $J\colon \mathcal Z\to \End(U)$, relating an admissible Clifford module and a composition of quadratics forms, can be used for the construction of 2-step nilpotent Lie algebras in the following way. Recall, that a Lie algebra $\mathfrak g$ is called 2-step nilpotent, if it satisfies the condition $[[\mathfrak g,\mathfrak g],\mathfrak g]=\{0\}$. Define the skew-symmetric bilinear form $[.\,,.]\colon U\times U\to\mathcal  Z$ by
\begin{equation}\label{eq:J}
(J_zu,u')_U = (z, [u,u'])_{\mathcal Z}.
\end{equation}
It is straightforward to verify that the map $[.\,,.]$ satisfies the Jacobi identity and that the set $\mathcal Z$ is the centre of the 2-step nilpotent Lie algebra $\mathfrak n=(U\oplus_{\perp}\mathcal Z, [.\,,.])$. The defining relation~\eqref{eq:J} shows that $J_{(\cdot)}(u)\colon \mathcal Z\to U$ is the adjoint map to the linear map $\ad_u(\cdot)\colon U\to \mathcal Z$ with respect to the metric $(.\,,.)=(.\,,.)_U+(.\,,.)_{\mathcal Z}$. Let us give some formal definitions.

\begin{definition}\label{def:general}\cite{GKM}
Let $\mathfrak n=\big(U\oplus_{\bot}\mathcal Z,[.\,,.],(.\,,.)\big)$ be a 2-step nilpotent Lie algebra, and let $\mathcal U_u$ be the orthogonal complement to the kernel $\ker(\ad_u)$ of the linear map $\ad_u(\cdot):=[u,\cdot]\colon U\to \mathcal Z$ for some $u\in U$. Then the Lie algebra $\mathfrak n$ is called of general $H$-type if 
$$
\ad_u\colon (\mathcal U_u, (.\,,.)_{\mathcal U_u})\to (\mathcal Z,(.\,,.)_{\mathcal Z})
$$ 
is a surjective isometry for all $u\in U$ with $(u,u)=1$, and a surjective anti-isometry for all $u\in U$ with $(u,u)=-1$. 
\end{definition}
In Definition~\ref{def:general}, it is assumed that the spaces $(\mathcal U_u, (.\,,.)_{\mathcal U_u})$, $(\mathcal Z,(.\,,.)_{\mathcal Z})$ are non-degenerate, where $(.\,,.)_{\mathcal U_u}$ and $(.\,,.)_{\mathcal Z}$ are the restrictions of $(.\,,.)$ to the spaces $\mathcal U_u$ and $\mathcal Z$, respectively. 

If the scalar product $(.\,,.)$ is positive-definite, then Definition~\ref{def:general} coincides with the definition of A.~Kaplan in~\cite{Kaplan}, who called these kind of algebras {\it $H$-type} in attempt to generalise the Heisenberg algebra. We emphasise, that the word ``general'' stands for the general scalar product, not only for  a positive definite. In works~\cite{GKM,Kaplan} it was shown, that given a general $H$-type Lie algebra, the Lie product $[.\,,.]$ defines a skew-symmetric operator $J\colon \mathcal Z\times U\to U$ by means of~\eqref{eq:J} that satisfies the orthogonality condition
\begin{equation}\label{eq:J_orth}
( J_z(u),J_z(v))_U=(z,z)_{\mathcal Z} ( u, v)_U.
\end{equation}
It can be lifted to the composition of quadratic forms~\eqref{eq:composition} of $(W=W_0\oplus \mathcal Z,(.\,,.)_W)$ and $(U(.\,,.)_U)$. The converse statement is also true. All  $H$-types algebras descend from a composition of some quadratic forms. The map $J(\cdot)(u)$ is not only the formal adjoint to $\ad_u(\cdot) $, but also the inverse map to the (anti-)isometry $\ad_u\colon \mathcal U\to \mathcal Z$, see~\cite{GKM}.

P.~Ciatti~\cite{Ciatti} introduced the following notion.

\begin{definition}\label{def:pseudo}\label{Ciatti}
A 2-step nilpotent Lie algebra $\mathfrak n=\big(U\oplus_{\bot}\mathcal Z,[\cdot\,,\cdot],(\cdot\,,\cdot)\big)$
is called of pseudo $H$-type, if 
the map $J\colon \mathcal Z\times U\to U$ defined by~\eqref{eq:J}
satisfies the orthogonality condition~\eqref{eq:J_orth}.
\end{definition} 

In the same work it was proved that a 2-step nilpotent Lie algebra $\mathfrak n=\big(U\oplus_{\bot}\mathcal Z,[.\,,.],(.\,,.)\big)$
is a pseudo $H$-type algebra, if and only if, $(U,J,(.\,,.)_U)$ is an admissible module for the Clifford algebra $\Cl(\mathcal Z,(.\,,.)_{\mathcal Z})$. Here the map $J$ is defined by~\eqref{eq:J} and the restrictions $(.\,,.)_U$ and $(.\,,.)_{\mathcal Z}$ of the scalar product $(.\,,.)$ are supposed to be non-degenerate.
All in all, we conclude the following equivalence between the definitions.
\begin{itemize}
\item{\it 
The general $H$-type algebra} exists, if and only if, {\it the composition of quadratic form} exists, see~\cite{GKM,Kaplan}.
\item
 {\it The composition of quadratic form} exists, if and only if,
{\it the admissible module} exists, see~\cite{Lam}.
\item
{\it The admissible module} exists, if and only if, 
{\it the pseudo $H$-type algebra} exists, see~\cite{Ciatti}.
\end{itemize}
Discussions about equivalence of Definitions~\ref{def:general} and~\ref{def:pseudo} can be found also in~\cite{AFM}. From now on, we will use the term pseudo $H$-type algebras, because it was introduced in~\cite{Ciatti} before the general $H$-type algebras~\cite{GKM}.

Relation~\eqref{eq:J} implies the skew symmetry~\eqref{eq:skew_sym} of the map $J$
that jointly with the orthogonality~\eqref{eq:J_orth} leads to the defining property of the Clifford module
\begin{equation}\label{eq:J2_Cliff}
J^2_z=-(z,z)_{\mathcal Z}\Id_{U}
\end{equation} 
as was shown in~\eqref{eq:J_Cliff}.
Moreover, any two conditions from~\eqref{eq:skew_sym},~\eqref{eq:J_orth}, and~\eqref{eq:J2_Cliff} imply the third one. 

\begin{remark}\label{rem:neutral}
Examples of non-admissible modules were constructed in~\cite{Ciatti}.

Due to the skew symmetry of $J_z$, the orthogonal complement of a submodule is again a submodule.

 If a $\Cl_{r,s}$-module is admissible for $s>0$, then the representation space is  neutral with respect to the scalar product, which means that the dimension of the representation space is even and the dimensions of maximal subspaces, where the scalar product is positive definite or negative definite, coincide. If $s=0$, then the module is admissible with respect to any inner product. 
\end{remark}


\subsection{Lattices and nilmanifolds}


One of the aims of this paper is to prove that the pseudo $H$-type groups admit  lattices, or equivalently, the  corresponding pseudo $H$-type algebras admit a basis with rational structural constants. We explain this relation.

\begin{definition} A subgroup $K$ of $G$ is called (co-compact) lattice if $K$ is discrete and the right quotient $K\backslash G$ is compact. The space $K\backslash G$ is called a compact nilmanifold or a compact $2$-step nilmanifold if $G$ is a $2$-step nilpotent Lie group.
\end{definition}

\begin{theorem}[Mal'cev criterion \cite{Malc}]\label{MC} The group $G$ admits a lattice $K$, if and only if,  the Lie algebra $\mathfrak g$ admits a basis $\mathcal B=\{b_1,\dots, b_n\}$
with rational structural constants $[b_i,b_j]=\sum_{k=1}^{n}C_{ij}^k b_k$, $C_{ij}^k\in\mathbb Q$.
\end{theorem}

We denote the
Lie exponent and the Lie logarithm by $\exp\colon \mathfrak g\to G$ and $\log\colon G\to \mathfrak g$ respectively.
Given a lattice $K$, one can construct the corresponding basis $\mathcal B$ as follows. Set $\mathfrak g_{\mathbb Q}=\spn_{\mathbb Q}\log K$, which is a Lie algebra over the field $\mathbb Q$. Denote by $\mathcal B_{\mathbb Q}$ a $\mathbb Q$-basis in $\mathfrak g_{\mathbb Q}$. Then it is also an $\mathbb R$-basis $\mathcal B$ in $\mathfrak g$.

Conversely, given a basis $\mathcal B$ defined as in Theorem~\ref{MC}, let $\Lambda$ be a vector lattice in  $\mathfrak g$, such that $\Lambda\subset \spn_{\mathbb Q}\mathcal B$. Then the lattice $K$ is generated by the elements $\exp \Lambda$, and $\spn_{\mathbb Q}(\log K)=\spn_{\mathbb Q}\mathcal B$.


\subsection{Standard metric 2-step nilpotent Lie algebras}\label{subseq:standard_metric}


In this section, we present shortly ideas from~\cite{Eber04,Eber03}, showing that any 2-step nilpotent Lie algebra can be endowed with a canonical positively definite scalar product and the choice of this inner product is unique up to the Lie algebra isomorphism. Inspired by the definition of general $H$-type algebras, in Section~\ref{sec:pseudo_metric} we generalise the ideas from~\cite{Eber04,Eber03}, showing that actually a non-degenerate scalar product of any index can be chosen. 

Through out the present work we assume that a 2-step nilpotent Lie algebra $\mathfrak g$ has a commutator ideal $[\g,\g]$ of dimension $n$  and its complement is of dimension $m$.
A basis $\mathcal B=\{v_1,\ldots,v_m,z_1,\ldots,z_n\}$ of the Lie algebra $\g$ is called {\it adapted} if $\{z_1,\ldots,z_n\}$ is a basis of $[\g,\g]$. Define the skew-symmetric $(m\times m)$-matrices $C^1,\ldots, C^n$ by 
\[
[v_\alpha,v_\beta]=\sum_{k=1}^{n}C^k_{\alpha\beta}z_k.
\]
Matrices $C^k$ are elements of the Lie algebra $\so(m)$, and they are linearly independent in $\so(m)$, see~\cite{Eber03}. 
Then the $n$-dimensional subspace $\mathcal C^n=\spn\{C^1,\ldots,C^n\}\subset\so(m)$ is isomorphic to $[\g,\g]=\spn \{z_1,\ldots,z_n\}$ and is called the structure space determined by the adapted basis $\mathcal B$. The vector space $$\spn\{v_1,\ldots,v_m\}\oplus\spn\{z_1,\ldots,z_n\}$$ of the 2-step nilpotent Lie algebra $\g$ is isomorphic to the direct sum $\R^m\oplus\mathcal C^n$. The $n$-dimensional subspace $\mathcal C^n\subset\so(m)$ depends on the choice of the adapted basis, nevertheless all possible subspaces defined by an arbitrary choice of an adapted bases form the set $\{A\mathcal C^n A^{\mathbf t}\mid\ A\in\GL(m)\}$, where $A^{\mathbf t}$ is the transpose of $A$.

The spaces $\mathbb R^m$ and $\mathcal C^n\subset\so(m)$ have a natural choice of inner products that will define the Lie algebra product on $\mathcal G=\mathbb R^m\oplus\mathcal C^n$. Denote by $\langle .\,,.\rangle_{\so(m)}$ the positive-definite product on $\so(m)$ defined by 
\[
\langle Z,Z'\rangle_{\so(m)}=-\tr(ZZ'),
\]
and denote by $\langle.\,,.\rangle_m$ the standard Euclidean inner product in $\R^m$. The notation $\langle .\,, .\rangle_{\so(m)}$ is also used for the restriction of this inner product on $\mathcal C^n\subset\so(m)$. Then the inner product $(.\,,.)=\langle .\,,. \rangle_m+\langle .\,, .\rangle_{\so(m)}$ makes the direct sum $\mathcal G=\R^m\oplus \mathcal C^n$ orthogonal. Let $[.\,,.]$ be a unique Lie product on $\mathcal G$, such that $\mathcal C^n$ belongs to the centre of $\mathcal G$, and 
\[
\langle Z(x),y\rangle_m=\langle Z,[x,y]\rangle_{\so(m)}\quad\text{for arbitrary}\quad x,y\in\R^m,\ \ Z\in\mathcal C^n,
\]
where $Z(x)$ simply denotes the action of $Z\in\mathcal C^n\subset\so(m)$ on a vector $x\in\R^m$ defined by matrix multiplication.
It is easy to see that $(\mathcal G,[\cdot\,,\cdot])$ is a 2-step nilpotent Lie algebra, such that $[\mathcal G,\mathcal G]=\mathcal C^n$ and, endowed with the inner product $(.\,,.)=\langle .\,,. \rangle_m+\langle .\,, .\rangle_{\so(m)}$, it is called a standard metric 2-step nilpotent Lie algebra. It was shown in~\cite{Eber04} that {\it any 2-step nilpotent Lie algebra $\g$ is isomorphic to a standard metric 2-step nilpotent Lie algebra $\mathcal G=\big(\R^m\oplus \mathcal C^n,[\cdot\,,\cdot],(\cdot\,,\cdot)\big)$.}


\section{Pseudo-metric on 2-step nilpotent Lie algebras}\label{sec:pseudo_metric}


In this section, we continue to develop the approach proposed in Section~\ref{subseq:standard_metric}. The choice of the Euclidean product in $\mathbb R^m$ is very natural, but it is also possible to choose the metric $\langle x,y \rangle_{p,q}=\sum_{i=1}^{p}x_iy_i-\sum_{i=p+1}^{p+q}x_iy_i$, $p+q=m$ of an arbitrary index $(p,q)$. It leads to the change of the structural space $\mathcal C\in\so(m)$ to the space $\mathcal D\subset\so(p,q)$, and of the positive definite metric on $\so(m)$ to the indefinite metric on $\so(p,q)$. The main motivation of this choice is the following. The standard  metric form for classical $H$-type algebras carries a positive definite scalar product and in this case the Lie algebras are  isometric also as  scalar product spaces. Meanwhile the pseudo $H$-type Lie algebras, introduced in Section~\ref{subseq:pseudoH} are isomorphic (and isometric) to a standard pseudo-metric form with an indefinite scalar product related to the scalar product of the underlying Clifford algebras. Notice, that being  2-step nilpotent Lie algebras, the pseudo $H$-type Lie algebras are also isomorphic to a standard metric form with a positive definite metric, see~\cite{Eber03}, but in this case they are not isometric and the isomorphism neglects the relation with the Clifford algebras generating pseudo $H$-type Lie algebras. The approach using indefinite scalar products also allows us to distinguish those Lie groups which admit  indefinite left invariant metrics.  We also aim to show that any 2-step nilpotent Lie algebra is isomorphic to some metric Lie algebra with an indefinite scalar product.


\subsection{Pseudo-orthogonal groups}


We start reminding the structure of the pseudo-orthogonal group and its Lie algebra.
We use the notation $\eta_{p,q}=\diag(I_{p},-I_{q})$ for diagonal $(m\times m)$-matrix, $m=p+q$, having the first $p$ entries on the main diagonal $1$ and the last $q$ equal to $-1$. Further we continue to use $I_p$ to denote the $(p\times p)$ unit matrix. Let $\langle.\,,.\rangle_{p,q}$ be a scalar product in $\mathbb R^m$, $p+q=m$, defined by the matrix $\eta_{p,q}$: $
\langle x,y\rangle_{p,q}=x^{\mathbf t}\eta_{p,q} y$ for $x,y\in\mathbb R^m$,
where $x^{\mathbf t}$ is the  transpose to $x$. We use the following notation
\begin{equation}\label{eq:nu}
\nu_i=\nu_i(p,q)=\begin{cases}
1,\quad &\text{if}\quad 1\leq i\leq p,
\\
-1,\quad&\text{if}\quad p+1\leq i\leq p+q=m,
\end{cases}
\end{equation}
to indicate the sign in the scalar product of the vectors from the orthonormal basis of $\mathbb R^{p,q}$.
A vector $x\in \mathbb R^{p,q}$ is called
\begin{itemize}
\item spacelike if $\langle x,x\rangle_{p,q}>0$ or $x= 0$,
\item timelike if $\langle x,x\rangle_{p,q}<0$,
\item null if $x\neq 0$ and $\langle x,x\rangle_{p,q}=0$.
\end{itemize}

We denote by $\Orth(p,q)$ the pseudo-orthogonal group
$$\Orth(p,q)=\{ X\in
GL(m)|\,\,X^{\mathbf t}\eta_{p,q} X=\eta_{p,q}\},
$$
where $X^{\mathbf t}$ is the matrix transposed to $X$. The pseudo-orthogonal group preserves the scalar product $\langle\cdot\,,\cdot\rangle_{p,q}$ in the following sense
$$
\langle Xx,Xy\rangle_{p,q}=x^{\mathbf t}X^{\mathbf t}\eta_{p,q} Xy=x^{\mathbf t}\eta_{p,q} y=\langle x,y\rangle_{p,q},\quad x,y\in\mathbb R^{p,q},\ \ p+q=m.
$$ 
The inverse $X^{-1}$ of $X$ is given by $X^{-1}=\eta_{p,q}X^{\mathbf t}\eta_{p,q}$. For any matrix $A$ define the matrix $A^{\eta_{p,q}}$ by $A^{\eta_{p,q}}:=\eta_{p,q} A^{\mathbf t}\eta_{p,q}$. Thus, if $X\in \Orth(p,q)$, then $X^{\eta_{p,q}}X=XX^{\eta_{p,q}}=I_m$, which implies $X^{\eta_{p,q}}=X^{-1}$.

If we replace $\eta_{p,q}$ by any symmetric  matrix $\tilde \eta$ with $p$ positive
 and $q$ negative eigenvalues, then we get a group isomorphic to $\Orth(p,q)$. Diagonalising the matrix $\tilde\eta$ gives a conjugation of this group with the standard group $\Orth(p,q)$. It follows from the definition that all matrices in $\Orth(p,q)$ have determinant equal to $\pm 1$.
 A matrix  $X\in\Orth(p,q)$ can be written in a block form as
\begin{equation*}\label{blocos}
X=\left[\begin{array}{c|c}
                     X_S&  B
                     \\
                     \hline
                     C  & X_T
                   \end{array}
\right],
\end{equation*}
where $X_S$ and $X_T$ are invertible $(p\times p)$  and  $(q\times q)$ matrices, respectively. An element $X \in \Orth(p,q)$ preserves (reverses) time orientation provided that $\det(X_T)>0$ ($<0$), and preserves (reverses) space orientation provided that  $\det(X_S)>0$ ($<0$). $\Orth(p,q)$ can then be split into four disjoint sets $\Orth^{++}(p,q)$, $\Orth^{+-}(p,q)$, $\Orth^{-+}(p,q)$, and $\Orth^{--}(p,q)$, indexed by the signs of the determinants of $X_S$ and $X_T$, in the respective order. The following three disconnected subgroups of $\Orth(p,q)$ define the orientation on $\mathbb R^{p,q}$:
\begin{equation}
\Orth^{++}(p,q)\cup \Orth^{--}(p,q),\quad \Orth^{++}(p,q)\cup \Orth^{+-}(p,q),\quad \Orth^{++}(p,q)\cup \Orth^{-+}(p,q).
\end{equation}
According to~\cite{ONeill}, we call the first group {\it orientation} preserving, the second one {\it space-orientation} preserving and the last one {\it time-orientation} preserving.
The connected component $\Orth^{++}(p,q)$ contains the identity, preserves time orientation, space orientation, and the orientation of $\mathbb R^{p,q}$. The component $\Orth^{++}(p,q)$ is, in some sense, an analogue of the special orthogonal subgroup $\SO(m)$ of the orthogonal group $\Orth(m)$, and therefore, we use the notation $\SO(p,q)=\Orth^{++}(p,q)$. The group $\Orth(p,q)$ is not compact, but contains the compact subgroups $\Orth(p)$ and $\Orth(q)$ acting on the subspaces on which the scalar product $\langle .\,,.\rangle_{p,q}$ is sign-definite. In fact, $\Orth(p) \times \Orth(q)$ is a maximal compact subgroup of $\Orth(p,q)$, while $S(\Orth(p) \times  \Orth(q))$ is a maximal compact subgroup of $\Orth^{++}(p,q)\cup \Orth^{--}(p,q)$. Likewise, $\SO(p) \times  \SO(q)$ is a maximal compact subgroup of the component $\SO(p, q)$. Thus up to homotopy, the spaces $S(\Orth(p) \times  \Orth(q))$ and $\SO(p) \times  \SO(q)$ are products of (special) orthogonal groups, from which algebraic-topological invariants can be computed.

The Lie algebra of  $\Orth(p,q)$, and thus of $\SO(p,q)$, equipped with the Lie bracket defined by the commutator $[\mathcal A,\mathcal B]=\mathcal A\mathcal B-\mathcal B\mathcal A$, is the set
$$\so(p,q)=\{ \mathcal A\in
\mathfrak{gl}(m)|\,\, \eta_{p,q}\mathcal A^{\mathbf t}\eta_{p,q}=-\mathcal A\}.$$
So, an element $\mathcal X\in\so(p,q)$ satisfies $\mathcal X^{\eta_{p,q}}=-\mathcal X$, and one has
 $
\mathcal X^{\eta_{p,q}}\mathcal X=\mathcal X\mathcal X^{\eta_{p,q}}=-\mathcal X^2$.
In general, for an arbitrary $\mathcal A \in \mathfrak{gl}(m)$ the following is true:
$
(\mathcal A^{\eta_{p,q}})^{\eta_{p,q}}=\mathcal A$ and  $(\mathcal A\mathcal B)^{\eta_{p,q}}=\mathcal B^{\eta_{p,q}}\mathcal A^{\eta_{p,q}}$. 

The Lie algebra $\so(p,q)$ can be equipped with the scalar product $\la .\,,.\ra_{\so(p,q)}$ defined by
 $\la\mathcal X,\mathcal Y\ra_{\so(p,q)}=\tr{(\mathcal X^{\eta_{p,q}}\mathcal Y)}=-\tr(\mathcal X\mathcal Y)$. The scalar product is positive-definite only for $q = 0$. Analogously to the causal structure in $\mathbb R^{p,q}$, we say that a non-zero element $\mathcal X\in \so(p,q)$ is timelike if $\la\mathcal X, \mathcal X\ra_{\so(p,q)}<0$,  is spacelike if $\la\mathcal X, \mathcal X\ra_{\so(p,q)}>0$, and  is null if $\la\mathcal X, \mathcal X\ra_{\so(p,q)}=0$. The zero element is declared to be spacelike. Matrices in $\so(p,q)$ can be written as
$$
\mathcal X=
\begin{pmatrix}
a_{p} & b
\\
b^{\mathbf t}  &a_{q}
\end{pmatrix},\qquad a_{p}\in\so(p),\quad a_{q}\in \so(q).
$$
So, for $\mathcal X\in \so(p,q)$ one has
$$
\la\mathcal X, \mathcal X\ra_{\so(p,q)}  =\tr(\mathcal X^{\eta_{p,q}}\mathcal X) =-\tr(\mathcal X^2)
=
-\tr(a_{p}^2+a_{q}^2)
-2\tr (bb^{\mathbf t}).
$$
As we see, the first term in the right hand side, involving the  skew-symmetric matrices $a_{p}$ and  $a_{q}$, is always positive and represents the spacelike part. The matrix $b$ is responsible for the timelike character of elements of the Lie algebra. The metric defined by trace has index $\big(\frac{p(p-1)+q(q-1)}{2},pq\big)$ as one can see from the dimensions of $\so(p)$ and $\so(q)$.

Notice that if $\mathcal X\in \so(p,q)$ and $x,y\in\mathbb R^m$, $p+q=m$, then
$$
\langle \mathcal Xx,y\rangle_{p,q}=x^{\mathbf t}\mathcal X^{\mathbf t}\eta_{p,q} y=-x^{\mathbf t}\eta_{p,q}\mathcal X y=-\langle x, \mathcal X y\rangle_{p,q}.
$$
Thus matrices from $\so(p,q)$ are skew-symmetric with respect to $\langle .\,,.\rangle_{p,q}$.

At the end of the section we consider a generalisation of the above constructions. Let $(V,\langle.\,,.\rangle_V)$ be a scalar product vector space. We denote by $\orth(V,\langle.\,,.\rangle_V)$ or shortly $\orth(V)$ the subspace of the space $\End(V)$ of linear maps $\mathcal A\colon V\to V$ such that 
\begin{equation}\label{eq:skew}
\langle \mathcal A v,w\rangle_V=-\langle v,\mathcal A w\rangle_V.
\end{equation}
We call $\orth(V)$ the space of {\it skew-symmetric $($with respect to $\langle.\,,.\rangle_V$$)$ maps} and note that it coincides with $\so(p,q)$ when $V=\mathbb R^{p,q}$, and $\langle.\,,.\rangle_V=\langle.\,,.\rangle_{p,q}$. In general, it can be shown that $\orth(V)$  is isomorphic to the space $\so(p,q)$ for any $m$-dimensional scalar product space $(V,\langle.\,,.\rangle_V)$ with a scalar product of index $(p,q)$, $p+q=m$.
We can endow the space $\orth(V)$ with the following scalar product 
$$
\langle \mathcal A,\mathcal B\rangle_{\orth(V)}=-\tr(\mathcal A\mathcal B).
$$
One can prove that the index of $\langle .\,,.\rangle_{\orth(V)}$ is $\big(\frac{p(p-1)+q(q-1)}{2},pq\big)$ by the isomorphism property with $\so(p,q)$.


\subsection{Lie product and compatible scalar product}


The relation between skew-symmetric representations of Clifford algebras and some class of 2-step nilpotent Lie algebras, namely, pseudo $H$-type Lie algebras was described  in Section~\ref{subseq:pseudoH} . This relation is actually more general and can be given for arbitrary skew-symmetric maps and 2-step nilpotent Lie algebras endowed with some scalar product. 
 
{\it From Lie algebras to skew-symmetric maps}. 
Let $(\g,[.\,,.],\langle.\,,.\rangle_{\g})$ be a 2-step Lie algebra with a centre $U$ and a scalar product $\langle.\,,.\rangle_{\g}$ on $\g$. We write $\g=V\oplus_{\perp}U$ where the decomposition is orthogonal with respect to $\langle.\,,.\rangle_{\g}$. Here we also assume that the restriction $\langle.\,,.\rangle_V$ of $\langle.\,,.\rangle_{\g}$ to $V$ is non-degenerate which also leads to the non-degeneracy of the space $U$ with respect to the restriction $\langle.\,,.\rangle_{U}$ of $\langle.\,,.\rangle_{\g}$ to $U$. As it was mentioned before, every $z\in U$ and the Lie product on $\g$ define a map $J_z\colon V\to V$ by
\begin{equation}\label{eq:defJ}
\langle J_zv,w\rangle_V=\langle z,[v,w]\rangle_{U}\quad\text{for all}\quad  v,w\in V\ \ \text{and all}\ \ z\in U.
\end{equation}
It is clear that $J_z$ satisfies~\eqref{eq:skew} and is linear with respect to both variables: $z\in U$ and $v\in V$. Therefore, we conclude that the scalar product and the Lie product together define a linear skew-symmetric map $J\colon U\to \orth(V)$.

{\it From skew-symmetric maps to Lie algebras}. Let now $(V,\langle.\,,.\rangle_V)$ and $(U,\langle.\,,.\rangle_U)$ be two scalar product spaces and $J\colon U\to \orth(V)$. Then the sum $\g=V\oplus U$ is orthogonal with respect to the non-degenerate scalar product $\langle.\,,.\rangle_{\g}=\langle.\,,.\rangle_V+\langle.\,,.\rangle_U$, and we are able to define the Lie bracket for $\g$ by making use of~\eqref{eq:defJ}. Then $\g=\big(V\oplus U,[.\,,.],\langle.\,,.\rangle_{\g}\big)$ becomes a 2-step Lie algebra endowed with a non-degenerate scalar product, where $U$ belongs to the centre.

The discussions above raise the following question. Let two finite dimensional vector spaces $U$ and $V$ are given, and let $J\colon U\to \End(V)$ be a linear map. When can one find a scalar product  $\langle.\,,.\rangle_V$ on $V$, such that $J_z$ satisfies~\eqref{eq:skew} for all $z\in U$? Thus, we are looking for a scalar product  $\langle.\,,.\rangle_V$ on $V$ such that $J\colon U\to \orth(V,\langle.\,,.\rangle_V)$. 

\begin{definition}\label{def:Winv}
Let $J\colon U\to \End(V)$ be a linear map. A scalar product $\langle.\,,.\rangle_V$ on $V$ satisfying  
$$
\langle J_z v,w\rangle_V=-\langle v,J_z w\rangle_V\quad \text{for all}\quad z\in U,
$$
we call {\it $W$-invariant}, where $W=J(U)$.
\end{definition}
If, moreover, a non-degenerate scalar product $\langle.\,,.\rangle_U$ on $U$ is given, then the decomposition $V\oplus U$ is orthogonal with respect to $\langle.\,,.\rangle_{\g}=\langle.\,,.\rangle_V+\langle.\,,.\rangle_U$, and we are able to define a Lie algebra structure on $V\oplus U$ by means of~\eqref{eq:defJ} as was described above.


\subsection{Uniqueness properties}


In this section we study the uniqueness of the choice of a scalar product invariant in the sense of Definition~\ref{def:Winv}. We start with a simple proposition.

\begin{proposition}\label{prop:uniq1} Let $(V,\langle.\,,.\rangle_V)$ and $(U,\langle.\,,.\rangle_U)$ be scalar product spaces, and let $J\colon U\to \orth(V)$. The multiplication of both scalar products $\langle.\,,.\rangle_V$ and $\langle.\,,.\rangle_U$ by a non-zero number $c$ does not change the brackets defined by the equality $\langle J_zv,w\rangle_V=\langle z,[v,w]\rangle_{U}$ for all $v,w\in V$, and for all $z\in U$. 
\end{proposition}

\begin{lemma}\label{lem:uniq1}
Let $V$ and $U$ be finite dimensional vector spaces, and let $\langle.\,,.\rangle_U$ be a non-degenerate scalar product on $U$. Let  $J\colon U\to \End(V)$ be a linear map with $W=J(U)$ and $\langle.\,,.\rangle_V^1$ and $\langle.\,,.\rangle_V^2$ be two $W$-invariant scalar products of equal index and such that the set of spacelike (timelike and correspondingly null) vectors coincide. Assume that $[.\,,.]^1$ and $[.\,,.]^2$ are Lie products defined by~\eqref{eq:defJ} with respect to the scalar products $\langle.\,,.\rangle_V^1$ and $\langle.\,,.\rangle_V^2$ on $\g=V\oplus_{\perp}U$. Then the Lie algebras $(\g,[.\,,.]^1)$ and $(\g,[.\,,.]^2)$ are isomorphic.
\end{lemma}

\begin{proof}
We define the linear map $S\colon V\to V$ by 
\begin{equation}\label{eq:S}
\langle v, w\rangle_V^2=\langle Sv,w\rangle_V^1.
\end{equation}
Then $S$ is injective because assuming that there is $v\in V$, $v\neq 0$, such that $Sv=0$, we arrive at $\langle v, w\rangle_V^2=0$ by~\eqref{eq:S} for any $w\in V$. So we conclude that $v=0$ by the non-degeneracy of the scalar product, which contradicts the assumption.

The map $S$ is symmetric with respect to both scalar products. Indeed
\begin{eqnarray}\label{eq:symmetric}
&\langle Sv,w\rangle_V^1=\langle v,w\rangle_V^2=\langle w,v\rangle_V^2=\langle Sw,v\rangle_V^1,
\\
&\langle Sv,w\rangle_V^2=\langle w,Sv\rangle_V^2=\langle Sw,Sv\rangle_V^1=\langle Sv,Sw\rangle_V^1=\langle v,Sw\rangle_V^2.\nonumber
\end{eqnarray}

Note that the operator $S$ has only real eigenvalues. Denote by $A_1$ the matrix of the first scalar product: $
\langle u,w\rangle^1_V=u^{\mathbf t} A_1w$.
The real matrices $A_1$ and $S$ satisfy the relations: $A_1=A_1^{\mathbf t}$ and $S^{\mathbf t}A_1=A_1S$.
We claim that the number $\bar u^{\mathbf t}A_1Su$ is real for any vector $u\in V$. Indeed
$$
\overline{\bar u^{\mathbf t}A_1Su}=u^{\mathbf t}A_1S\bar u=(u^{\mathbf t}A_1S\bar u)^{\mathbf t}=\bar u^{\mathbf t}S^{\mathbf t}A_1^{\mathbf t}u=\bar u^{\mathbf t}S^{\mathbf t}A_1u=\bar u^{\mathbf t}A_1Su.
$$
In the same way we show that $\bar u^{\mathbf t}A_1u\in\mathbb R$. It implies that if $Su=\lambda u$, then $\lambda$ is actually real, because of the following relation
$$
\mathbb R\ni \bar u^{\mathbf t}A_1Su=\langle Su,u\rangle_V^1=\lambda(\bar u^{\mathbf t}A_1u).
$$

Moreover $S$ has only positive eigenvalues, because if $Su=\lambda u$ and $\langle u,u\rangle_V^i\neq 0$, $i=1,2$, then
$$
\lambda\langle u,u\rangle_V^1=\langle Su,u\rangle_V^1=\langle u,u\rangle_V^2.
$$
Since $\langle u,u\rangle_V^1$ and $\langle u,u\rangle_V^2$ have always the same sign by the assumption of the lemma, we conclude that $\lambda>0$. If $Su=\lambda u$ and $\langle u,u\rangle_V^i=0$, $i=1,2$, then we change the arguments. Let $\{e_1,\ldots,e_m\}$ be an orthonormal basis with respect to $\langle .\,,.\rangle_V^1$, that always exists since the scalar product is non-degenerate. Choose a basis vector $e_k$ such that 
$\langle e_k,u\rangle_V^1\neq 0$. Such kind of vector $e_k$ exists,  otherwise $u$ would be the zero vector which contradicts the requirement that $u$ is an eigenvector. Then
$\langle ce_k-u,ce_k-u\rangle_{V}^1=0$ for $c=2\langle e_k,e_k\rangle_V^1\langle e_k,u\rangle_V^1$. Write $v=ce_k$. Then $\langle v-u,v-u\rangle_V^i= 0$ for $i=1,2$. This implies
$$
0=\langle v-u,v-u\rangle_V^i=\langle v,v\rangle_V^i-2\langle v,u\rangle_V^i,
$$
and we conclude that the non-vanishing value of $\langle v,u\rangle_V^i$ has the same sign in both vector spaces. Thus,
$$
\lambda\langle u,v\rangle_V^1=\langle Su,v\rangle_V^1=\langle u,v\rangle_V^2,
$$
and we conclude that $\lambda>0$.

The map $S$ commutes with $J_z$ for any $z\in U$ by
\begin{equation}\label{eq:commute}
\langle J_zSv,w\rangle_V^1=-\langle Sv,J_zw\rangle_V^1=-\langle v,J_zw\rangle_V^2=\langle J_zv,w\rangle_V^2=\langle SJ_zv,w\rangle_V^1.
\end{equation}

Let $V_1,\ldots, V_N$ be eigenspaces of the map $S$ corresponding to different eigenvalues, which we denote by $\lambda_1^2,\ldots,\lambda_N^2$. Then   $V_1,\ldots, V_N$ are mutually orthogonal with respect to both scalar products because the map $S$ is symmetric with respect to them. Let us write $V\ni v=\sum_{k=1}^Nv_k$  and $V\ni w=\sum_{k=1}^Nw_k$, where $v_k,w_k\in V_k$. We claim that 
$$
[v_k,v_j]^i=0\quad\text{for}\quad v_k\in V_k,\ \  v_j\in V_j,\quad k\neq j,\quad i=1,2.
$$
First, observe that the subspaces $V_k$, $k=1,\ldots,N$, are invariant under $J_z$ for any $z\in U$ because $SJ_z=J_zS$. We calculate
$$
\langle z,[v_k,v_j]^i\rangle_U=\langle J_zv_k,v_j\rangle_V^i=\langle v'_k,v_j\rangle_V^i=0,\quad i=1,2,\quad v'_k\in V_k
$$
for any $z\in U$. The scalar product $\langle .\,,.\rangle_U$ is non-degenerate and we conclude that $[v_k,v_j]^i=0$.

We are ready to define the Lie algebra isomorphism $(V\oplus U,[.\,,.]^2)\to(V\oplus U,[.\,,.]^1)$. Set $\varphi\colon V\oplus U\to V\oplus U$ by
\begin{equation}\label{eq:phi_isom}
\varphi=\begin{cases}
\lambda_k\Id_{V_k}, \quad k=1,\ldots, N,\quad &\text{on $V$},
\\
\Id_U, \quad &\text{on $U$},
\end{cases}
\end{equation}
and  check $\varphi([v,w]^2)=[\varphi(v),\varphi(w)]^1$. We obtain on one hand
\begin{eqnarray*}
\langle z,\varphi([v,w]^2)\rangle_U & = &\langle z,[v,w]^2\rangle_U=
\sum_{k=1}^N\langle z,[v_k,w_k]^2\rangle_U=
\sum_{k=1}^N\langle J_zv_k,w_k\rangle_V^2
\\
& = &\sum_{k=1}^N\lambda_k^2\langle J_zv_k,w_k\rangle_V^1,
\end{eqnarray*}
since $\langle.\,,.\rangle^2_{V_k}=\lambda_k^2 \langle.\,,.\rangle^1_{V_k}$.
On the other hand,
$$
\langle z,[\varphi(v),\varphi(w)]^1\rangle_U=\sum_{k=1}^N\lambda_k^2\langle z,[v_k,w_k]^1\rangle_U=\sum_{k=1}^N\lambda_k^2\langle J_zv_k,w_k\rangle_V^1,
$$
that finishes the proof.
\end{proof}

\begin{corollary}
Let $(V,\langle.\,,.\rangle_V)$ and $(U,\langle.\,,.\rangle_U)$ be two scalar product spaces, and let $J\colon U\to\orth(V)$. Then every scalar product $\langle.\,,.\rangle_U$ on $U$ defines a unique 2-step nilpotent Lie algebra structure given by~\eqref{eq:defJ} on the vector space $\g=V\oplus_{\perp}U$.
\end{corollary}

We generalise Lemma~\ref{lem:uniq1} in the following form.

\begin{proposition}
Let $V$ and $U$ be finite dimensional vector spaces, and let $\langle.\,,.\rangle_U$ be a non-degenerate scalar product on $U$. Let  $J\colon U\to \End(V)$ be a linear map with $W=J(U)$, and  let $\langle.\,,.\rangle_V^1$ and $\langle.\,,.\rangle_V^2$ be two $W$-invariant scalar products admitting a linear map $S\colon V\to V$ satisfying the following conditions.
\begin{itemize}
\item[1.] The eigenspaces $V_1,\ldots, V_N$ of the map $S$ are orthogonal with respect to both scalar products.
\item[2.] The eigenspaces are invariant under all maps $J_z\in W$.
\item[3.] The restriction $S\vert_{V_k}\colon V_k\to V_k$ is an isometry (or anti-isometry) for all $k=1,\ldots,N$.
\end{itemize}
Assume that $[.\,,.]^1$ and $[.\,,.]^2$ are Lie products defined by~\eqref{eq:defJ} with respect to the scalar products $\langle.\,,.\rangle_V^1$ and $\langle.\,,.\rangle_V^2$ on $\g=V\oplus_{\perp}U$. Then the Lie algebras $(\g,[.\,,.]^1)$ and $(\g,[.\,,.]^2)$ are isomorphic.
\end{proposition}

\begin{proof}
The first two properties imply that $
[v_k,v_j]^i=0$ for all $v_k\in V_k$, $v_j\in V_j$, $k\neq j$, $i=1,2$. The isometry property gives
$$
\lambda_k^2\langle u,v\rangle_{V_k}^2=\langle Su,Sv\rangle_{V_k}^2=\langle u,v\rangle_{V_k}^1\quad\text{for all}\quad u,v\in V_k.
$$
In the case of anti-isometry we obtain
$$
\lambda_k^2\langle u,v\rangle_{V_k}^2=\langle Su,Sv\rangle_{V_k}^2=-\langle u,v\rangle_{V_k}^1\quad\text{for all}\quad u,v\in V_k,
$$
for each eigenspace $V_k$. In the case of isometry we define the Lie algebra isomorphism $\varphi$ as in~\eqref{eq:phi_isom}, and in the case of anti-isometry we set
\begin{equation}\label{eq:phi_anisom}
\varphi=\begin{cases}
\lambda_k\Id_{V_k}, \quad k=1,\ldots, N,\quad &\text{on $V$},
\\
-\Id_U, \quad &\text{on $U$}.
\end{cases}
\end{equation}
The proof finishes as in Lemma~\ref{lem:uniq1}.
\end{proof}


\subsubsection{2-step nilpotent Lie algebras with trivial abelian factor}


The map $J\colon U\to\orth(V)$ is not necessarily injective. Nevertheless, if it is so, the corresponding 2-step nilpotent Lie algebra possesses nice properties. Let $\g$ be a 2-step nilpotent Lie algebra. Then in the commutative ideal $[\g,\g]$, there can be  a proper subspace   of the centre $\mathcal Z$ of the Lie algebra $\g$. The case $[\g,\g]=\mathcal Z$ corresponds to the injective map $J\colon \mathcal Z\to\orth(V)$. We recall some results for arbitrary 2-step nilpotent Lie algebras in this direction.

\begin{proposition}\cite{Eber03}\label{prop:abelian_factor}
Let $\g$ be a 2-step nilpotent Lie algebra with a centre $\mathcal Z$. Then there is an ideal $\g^*$  and an abelian ideal $\mathfrak a$ of $\g$ with $\mathfrak a\subseteq \mathcal Z$, such that 
\begin{itemize}
\item[1.]{$\g=\g^*\oplus\mathfrak a$ and $\mathcal Z=[\g,\g]\oplus\mathfrak a$;}
\item[2.]{$\g^*$ is a 2-step nilpotent Lie algebra such that $[\g,\g]=[\g^*,\g^*]=\mathcal Z^*$, where $\mathcal Z^*$ is the centre of $\g^*$;}
\item[3.]{The ideals $\g^*$ and $\mathfrak a$ are uniquely defined up to an isomorphism by item 1;}
\item[4.]{If $\g$ has a basis $\mathcal B$ with rational structure constants, then $\g^*$ has a basis $\mathcal B^*$ with integer structure constants.}
\end{itemize}
\end{proposition}
The factor $\mathfrak a$ in Proposition~\ref{prop:abelian_factor} is called an {\it abelian factor}. The proposition has the following useful corollary.
\begin{corollary}
Let $\g$ be a 2-step nilpotent Lie algebra with a centre $\mathcal Z$. Then $\g$ has a trivial abelian factor if and only if $[\g,\g]=\mathcal Z$.
\end{corollary}

\begin{lemma}\label{lem:eqiv2}
Let $(\g,[.\,,.],\langle.\,,.\rangle_{\g})$ be a 2-step nilpotent Lie algebra with a centre $\mathcal Z$, and let a scalar product $\langle.\,,.\rangle_{\g}$ be such that its restrictions to $\mathcal Z$ and $[\g,\g]$ are non-degenerate. Let $V=\mathcal Z^{\perp}$, and let $J\colon \mathcal Z\to \orth(V)$ be the linear map defined by~\eqref{eq:defJ}. Then the following statements are equivalent: 
\begin{itemize}
\item[1.]{The commutative ideal $[\g,\g]$ has co-dimension $d\geq 0$ in $\mathcal Z$;}
\item[2.]{The kernel of $J$ has dimension $d$.}
\end{itemize}
\end{lemma}
\begin{proof}
Let us write $\mathcal Z=[\g,\g]\oplus[\g,\g]^{\perp}$. Then $\langle J_zv,w\rangle_V=\langle z,[v,w]\rangle_\mathcal Z$, and the non-degeneracy of the scalar products imply that $J_zv=0$, if and only if, $z\in[\g,\g]^{\perp}$, that proves the equivalence of items 1 and 2.
\end{proof}
An immediate corollary of Lemma~\ref{lem:eqiv2} follows.
\begin{corollary}\label{cor:trivial_factor}
Let $\g$ be a 2-step nilpotent Lie algebra with a centre $\mathcal Z$. Then the following statements are equivalent. 
\begin{itemize}
\item[1.] {The Lie algebra $\g$ has a trivial abelian factor;}
\item[2.] {Let $\g$ admit a non-degenerate scalar product,  such that its restriction to $\mathcal Z$ is non-degenerate. Let $V=\mathcal Z^{\perp}$, and let the linear map $J\colon \mathcal Z\to\orth(V)$ be defined by~\eqref{eq:defJ}. Then the map $J$ is injective.}
\end{itemize}
\end{corollary}


\subsection{Examples}


Now we give several examples of skew-symmetric maps and Lie algebras related to them. 

\begin{example}
Consider $\mathbb R^{p,q}$, $p+q=m$ with the metric $\langle x,y\rangle_{p,q}=x^{\mathbf t}\eta_{p,q}y$. Let $W$ be a non-zero subspace of $\so(p,q)$. The inclusion map $\iota\colon W\to\so(p,q)$ defines a skew-symmetric map in the following sense: if $z\in W$ and $\iota_z=\iota(z)=Z\in\so(p,q)$, then
$$
\la\iota_z(x),y\ra_{p,q}=\la Zx,y\ra_{p,q}=-\la x,Zy\ra_{p,q}=-\la x,\iota_z(y)\ra_{p,q}.
$$
If the restriction of the metric, defined by the trace on $\so(p,q)$, to the vector subspace $W\subset\so(p,q)$ is non-degenerate, then we can define a Lie algebra structure on $\mathbb R^{p,q}\oplus W$. If $W=\so(p,q)$, then the constructed Lie algebra on $\mathbb R^{p,q}\oplus \so(p,q)$ will be the free 2-step nilpotent Lie algebra that we denote by $F_2(p,q)$. Thus, $F_2(p,q)=\mathbb R^{p,q}\oplus \so(p,q)$, where the commutator on $\mathbb R^{p,q}$ is defined by 
\begin{equation}\label{eq:brack_R}
[w, v]_{F_2(p,q)}= -\frac{1}{2}(wv^{\mathbf t}-vw^{\mathbf t})\eta_{p,q}.
\end{equation}
For the standard basis $\{e_i\}$ of $\mathbb R^{p,q}$ we get $[e_i, e_j]_{F_2(p,q)}= -\frac{1}{2}(E_{ij} - E_{ji})\eta_{p,q}$, where $E_{ij}$ denote the $(m\times m)$ matrix with zero entries except of $1$ at the position $ij$. Since $F_2(p,q)$ is the 2-step nilpotent Lie algebra, we obtain that $\so(p,q)$ forms a centre. Particularly, if $q=0$, then we get the free Lie algebra $F_2(m)$ studied in~\cite{Eber03}.
\end{example}

The next example is closely related to Section~\ref{subseq:pseudoH}.
\begin{example}
Let $\g$ be a pseudo $H$-type Lie algebra. Then the linear map defined by~\eqref{eq:defJ} is skew-symmetric and defines a representation of the Clifford algebra. Conversely, given a representation $J\colon\Cl(U,\langle .\,,.\rangle_U)\to V$, which is also skew-symmetric with respect to a scalar product on $V$, we can construct a 2-step nilpotent Lie algebra that will be a general $H$-type Lie algebra. All details are described in Section~\ref{subseq:pseudoH}.
\end{example}


\subsection{Standard pseudo-metric 2-step nilpotent Lie algebras}


We present the construction of 2-step nilpotent Lie algebras with some standard choice of the metrics.

Let $(V,\langle .\,,.\rangle_V)$ be an $m$-dimensional scalar product space, and let $\orth(V)$ space of skew-symmetric maps with respect to $\langle .\,,.\rangle_V$. Equip the space $\orth(V)$ with the metric $\langle z,z'\rangle_{\orth(V)}=-\tr(zz')$, $z\in\orth(V)$. Observe that if the scalar product $\langle .\,,.\rangle_V$ has index $(p,q)$, $p+q=m$, then the scalar product $\langle .\,,.\rangle_{\orth(V)}$ has index $\big(\frac{p(p-1)+q(q-1)}{2},pq\big)$. Since the Lie algebra $\orth(V)$ is simple, then any symmetric bilinear form is a multiple of the Killing form.

Let $W$ be an $n$-dimensional subspace of $\orth(V)$, such that the restriction of $\langle .\,,.\rangle_{\orth(V)}$ to $W$ is non-degenerate. Let $\mathcal G=V\oplus W$ and $\langle.\,,.\rangle_{\mathcal G}=\langle .\,,.\rangle_V+\langle .\,,.\rangle_{\orth(V)}$. The direct sum $\mathcal G=V\oplus W$ is orthogonal with respect to $\langle.\,,.\rangle_{\mathcal G}$. Let $[.\,,.]_{\mathcal G}$ be the Lie product on $\mathcal G$ defined as follows. If $v,w\in V$, then $[v,w]_{\mathcal G}$ is a unique element of $W$, such that 
\begin{equation}\label{eq:standard}
\langle [v,w]_{\mathcal G},z\rangle_{\orth(V)}=\langle z(v),w\rangle_V
\end{equation}
for every $z\in W$. 
\begin{definition}\label{def:pseudo_standard}
We call the Lie algebra $\mathcal G$ constructed above {\it the standard pseudo-metric} 2-step nilpotent Lie algebra and write $\mathcal G=(V\oplus_{\bot} W, [.\,,.]_{\mathcal G},\langle.\,,.\rangle_{\mathcal G})$.
\end{definition} 
If $V=\mathbb R^{p,q}$ and $\langle .\,,.\rangle_{V}=\langle .\,,.\rangle_{p,q}$ is the scalar product defined by the matrix $\eta_{p,q}=\diag(I_p,-I_q)$, then we write $\so(p,q)$ for skew-symmetric maps, and the standard pseudo-metric 2-step nilpotent Lie algebra is $\mathcal G=(\mathbb R^{p,q}\oplus_{\bot}W,[.\,,.],\langle.\,,.\rangle_{\mathcal G})$ with $\langle.\,,.\rangle_{\mathcal G}=\langle .\,,.\rangle_{p,q}+\langle .\,,.\rangle_{\so(p,q)}$.

We also say that the standard pseudo-metric 2-step nilpotent Lie algebra is {\it involutive}, if $W$ is a subalgebra in $\orth(V,\langle .\,,.\rangle_V)$.
It is easy to see that $[\mathcal G,\mathcal G]=W$ and $W$ is the centre of $\mathcal G$, if and only if, for any $v\neq 0$, $v\in V$ there is $z\in W$ such that $z(v)\neq 0$.


\begin{example} {\it Free standard pseudo-metric Lie algebra.} Let us equip the 2-step free Lie algebra $F_2(p,q)=\mathbb{R}^{p,q}\oplus\so(p,q)$ with the scalar product $\la.\,,.\ra=\la .\,,.\ra_{\so(p,q)}+\la  .\,,. \ra_{p,q}$. Then 
$$\la [w, v]_{F_2(p,q)}, Z \ra_{\so(p,q)}=\la Zw,v \ra_{p,q}$$ for all $w,v \in \mathbb{R}^{p,q}$ and all $Z \in \so(p,q)$,
where the Lie brackets are introduced in~\eqref{eq:brack_R}.
First we calculate $\la Zw,v \ra_{p,q}$ and obtain
\begin{eqnarray*}
\la Zw, v \ra_{p,q}= w^{\mathbf t}Z^{\mathbf t}\eta_{p,q} v=-w^{\mathbf t}\eta_{p,q} Z  v = - \tr(w^{\mathbf t}  \eta_{p,q} Z v)=- \tr(vw^{\mathbf t}  \eta_{p,q} Z),
\end{eqnarray*}
where $\tr(w^{\mathbf t}Z \eta_{p,q} v)=w^{\mathbf t}Z \eta_{p,q} v$ as $w^{\mathbf t}Z \eta_{p,q} v \in \mathbb{R}$, and where we used $Z^{\mathbf t} \eta_{p,q} = -\eta_{p,q} Z$ for all $Z \in \so(p,q)$. Moreover, since $Z\in\so(p,q)$ we also get 
$$
\la Zw, v \ra_{p,q}=-\la w, Zv \ra_{p,q}=\tr(wv^{\mathbf t}\eta_{p,q} Z).
$$
With these relations we calculate $\la [ w, v], Z \ra_{\so(p,q)}$ and obtain the desired equality
\begin{eqnarray*}
\la [ w, v], Z \ra_{\so(p,q)}&=&  -\tr\left(-\frac{1}{2}(wv^{\mathbf t}-vw^{\mathbf t})\eta_{p,q} Z \right)  =\frac{1}{2}\left(\tr(wv^{\mathbf t}\eta_{p,q} Z)-\tr(vw^{\mathbf t} \eta_{p,q} Z)\right) 
\\
&=&  \la Zw, v \ra_{p,q}.
\end{eqnarray*}
\end{example}


\begin{example}\label{ex:Cl}
{\it Representation of Clifford algebras.}
Let $(\mathbb R^{r,s},\langle.\,,.\rangle_{r,s})$, and let $\Cl_{r,s}$ denote the Clifford algebra generated by $\mathbb R^{r,s}$. Let $J\colon \Cl_{r,s}\to\End(V)$ be a Clifford algebra representation on the finite-dimensional vector space $V$. We identify $V$ (or $V\oplus V$ if it is necessary) with $\mathbb R^{p,p}$, $2p=m$, equipped with the scalar product $\langle.\,,.\rangle_{p,p}$, such that $W=J(\mathbb R^{r,s})\subseteq \so(p,p)$ if $s>0$. If $s=0$, then we identify $V$ with the Euclidean space $\mathbb R^m$, and in this case $W=J(\mathbb R^{r,0})\subseteq \so(m)$. 
As it was observed in Remark~\ref{rem:neutral} the scalar product on $V$ should be neutral in the case $s>0$, that determines the choice of the scalar product $\langle.\,,.\rangle_{p,p}$ and the inclusion of $W=J(\mathbb R^{r,s})$ into the space $\so(p,p)$. \end{example}


\subsection{Reduction of a 2-step nilpotent Lie algebra to the standard pseudo-metric 2-step nilpotent Lie algebra}

We start from the following observation relating elements in $\so(m)$ and $\so(p,q)$ with $p+q=m$. 
Let $\eta_{p,q}=\diag(I_p,-I_q)$, $p+q=m$, and let $\nu_i=\nu_i(p,q)$ be defined by~\eqref{eq:nu}.
Then, for any $(m\times m)$ matrix $A=\{a_{ij}\}_{i,j=1}^{m}$, we have 
$$
(A\eta_{p,q})_{ij}=a_{ij}\nu_{j},\quad (\eta_{p,q} A)_{ij}=a_{ij}\nu_{i}.
$$

Let $C\in \so(m)$, and define $D=C\eta_{p,q}$ (or equivalently, $D_{ij}=\nu_jC_{ij}$). We claim that $D\in \so(p,q)$. Indeed, 
$$\eta_{p,q} D^{\mathbf t}\eta_{p,q}=\eta_{p,q}(C\eta_{p,q})^{\mathbf t}\eta_{p,q}=\eta_{p,q}\eta^{\mathbf t}_{p,q}C^{\mathbf t}\eta_{p,q}=-C\eta_{p,q}=-D.
$$
Analogously, we can show that $\widetilde D=\eta_{p,q}C\in\so(p,q)$ if $C\in\so(m)$, $m=p+q$.
Let us prove the following technical lemma.

\begin{lemma}\label{lem:lin_indep}
Let $\g$ be a 2-step nilpotent Lie algebra, such that $\dim([\g,\g])=n$, and let the complement $V$ to $[\g,\g]$ be of dimension $m$. Denote by $z_1,\ldots,z_n$ a basis of $[\g,\g]$, and by $v_1,\ldots,v_m$ a basis of $V$. Let $[v_i,v_j]=\sum_{k=1}^{n}C^k_{ij}z_k$ for $1\leq i,j\leq m$. Then the matrices $D^k=C^k\eta_{p,q}$ are linearly independent in any $\so(p,q)$, $p+q=m$.
\end{lemma}

\begin{proof}
It was proved in~\cite{Eber03} that $C^1,\ldots, C^n$ are linearly independent in $\so(m)$. Thus for any real numbers $\alpha_1,\ldots,\alpha_n$ we have
$$
\sum_{k=1}^{n}\alpha_kC^k=0\quad\Longleftrightarrow\quad \alpha_k=0,\ k=1,2,\ldots,n.
$$
Then 
$$0=\Big(\sum_{k=1}^{n}\alpha_kC^k\Big)\eta_{p,q}=\sum_{k=1}^{n}\alpha_kC^k\eta_{p,q}=\sum_{k=1}^{n}\alpha_kD^k\quad\Longleftrightarrow\quad \alpha_k=0,\ k=1,2,\ldots,n.
$$
\end{proof}

Any 2-step nilpotent Lie algebra $\mathfrak g$ defines a subspace $\mathcal C\subset\so(m)$, where $\mathcal C=\spn_{\mathbb R}\{C^1,\ldots,C^k\}$ and moreover, this subspace is non-degenerate  in $\so(m)$. This fact allows to construct the isomorphism between $\mathfrak g$ and the corresponding standard metric Lie algebra $\mathcal G=\mathbb R^m\oplus \mathcal C$ with positive-definite scalar product, see~\cite{Eber03}. The space $\mathcal C$ also generates spaces $\mathcal D=\spn_{\mathbb R}\{D^1,\ldots,D^k\}$, $D^j=C^j\eta_{p,q}$, in each $\so(p,q)$. Moreover, if $\mathcal D\subset \so(p,q)$ is non-degenerate with respect to the restriction  of the indefinite trace metric to $\mathcal D$ in $\so(p,q)$, then there is an isomorphism between $\mathfrak g$ and the standard pseudo-metric Lie algebra $\mathcal G=\mathbb R^{p,q}\oplus \mathcal D$ as it is shown in the following theorem.

\begin{theorem}\label{th:isom}
Let $\g$ be a 2-step nilpotent Lie algebra such that $\dim([\g,\g])=n$, and the complement $V$ to $[\g,\g]$ is of dimension $m$. Then
there exists an $n$-dimensional subspace $\mathcal D$ of $\so(p,q)$, $p+q=m$, $n\leq\frac{m(m-1)}{2}$ such that if $\mathcal D$ is an $n$-dimensional non-degenerate subspace of $\so(p,q)$, then $\g$ is isomorphic as a Lie algebra to the standard pseudo-metric 2-step nilpotent Lie algebra $\mathcal G=\mathbb R^{p,q}\oplus_{\perp}\mathcal D$.
\end{theorem}

\begin{proof}
Let $\g=V\oplus[\g,\g]$, $v_1,\ldots,v_m$ be a basis of $V$, and let $z_1,\ldots,z_n$ be a basis of $[\g,\g]$.
Let $e_1,\ldots,e_{p+q}$ be the standard orthonormal basis in $\mathbb R^{p,q}$ with  the scalar product $\la .\,,.\ra_{p,q}$. 

Let $[v_i,v_j]_{\g}=\sum_{k=1}^{n}C^k_{ij}z_k$ for $1\leq i,j\leq m$ and  $D^k=\eta_{p,q}C^k$. 
Choose a pair $p,q\in \mathbb N$, $p+q=m$, such that the space $\mathcal D=\spn\{D^1,\ldots,D^n\}\subset \so(p,q)$ is non-degenerate with respect to the metric $\langle.\,,.\rangle_{\so(p,q)}$. Let ${\rho_1,\ldots,\rho_n}$ be a basis of $\mathcal D$, such that $\langle \rho_k,D^l\rangle_{\so(p,q)}=\delta_{kl}$ for $1\leq k,l\leq n$.

Define a linear isomorphism $T\colon \g\to\mathcal G$ by
$$
T(v_i)=e_i,\ \ i=1,\ldots, m,\qquad T(z_{k})=-\rho_k,\ \ k=1,\ldots,n.
$$
We claim that $T$ is a Lie algebra isomorphism and it suffices to show that
$$
T([v_i,v_j]_{\g})=[T(v_i),T(v_j)]_{\mathcal G}.
$$
Note that 
\begin{eqnarray*}
\langle[T(v_i),T(v_j)]_{\mathcal G},D^k\rangle_{\so(p,q)} & = & \langle[e_i,e_j]_{\mathcal G},D^k\rangle_{\so(p,q)}=\langle D^k(e_i),e_j\rangle_{p,q}
\\
& =& (e_i)^{\mathbf t}( D^k)^{\mathbf t}\eta_{p,q} e_j=((D^{k})^{\mathbf t}\eta_{p,q})_{ij}
\\
& = & ((C^k)^{\mathbf t})_{ij}=-C^k_{ij}=C^k_{ji}.
\end{eqnarray*} 
On the other hand,
\begin{eqnarray*}
\langle T([v_i,v_j]_{\g}),D^k\rangle_{\so(p,q)} & = & \langle \sum_{r=1}^{n}C^r_{ij}T(z_r),D^k\rangle_{\so(p,q)}=-\sum_{r=1}^{n}C^r_{ij}\langle \rho_r,D^k\rangle_{\so(p,q)}
\\
& = & -\sum_{r=1}^{n}C^r_{ij}\delta_{rk}
 = -C^k_{ij}=C^k_{ji},
\end{eqnarray*}
which finishes the proof. 
\end{proof}


\subsection{Examples of standard pseudo-metric algebras}


Let us consider three pseudo $H$-type Lie algebras $\n_{2,0}$, $\n_{1,1}$, and $\n_{0,2}$ and show that they can be realised as standard pseudo-metric algebras for some choice of $\mathfrak{so}(p,q)$.

\medskip
{\sc The pseudo $H$-type Lie algebra $\n_{2,0}$.} The algebra $\n_{2,0}$ is constructed from the Clifford algebra $\Cl_{2,0}$. Thus the centre of $\n_{2,0}$ is isomorphic to $\mathbb R^2$ and the complement to the centre is isomorphic to $\mathbb R^4$ with the standard Euclidean metrics. Let $(z_1,z_2)$ be the standard basis of $\mathbb R^2$, and let  $J_{z_1}$, $J_{z_2}\in\so(4)$ be such that 
$$
J_{z_1}^2=J_{z_2}^2=-\Id_{\mathbb R^4},\quad J_{z_1}J_{z_2}=-J_{z_2}J_{z_1}.
$$
We chose the following orthonormal basis in $\mathbb R^4$ constructed by
$$
v_1=e_1,\quad v_2=J_{z_2}J_{z_1}v_1,\quad v_3=J_{z_1}v_1,\quad v_4=J_{z_2}v_1.
$$
In the basis $\{v_1,v_2,v_3,v_4\}$ the matrices of the maps $J_{z_1},J_{z_2}$ take the following form:
$$
J_{z_1}=\begin{pmatrix}
0&0&-1&0
\\
0&0&0&-1
\\
1&0&0&0
\\
0&1&0&0
\end{pmatrix},
\qquad
J_{z_2}=\begin{pmatrix}
0&0&0&-1
\\
0&0&1&0
\\
0&-1&0&0
\\
1&0&0&0
\end{pmatrix}.
$$
Maps $J_{z_i}$ permute the basis of $\mathbb R^4$ by the rule:
$$
\begin{array}{lllllllllll}
& J_{z_1}v_1  =  v_3,\quad & J_{z_1}v_2  =  v_4,\quad  & J_{z_1}v_3  =-v_1,\quad & J_{z_1}v_4  =  -v_2,
\\
& J_{z_2}v_1  =  v_4,\quad & J_{z_2}v_2  =  -v_3,\quad & J_{z_2}v_3  =v_2,\quad & J_{z_2}v_4  =  -v_1.
\end{array}
$$
According to the equality $\langle[v_{\alpha},v_{\beta}],z_{i}\rangle_{2,0}=\langle J_{z_i}v_{\alpha},v_{\beta}\rangle_{4,0}$, we calculate the structural constants in
$[v_\alpha,v_{\beta}]=C^{1}_{\alpha\beta}z_1+C^{2}_{\alpha\beta}z_2$ as 
\begin{equation}\label{C20}
C^1=\begin{pmatrix}
0&0&1&0
\\
0&0&0&1
\\
-1&0&0&0
\\
0&-1&0&0
\end{pmatrix},
\qquad
C^{2}=\begin{pmatrix}
0&0&0&1
\\
0&0&-1&0
\\
0&1&0&0
\\
-1&0&0&0
\end{pmatrix}.
\end{equation}
We see that $C^{i}=-J_{z_i}$. This also follows from the choice of the orthonormal basis by
$$
C^i_{\alpha\beta}=\langle[v_{\alpha},v_{\beta}],z_{i}\rangle_{2,0}=
\langle J_{z_i}v_{\alpha},v_{\beta}\rangle_{4,0}=
v_{\alpha}^{\mathbf t}J_{z_i}^{\mathbf t}v_{\beta}=
(J_{z_i}^{\mathbf t})_{\alpha\beta}=-(J_{z_i})_{\alpha\beta}.
$$

\medskip

{\sc The pseudo $H$-type Lie algebra $\n_{1,1}$.} The Lie algebra is constructed from the Clifford algebra $\Cl_{1,1}$, and therefore, the centre of $\n_{1,1}$ is isomorphic to $\mathbb R^{1,1}$ and the complement to the centre is isomorphic to $\mathbb R^{2,2}$ with the corresponding metric. We start from the basis $\{z_1,z_2\}$ of the centre and two skew-symmetric maps $J_{z_1}$, $J_{z_2}\in\so(2,2)$, satisfying 
$$
J_{z_1}^2=-\Id_{\mathbb R^{2,2}},\quad J_{z_2}^2=\Id_{\mathbb R^{2,2}},\quad J_{z_1}J_{z_2}=-J_{z_2}J_{z_1}.
$$
Choose the orthonormal basis in $\mathbb R^{2,2}$
$$
v_1=e_1,\quad v_2=J_{z_1}v_1,\quad v_3=J_{z_2}v_1,\quad v_4=J_{z_2}J_{z_1}v_1.
$$
The maps $J_{z_1},J_{z_2}$ take the form in the basis $\{v_1,v_2,v_3,v_4\}$:
$$
J_{z_1}=\begin{pmatrix}
0&-1&0&0
\\
1&0&0&0
\\
0&0&0&1
\\
0&0&-1&0
\end{pmatrix},
\qquad
J_{z_2}=\begin{pmatrix}
0&0&1&0
\\
0&0&0&1
\\
1&0&0&0
\\
0&1&0&0
\end{pmatrix},
$$
and they permute the basis vectors as
$$
\begin{array}{lllllllllll}
& J_{z_1}v_1  =  v_2,\quad & J_{z_1}v_2  =  -v_1,\quad  & J_{z_1}v_3  =-v_4,\quad & J_{z_1}v_4  =  v_3,
\\
& J_{z_2}v_1  =  v_3,\quad & J_{z_2}v_2  =  v_4,\quad & J_{z_2}v_3  =v_1,\quad & J_{z_2}v_4  =  v_2.
\end{array}
$$
We calculate the structural constants in
$[v_\alpha,v_{\beta}]=C^{1}_{\alpha\beta}z_1+C^{2}_{\alpha\beta}z_2$ according to the rule $\langle[v_{\alpha},v_{\beta}],z_{i}\rangle_{1,1}=\langle J_{z_i}v_{\alpha},v_{\beta}\rangle_{2,2}$ as 
\begin{equation}\label{C11}
C^1=\begin{pmatrix}
0&1&0&0
\\
-1&0&0&0
\\
0&0&0&1
\\
0&0&-1&0
\end{pmatrix},
\qquad
C^{2}=\begin{pmatrix}
0&0&1&0
\\
0&0&0&1
\\
-1&0&0&0
\\
0&-1&0&0
\end{pmatrix}.
\end{equation}
We see that $C^{1}=-\eta_{2,2} J_{z_1}$ and $C^2=\eta_{2,2} J_{z_2}$. They are also defined by the choice of the orthonormal basis as 
$$
\nu_i(1,1)C^i_{\alpha\beta}=\langle[v_{\alpha},v_{\beta}],z_{i}\rangle_{1,1}=\langle J_{z_i}v_{\alpha},v_{\beta}\rangle_{2,2}=-v_{\alpha}^{\mathbf t}\eta_{2,2} J_{z_i}v_{\beta}=-(\eta_{2,2} J_{z_i})_{\alpha\beta}.
$$
Recall the notation~\eqref{eq:nu} for $\nu_i(p,q)$.

\medskip

{\sc The pseudo $H$-type Lie algebra $\n_{0,2}$.} This Lie algebra is related to the representation $J\colon \Cl_{0,2}\to \End(\mathbb R^{2,2})$. We start from an orthonormal basis $\{z_1,z_2\}$ of the centre isomorphic to $\mathbb R^{0,2}$ and skew-symmetric maps $J_{z_1}$, $J_{z_2}\in\so(2,2)$: 
$$
J_{z_1}^2=J_{z_2}^2=\Id_{\mathbb R^{2,2}},\quad J_{z_1}J_{z_2}=-J_{z_2}J_{z_1}.
$$
Choose the orthonormal basis for $\mathbb R^{2,2}$ as
$$
v_1=e_1,\quad v_2=J_{z_1}J_{z_2}v_1,\quad v_3=J_{z_1}v_1,\quad v_4=J_{z_2}v_1.
$$
The matrices of the maps $J_{z_1},J_{z_2}$ written in the basis $\{v_1,v_2,v_3,v_4\}$ are:
$$
J_{z_1}=\begin{pmatrix}
0&0&1&0
\\
0&0&0&1
\\
1&0&0&0
\\
0&1&0&0
\end{pmatrix},
\qquad
J_{z_2}=\begin{pmatrix}
0&0&0&1
\\
0&0&-1&0
\\
0&-1&0&0
\\
1&0&0&0
\end{pmatrix}
$$
and the permutation rule is
$$
\begin{array}{lllllllllll}
& J_{z_1}v_1  =  v_3,\quad & J_{z_1}v_2  =  v_4,\quad  & J_{z_1}v_3  =v_1,\quad & J_{z_1}v_4  =  v_2
\\
& J_{z_2}v_1  =  v_4,\quad & J_{z_2}v_2  =  -v_3,\quad & J_{z_2}v_3  =-v_2,\quad & J_{z_2}v_4  =  v_1.
\end{array}
$$
According to the relation $\langle[v_{\alpha},v_{\beta}],z_{i}\rangle_{0,2}=\langle J_{z_i}v_{\alpha},v_{\beta}\rangle_{2,2}$ we calculate the structure constants in
$[v_\alpha,v_{\beta}]=C^{1}_{\alpha\beta}z_1+C^{2}_{\alpha,\beta}z_2$ as 
$$
C^1=\begin{pmatrix}
0&0&1&0
\\
0&0&0&1
\\
-1&0&0&0
\\
0&-1&0&0
\end{pmatrix},
\qquad
C^{2}=\begin{pmatrix}
0&0&0&1
\\
0&0&-1&0
\\
0&1&0&0
\\
-1&0&0&0
\end{pmatrix}.
$$
We see that $C^{i}=\eta_{2,2} J_{z_i}$, or it can be found from
$$
\nu_i(0,2)C^i_{\alpha\beta}=\langle[v_{\alpha},v_{\beta}],z_{i}\rangle_{0,2}=-\langle v_{\alpha},J_{z_i}v_{\beta}\rangle_{2,2}=-v_{\alpha}^{\mathbf t}\eta_{2,2} J_{z_i}v_{\beta}=-(\eta_{2,2} J_{z_i})_{\alpha\beta}.
$$
Since $\nu_i(0,2)=-1$ for $i=1,2$, we obtain $C^{i}=\eta_{2,2} J_{z_i}$.

It follows from above that the pseudo $H$-type Lie algebras $\n_{2,0}$ and $\n_{0,2}$ coincide as Lie algebras. It can be interpreted as the following illustration to Theorem~\ref{th:isom}. The Lie algebra $\n_{2,0}$ is isomorphic to the standard metric Lie algebra $\mathcal G=\mathbb R^4\oplus\mathcal C$ with $\mathcal C=\spn\{C^1,C^2\}\subset\so(4)$ and with $C^1,C^2$ given by~\eqref{C20}. This standard metric Lie algebra is the $H$-type algebra because the skew-symmetric maps $J_{z_1}=-C^1$ and $J_{z_2}=-C^2$ satisfy the additional conditions $J^2_{z_i}=\Id_{\mathbb R^4}$ and $J_{z_1}J_{z_2}=-J_{z_2}J_{z_1}$. Let us check if the Lie algebra $\n_{2,0}$ can be isomorphic to the standard Lie algebra generated by other choices of $\mathbb R^{p,q}$, $p+q=4$.

\medskip

{\sc Cases $\mathbb R^{3,1}$ and $\mathbb R^{1,3}$.} We calculate the matrices
$$
D^1=C^1\eta_{3,1}=\begin{pmatrix}
0&0&1&0
\\
0&0&0&-1
\\
-1&0&0&0
\\
0&-1&0&0
\end{pmatrix},
\qquad
D^2=C_{2}\eta_{3,1}=\begin{pmatrix}
0&0&0&-1
\\
0&0&-1&0
\\
0&1&0&0
\\
-1&0&0&0
\end{pmatrix}.
$$


Since $\langle D^i,D^j\rangle_{\so(3,1)}=\tr(\eta_{3,1} (D^i)^{\mathbf t}\eta_{3,1} D^j)=0$, the subspace $\mathcal D=\spn\{D^1,D^2\}\subset \so(3,1)$ is degenerate, and actually, the scalar product $\langle .\,,.\rangle_{\so(3,1)}$ vanishes on $\mathcal D$, and therefore, the Lie algebra $n_{2,0}$ can not be realised as a standard pseudo-metric Lie algebra in $\mathbb R^{3,1}\oplus \mathcal D$. Recall, that the index of the space $\so(3,1)$ with respect to the trace metric $\langle .\,,. \rangle_{\so(3,1)}$ is $(3,3)$. The same calculations are valid for the case of $\mathbb R^{1,3}$, and we conclude that the Lie algebra $\n_{2,0}$ can  be realised as the standard pseudo-metric Lie algebra neither as $\mathbb R^{3,1}\oplus\mathcal D$, $\mathcal D\subset\so(3,1)$ nor as $\mathbb R^{1,3}\oplus\mathcal{\tilde D}$, $\mathcal{\tilde D}\subset\so(1,3)$.

\medskip

{\sc Case $\mathbb R^{2,2}$.} In this case we use $\eta_{2,2}$ and deduce the following matrices
$$
D^1=C^1\eta_{2,2}=\begin{pmatrix}
0&0&-1&0
\\
0&0&0&-1
\\
-1&0&0&0
\\
0&-1&0&0
\end{pmatrix},
\qquad
D^2=C_{2}\eta_{2,2}=\begin{pmatrix}
0&0&0&-1
\\
0&0&1&0
\\
0&1&0&0
\\
-1&0&0&0
\end{pmatrix},
$$
 from the matrices in~\eqref{C20}.
In this case $\langle D^1,D^1\rangle_{\so(2,2)}=-4$, $\langle D^2,D^2\rangle_{\so(2,2)}=-4$, and $\langle D^1,D^2\rangle_{\so(2,2)}=0$. The subspace $\mathcal D=\spn\{D^1,D^2\}\subset\so(2,2)$ is non-degenerate and has index $(r,s)=(0,2)$. Therefore, the Lie algebra $\mathfrak n_{2,0}$ can be realised as a standard metric Lie algebra $\mathbb R^{2,2}\oplus \mathcal D$, $\mathcal D\subset\so(2,2)$, and it gives the pseudo $H$-type Lie algebra $\mathfrak n_{0,2}$ constructed above. The last statement is valid due to the relations $J^2_{z_i}=\Id_{\mathbb R^{2,2}}$ and $J_{z_1}J_{z_2}=-J_{z_2}J_{z_1}$.

Now we turn to the Lie algebra $\mathfrak n_{1,1}$. The analogous calculations show that this Lie algebra can be realised in $\mathbb R^{4}\oplus\mathcal C$ with $\mathcal C=\spn\{C^1, C^2\}\subset\so(4)$, where $C^1, C^2$ are from~\eqref{C11}, but this is not an $H$-type Lie algebra (with a positive-definite scalar product), see Remark~\ref{rem:neutral}. The Lie algebra can be realised neither in $\so(3,1)$ nor in $\so(1,3)$, due to the degeneracy of the corresponding spaces $\mathcal D$. In the case $\so(2,2)$, the matrices 
$$
D^1=C^1\eta_{2,2}=\begin{pmatrix}
0&1&0&0
\\
-1&0&0&0
\\
0&0&0&-1
\\
0&0&1&0
\end{pmatrix},
\qquad
D^2=C_{2}\eta_{2,2}=\begin{pmatrix}
0&0&-1&0
\\
0&0&0&-1
\\
-1&0&0&0
\\
0&-1&0&0
\end{pmatrix}
$$
satisfying $\langle D^1,D^1\rangle_{\so(2,2)}=4$, $\langle D^2,D^2\rangle_{\so(2,2)}=-4$, and $\langle D^1,D^2\rangle_{\so(2,2)}=0$ span a two-dimensional non-degenerate space of index $(r,s)=(1,1)$ in $\so(2,2)$. Recall, that the index of the space $\so(2,2)$ is $(2,4)$. The standard metric Lie algebra $\mathbb R^{2,2}\oplus \mathcal D$, $\mathcal D\subset\so(2,2)$, in this case is the pseudo $H$-type Lie algebra $\n_{1,1}$.

Finally, we observe that $D^k=C^k\eta_{2,2}=-\eta_{2,2}\nu^k(1,1)J_{z_k}\eta_{2,2}$. Thus, we also have that $(D^1)^{\mathbf t}=-D^1$, $(D^2)^{\mathbf t}=D^2$ and $\mathcal D$ is closed under transposition.


\section{Isomorphism properties}\label{sec:isom}



\subsubsection{Isomorphism properties defined by skew-symmetric maps.}


Given a scalar product space $(V,\langle .\,,.\rangle_V)$ the space $\orth(V)$ of skew-symmetric maps has a scalar product defined by the trace. Let $J\colon U\to\orth(V,\langle .\,,.\rangle_V)$ be an injective
map, and let the space $J(U)$ be a non-degenerate subspace in $\orth(V,\langle .\,,.\rangle_V)$. Then, we
can pull back the trace metric from $\orth(V)$ to $U$. We write 
\begin{equation}\label{eq:tr_metric}
\langle z,z'\rangle_{U,c}=-c^2\tr(J_zJ_{z'}),\quad\text{for any}\quad z,z'\in U,
\end{equation}
and for any $c\neq 0$.
This scalar product has an index, which we denote by $(r,s)$, and it depends on the choice of the map $J\colon U\to\orth(V)$. The scalar product space $(U,\langle .\,,.\rangle_{U,c})$ is degenerate if $J(U)$ is degenerate with respect to the trace metric. Let us assume that $\big(U,\langle .\,,.\rangle_{U,c}\big)$ is a non-degenerate scalar product space, and let $[.\,,.]_c$ be the 2-step nilpotent Lie algebra structure on $\mathfrak G=V\oplus_{\perp} U$ defined by the map $J\colon U\to\orth(V)$ by means of~\eqref{eq:defJ}. The spaces $V$ and $U$ are orthogonal with respect to the scalar product $\langle .\,,.\rangle_{\g}=\langle .\,,.\rangle_V+\langle .\,,.\rangle_{U,c}$.

\begin{definition}\label{def:pseudo_induced}
The Lie algebra $\mathfrak G=(V\oplus_{\perp} U,[.\,,.]_c,\langle .\,,.\rangle_{\g}=\langle .\,,.\rangle_V+\langle .\,,.\rangle_{U,c})$ described above is called the standard pseudo-metric 2-step nilpotent Lie algebra {\bf induced by the map $J\colon U\to\orth(V,\langle .\,,.\rangle_V)$}.
\end{definition}

Diagonalising the matrix of the scalar product $\langle .\,,.\rangle_V$, we get the matrix $\eta=\diag(I_p,-I_q)$ defining the canonical scalar product $\langle u,v\rangle_{p,q}=\sum_{i=1}^pu_iv_i-\sum_{i=p+1}^{p+q}u_iv_i$ for $u=(u_1,\ldots, u_m)$, $v=(v_1,\ldots, v_m)$, $m=p+q$, and the matrix of the skew-symmetric map $J_z$ will satisfy the condition $\eta_{p,q} J_z^{\mathbf t}\eta_{p,q}=-J_z$. Since the trace does not depend on the choice of coordinates we get a symmetric bilinear form defining a scalar product on $U$, which also can be written as  $\langle z,z'\rangle_{U,c}=c^2\tr(\eta_{p,q} J_z^{\mathbf t}\eta_{p,q} J_{z'})=-c^2\tr(J_zJ_{z'})$. 

\begin{lemma}\label{lem:1}
In the notations above, if the scalar product $\langle z,z'\rangle_{U,c}$ is non-degenerate, then the standard pseudo-metric Lie algebra $\mathfrak G$ induced by $J$ has no abelian factor. If two scalar products $\langle .\,,.\rangle_V^1$ and $\langle .\,,.\rangle_V^2$ on $V$ have equal indices, and the sets of spacelike (timelike and null) vectors are the same, then the commutator $[.\,,.]_c$ does not depend on the choice of the scalar product $\langle .\,,.\rangle_V^i$ on~$V$, $i=1,2$.
\end{lemma}
\begin{proof}
If the scalar product $\langle z,z'\rangle_{U,c}$ is non-degenerate and the map $J\colon U\to \orth(V)$ is injective, then the Lie algebra structure $(\mathfrak G,[.\,,.]_c)$ is unique up to an isomorphism and $\mathfrak G$ has the trivial abelian factor by Lemma~\ref{lem:uniq1}  and Corollary~\ref{cor:trivial_factor}.
\end{proof}

\begin{lemma}\label{lem:2}
Let $(V,\langle .\,,.\rangle_V)$ be a scalar product space, let $U_1$, $U_2$ be two finite dimensional vector spaces, and let $J_1\colon U_1\to\orth(V,\langle .\,,.\rangle_V)$, $J_2\colon U_2\to\orth(V,\langle .\,,.\rangle_V)$ be two injective skew-symmetric linear maps such that $J_1(U_1)=J_2(U_2)=W\subseteq\orth(V,\langle .\,,.\rangle_V)$. Let $\mathfrak G_1=(V\oplus U_1,[.\,,.]_1)$ and $\mathfrak G_2=(V\oplus U_2,[.\,,.]_2)$ be two pseudo-metric Lie algebras induced by the maps $J_1$ and $J_2$. Then $\mathfrak G_1$ and $\mathfrak G_2$ are isomorphic as Lie algebras. 
\end{lemma}

\begin{proof}
It suffices to construct an isomorphism between the Lie algebras $\mathfrak G_1$ and $\mathfrak G_2$ only for the case, when $J_1(U_1)=W=U_2$ and $J_2=\iota\colon W\hookrightarrow\orth(V,\langle .\,,.\rangle_V)$ is the inclusion map. 

We define scalar products on $U_1$ and $U_2$ by
$$
\langle \zeta,\zeta'\rangle_{U_1}=-\tr(J_1(\zeta)J_1(\zeta')),\quad \zeta,\zeta'\in U_1,
$$
$$
\langle z,z'\rangle_{U_2}=-\tr\big(J_2(z)J_2(z')\big)=-\tr(zz'),\quad z,z'\in U_2=W\subseteq\orth(V,\langle .\,,.\rangle_V).
$$
Denote by $[.\,,.]_1$, $[.\,,.]_2$ the commutators constructed by means of these scalar products, respectively. Define the map $\varphi\colon V\oplus U_1\to V\oplus U_2=V\oplus W$ by 
$$
\varphi=\begin{cases}
\Id_V\quad&\text{on}\quad V,
\\
J_1\quad&\text{on}\quad U_1.
\end{cases}
$$
Then we need to show that $\varphi([v,w]^1)=[\varphi(v),\varphi(w)]^2$.
Let $v,w\in V$, $z\in W$ be arbitrarily chosen, and let $\zeta_0\in U_1$ be the unique element such that $J_1(\zeta_0)=z=J_2(z)$. Then,
\begin{eqnarray*}
\langle\varphi([v,w]^1),z\rangle_{U_2} & = &\langle J_1([v,w]^1),J_1(\zeta_0)\rangle_{U_2}=
-\tr(J_1([v,w]^1)J_1(\zeta_0))
\\
&=&
\langle [v,w]^1,\zeta_0\rangle_{U_1}=\langle J_1(\zeta_0)v,w\rangle_{V}
\\
&=&
\langle J_2(z)v,w\rangle_{V}=\langle [v,w]^2,z\rangle_{U_2}=\langle [\varphi(v),\varphi(w)]^2,z\rangle_{U_2},
\end{eqnarray*}
because $\varphi=\Id_V$. This finishes the proof because the scalar product is non-degenerate.
\end{proof}


\subsection{Action of $\GL(m)$ and $\mathfrak{gl}(m)$ on the Lie algebra $\so(p,q)$, $p+q=m$}

If we have two scalar product spaces $(U,(.\,,.)_U)$ and $(V,(.\,,.)_V)$ and an operator $A$ acts as $A\colon U\to V$,
we say that the formula $(A^Tx,y)_U= (x,Ay)_V$ defines the transpose $A^T$ to $A$ with respect to the scalar products $(.\,,.)_U$ and $(.\,,.)_V$. We call attention of the reader that the notation $A^{\mathbf t}$ is used for the transpose matrix $A$.

Let $\eta_{p,q}=\diag(I_p,-I_q)$, and let $A\in \GL(m)$. Define the action $\rho$ of $A$ on $\so(p,q)$ by
$$
Z\mapsto \rho(A)Z=AZA^{\eta_{p,q}},\quad\text{where}\quad A^{\eta_{p,q}}=\eta_{p,q} A^{\mathbf t}\eta_{p,q},\quad Z\in\so(p,q).
$$
Indeed, if $Z^{\eta_{p,q}}=-Z$, then $(AZA^{\eta_{p,q}})^{\eta_{p,q}}=AZ^{\eta_{p,q}}A^{\eta_{p,q}}=-AZA^{\eta_{p,q}}$. We remind that the operation $A^{\eta_{p,q}}$ gives us the transpose matrix to $A$ with respect to the scalar product $\langle.\,,.\rangle_{p,q}$. The action $\rho$ is a left action on $\so(p,q)$.
The map $\rho(A)$ is invertible and its inverse is given by $(\rho(A))^{-1}=\rho(A^{-1})$ which shows that $\rho(A)\in\Aut(\so(p,q))$. Thus, the map
$$
\rho\colon \GL(m)\to \Aut(\so(p,q))
$$ 
defines a group homomorphism.

The differential $d\rho$ of the map $\rho$ is the Lie algebra homomorphism
$$
d\rho\colon \mathfrak{gl}(m)\to\End(\so(p,q))
$$
defined by $\mathcal A\mapsto d\rho(\mathcal A)Z=\mathcal A Z+Z\mathcal A^{\eta_{p,q}}$, with $\mathcal A\in\mathfrak{gl}(m)$, $Z\in \so(p,q)$. 
Let us prove some properties of the maps $\rho$ and $d\rho$.

\begin{lemma}
Let $A\in \GL(m)$ and  $\mathcal A\in\mathfrak{gl}(m)$ be arbitrary  elements. Then
\begin{equation}\label{eq:transp}
\begin{array}{rrr}
 \langle \rho(A)Z,Z'\rangle_{\so(p,q)}& = & \langle Z,\rho(A^{\eta_{p,q}})Z'\rangle_{\so(p,q)},
\\
\langle d\rho(\mathcal A)Z,Z'\rangle_{\so(p,q)} & = & \langle Z,d\rho(\mathcal A^{\eta_{p,q}})Z'\rangle_{\so(p,q)}
\end{array}
\end{equation}
for any $Z,Z'\in\so(p,q)$. We can reformulate~\eqref{eq:transp} as
$$
\big(\rho(A)\big)^{T}=\rho(A^{\eta_{p,q}}),\quad \big(d\rho(\mathcal A)\big)^{T}=d\rho(\mathcal A^{\eta_{p,q}}),
$$
where the superscript $T$ stands for the  transpose map with respect to the scalar product  $\langle .\,,.\rangle_{\so(p,q)}$.
\end{lemma}
\begin{proof} We calculate
$$
\langle \rho(A)Z,Z'\rangle_{\so(p,q)}=-\tr(AZA^{\eta_{p,q}}Z')=-\tr(ZA^{\eta_{p,q}}Z'A)=\langle Z,\rho(A^{\eta_{p,q}})Z'\rangle_{\so(p,q)}
$$
by the property of the trace of the product. The other equality is obtained similarly. 
\end{proof}

\begin{lemma}\label{lem:free_isom}
All 2-step nilpotent free algebras $F_2(p,q)$ with $p+q=m$ are isomorphic.
\end{lemma}

\begin{proof}
To prove Lemma~\ref{lem:free_isom} we show that any 2-step nilpotent Lie algebra $F_2(p,q)=\mathbb R^{p,q}\oplus\so(p,q)$ with $p+q=m$ is isomorphic to $F_2(m)=\mathbb R^m\oplus\so(m)$. Recall the definition of the Lie bracket from Example~1 and formula~\eqref{eq:brack_R}.
Let 
$
v_{ij}=-\frac{1}{2}(E_{ij}-E_{ji})$, $i\leq j=1,\ldots,m$,
be a standard basis of the group $\so(m)$. Here $E_{ij}$ is $(m\times m)$-matrix having $1$ at the position $(ij)$ and $0$ everywhere else. Then the matrices 
$
\phi_{ij}=-\frac{1}{2}(E_{ij}-E_{ji})\eta_{p,q}$, $i\leq j=1,\ldots,m$, form a basis of the space $\so(p,q)$.
We define the isomorphism $ f\colon\so(m)\to\so(p,q)$ by $f(v_{ji})=\phi_{ji}$. 
Then we extend this isomorphism to the isomorphism $F_2(m)\to F_2(p,q)$ by
\begin{eqnarray*}
e_k \mapsto e_k, \qquad 
v_{ij} \mapsto \phi_{ij}, \quad \text{ for }\quad 0<k<m,\ 0<i\leq j\leq m=p+q.
\end{eqnarray*}
It follows that 
\begin{eqnarray*}
f([v_{jk}, e_i+v_{lr}])&=&0=[\phi_{jk} , e_i + \phi_{lr}]=[f(v_{jk}), f(e_i+v_{lr})], \\
f([e_i \,, e_j ])&=& f(v_{ij})=\phi_{ij}=-\frac{1}{2}(E_{ij} - E_{ji})\eta_{p,q}=[e_i, e_j]=[f(e_i), f(e_j)].
\end{eqnarray*}
Hence $f$ is a Lie algebra isomorphism.

At the end of the proof we observe that the orthogonal basis of $F_2(m)$ is mapped to the orthogonal basis of $F_2(p,q)$, $p+q=m$ under the isomorphism $f$. The equalities
$$
\langle E_{ij},E_{\alpha\beta}\rangle_{\so(m)} =-\tr(E_{ij}E_{\alpha\beta})=\delta_{i\alpha}\delta_{j\beta},
$$
show that the basis $v_{ij}-\frac{1}{2}(E_{ij}-E_{ji})$ is orthonormal with respect to the trace metric,
and the basis $\phi_{ij}=-\frac{1}{2}(E_{ij}-E_{ji})\eta_{p,q}$ of the space $\so(p,q)$ satisfies the relations
$$
\langle (\phi_{ij}),(\phi_{\alpha\beta})\rangle_{\so(p,q)} =-\tr\big(\phi_{ji}\phi_{\alpha\beta}\big)=\nu_{ij}\delta_{i\alpha}\delta_{j\beta},
$$
where 
$$
\nu_{ij}=\begin{cases}
1,\quad&\text{if}\quad i< j\leq p\ \text{or}\ i>p
\\
-1\quad&\text{if}\quad j>p\ \text{and}\ i\leq p.
\end{cases}
$$ 
\end{proof}

Lemma~\ref{lem:free_isom} allows us to reformulate some results proved in~\cite{Eber04,Eber03,Eber02} for the 2-step free Lie algebras $F_2(p,q)$. Let us denote by $\Aut(F_2(p,q))$ the group of automorphisms of $F_2(p,q)$. 

\begin{lemma}\label{lem:free_aut}
For any $\phi\in\Aut(F_2(p,q))$, there exists a unique element $A \in \GL(m)$, $m=p+q$ and $S \in \Hom(\mathbb{R}^{p,q} ,\so(p,q))$, such that
\begin{itemize}
\item[a)] $\phi(x)=Ax+S(x) \qquad \text{ for all } x \in \mathbb{R}^{p,q}$,
\item[b)]$ \phi(Z)=AZA^{\eta_{p,q}} \quad\qquad \text{ for all } Z \in \so(p,q).$
\end{itemize}
Conversely, given $(A,S) \in \GL(m) \times \Hom(\mathbb{R}^{p,q},\so(p,q))$, $m=p+q$, there is a unique automorphism $\phi \in \Aut(F_2(p,q))$ that satisfies $a)$ and $b)$. 
\end{lemma}

\begin{proof}
An analog of Lemma~\ref{lem:free_aut} for the free group $F_2(m)$ was proved in~\cite{Eber02}. Let $f$ be 
an isomorphism between $F_2(m)$ and $F_2(p,q)$, $m=p+q$, which exists by Lemma~\ref{lem:free_isom}. Then, for any $\varphi\in\Aut(F_2(m))$ the superposition
$\phi=f\circ \varphi \circ f^{-1}$ is an automorphism of $F_2(p,q)$. Thus, for every automorphism $\phi\in\Aut(F_2(p,q))$, there exists a unique $\varphi \in \Aut(F_2(m))$, $m=p+q$, with $\phi= f\circ \varphi \circ f^{-1}$, and moreover, a unique $A \in \GL(m)$, $S' \in \Hom(\mathbb{R}^{m} ,\so(m))$, such that the properties a) and b) are satisfied with $f \circ S'=S\in \Hom(\mathbb{R}^{p,q} ,\so(p,q))$.
The converse statement follows easily. 
\end{proof}

Let $\g$ be a $2$-step nilpotent Lie algebra with $\dim([\g,\g])=n$,  with $m$-dimensional complement $V$, 
and with the adapted basis $\{w_1, \dotso,w_{m},Z_1, \dotso, Z_n\}$, see Section~\ref{subseq:standard_metric}.
If $[w_i,w_j]=\sum_{k=1}^nC^k_{ij}Z_k$, then we call the space $\mathcal C=\spn\{C^1,\ldots,C^n\}\subset\so(m)$ the {\it structure space} and the spaces $\mathcal D_{p,q}=\spn \{C^1\eta_{p,q}, \dotso,C^{n} \eta_{p,q} \} \subset \so(p,q) $ are called the {\it structure $\eta_{p,q}$-spaces}.  In the following propositions, we aim at showing that the structure $\eta_{p,q}$-spaces of the $2$-step nilpotent Lie algebra $\g$ are orbits in the Grassmann manifold.

\begin{proposition}
Let $\{w_1, \dotso,w_{m},Z_1, \dotso, Z_n\}$ and $\{\hat w_1, \dotso,\hat w_{m},\hat Z_1, \dotso, \hat Z_n\}$ be two adapted bases of a 2-step nilpotent Lie algebra $\g$ with corresponding structure $\eta_{p,q}$-spaces $\mathcal D_{p,q}=\spn \{C^1 \eta_{p,q}, \dotso,C^{p} \eta_{p,q} \}$ and $\hat{\mathcal D}_{p,q}=\spn\{\hat C^1 \eta_{p,q}, \dotso,\hat C^{p} \eta_{p,q} \}$. Let $A \in \GL(m)$, $m=p+q$ be such that $\hat v_i=\sum_{j=1}^{m}{A_{ij}v_j}$. Then $A\mathcal D_{p,q}A^{\eta_{p,q}}=\hat{\mathcal D}_{p,q}$. 
\end{proposition}
\begin{proof}
The statement follows from the definition of the action of $\GL(m)$ on $\so(p,q)$.
\end{proof}

\begin{proposition}\label{eqg}
Let $d$ be an integer with $1 \leq d \leq \dim(\so(p,q))$. Let $W_1, W_2 \subset \so(p,q)$ be two $d$-dimensional non-degenerate with respect to $\la.\,,.\ra_{\so(p,q)}$ subspaces. Then, the following statements are equivalent:
\begin{itemize}
\item[1)] The Lie algebra $F_2(p,q)/ W_1$ is isomorphic to $F_2(p,q)/ W_2$; 
\item[2)] There exists an element $A\in \GL(m)$, $m=p+q$ such that $AW_1A^{\eta_{p,q}}=W_2$;
\item[3)] The Lie algebra $F_2(p,q)/ W_1^{\perp}$ is isomorphic to $F_2(p,q)/ W_2^{\perp}$.
\end{itemize}
\end{proposition}

\begin{proof} First we show that the statements 1) and 2) are equivalent.
Recall that for any pair $(p,q)$ with $p+q=m$ and $W_1, W_2 \subset \so(p,q)$ we have $W_1 \eta_{p,q}$, $W_2 \eta_{p,q}\in\so(m)$. The Lie algebras $F_2(m)/(W_1 \eta_{p,q})$ and $F_2(m)/(W_2 \eta_{p,q})$ are shown to be isomorphic~\cite{Eber02}, if and only if, there exists $A \in \GL(m)$, such that $A W_1 \eta_{p,q} A^{\mathbf t}=W_2 \eta_{p,q}$. The last equality can be written as $A W_1 A^{\eta_{p,q}} = W_2$. Let $f$ be an isomorphism between $F_2(m)$ and $F_2(p,q)$. Hence, $W_i  =f(W_i\eta_{p,q})$ and $F_2(p,q)/ W_i=f(F_2(m)/ (W_i \eta_{p,q}))$ for $i=1,2$. This implies that $F_2(m)/(W_1 \eta_{p,q})$ and $F_2(m)/(W_2 \eta_{p,q})$ are isomorphic, if and only if, $F_2(p,q)/ W_1$ is isomorphic to $F_2(p,q)/ W_2$.

Now we show that the statements 1) and 3) are equivalent. The arguments above illustrates that $F_2(p,q)/ W_1$ is isomorphic to $F_2(p,q)/ W_2$, if and only if, $F_2(m)/(W_1 \eta_{p,q})$ is isomorphic to $ F_2(m)/(W_2 \eta_{p,q})$. This is equivalent to the statement that $F_2(m)/ (W_1 \eta_{p,q})^{\perp}$ is isomorphic to $F_2(m)/ (W_2 \eta_{p,q})^{\perp}$ by~\cite{Eber02}. Define the map $f^* \colon F_2(m) \to F_2(p,q)$ by 
\begin{eqnarray*}
e_i \mapsto \begin{cases} e_i, & \text{ for } 1 \leq i \leq p, 
\\
-e_i, & \text{ for } p+1 \leq i \leq p+q, \end{cases} \qquad \frac{1}{2}(E_{ij}-E_{ji}) \mapsto \frac{1}{2}(E_{ij}-E_{ji})\eta_{p,q}. 
\end{eqnarray*}
Then $F_2(m)/ (W_1 \eta_{p,q})^{\perp}$ is isomorphic to the quotient $F_2(m)/ (W_2 \eta_{p,q})^{\perp}$, if and only if, $F_2(p,q)/ \eta_{p,q} (W_1 \eta_{p,q})^{\perp}$ is isomorphic to $F_2(p,q)/ \eta_{p,q} (W_2 \eta_{p,q})^{\perp}$. 

It only remains to prove that $W_i$, $i=1,2$, is orthogonal to $\eta_{p,q}(W_i \eta_{p,q})^{\perp}$ with respect to the metric $\la . \,,.\ra_{\so(p,q)}$. For any $w \in W_i$ and any $v \in (W_i \eta_{p,q})^{\perp}$ it follows that
\begin{eqnarray*}
\la w, \eta_{p,q} v \ra_{\so(p,q)}= - \tr(w\eta_{p,q} v)= \la w \eta_{p,q}, v \ra_{\so(m)}=0,
\end{eqnarray*} 
as $w \eta_{p,q} \in W_i \eta_{p,q}$ and $v\in(W_i \eta_{p,q})^{\perp}$. Since $\dim(\eta_{p,q}(W_i \eta_{p,q})^{\perp})=\dim(\so(p,q))-\dim(W_i)$ and $W_i$ non-degenerate, it follows that $ \eta_{p,q} (W_i \eta_{p,q})^{\perp} = W_i^{\perp}$.
\end{proof}

\begin{proposition}\label{structurespace}
Let $w_1, \dotso, w_{m},Z_1, \dotso,Z_n$ be an adapted basis for a $2$-step nilpotent Lie algebra $\g$ with the structure space $\mathcal C=\spn\{C^1, \dotso,C^n\}\subset \so(m)$. 

Let $\rho \colon F_2(p,q) \to \g$, $p+q=m$, be a unique Lie algebra homomorphism defined by $\rho(e_i)=w_i$ for $i=1, \dotso,m$. Then $\rho$ is surjective and if $\mathcal C\eta_{p,q} \subset \so(p,q)$ is non-degenerate, then $\ker(\rho)$ is the orthogonal complement $(\mathcal C\eta_{p,q})^{\perp}$ to $\mathcal C \eta_{p,q}$ in $\so(p,q)$ with respect to $\la .\,,. \ra_{\so(p,q)}$.
\end{proposition}
\begin{proof}
It is known that the Lie algebra homomorphism $\rho_1 \colon F_2(m) \to \g$ with $\rho_1(e_i)=w_i$ for  $i=1, \dotso,m$ is surjective and  $\ker(\rho_1)$ is the orthogonal complement to $\mathcal C$ in $\so(m)$ with respect to $\la .\,,.\ra_{\so(m)}$, see for instance~\cite{Eber02}. Then, we define the surjective linear map $\rho= \rho_1 \circ (f^*)^{-1} \colon F_2(p,q) \to \g$ with $f^*$ to be the isomorphism between $F_2(m)$ and $F_2(p,q)$ from the proof of Proposition~\ref{eqg}. 
Proposition~\ref{eqg} also shows that if $\mathcal C\eta_{p,q}$ is non-degenerate in $\so(p,q)$, then $(\mathcal C\eta_{p,q})^{\perp}=\eta_{p,q}( \mathcal C^{\perp})$. Since
$$(f^*)^{-1}((\mathcal C\eta_{p,q})^\perp)=(f^*)^{-1}(\eta_{p,q} (\mathcal C^{\perp}))=\eta_{p,q}^2 \mathcal C^{\perp}=\mathcal C^{\perp}=\ker(\rho_1),$$ 
it follows that $\ker(\rho)=(\mathcal C\eta_{p,q})^\perp$.
\end{proof}

\begin{corollary}\label{iso1}
Let $W_1$ and $W_2$ be non-degenerate $d$-dimensional subspaces of $\so(p,q)$, and let $\mathcal G_1=\mathbb{R}^{p,q} \oplus W_1$ and $\mathcal G_2=\mathbb{R}^{p,q} \oplus W_2$ be the corresponding standard pseudo-metric $2$-step nilpotent Lie algebras, then the following statements are equivalent.
\begin{itemize}
\item The Lie algebra $\mathcal G_1$ is isomorphic to $\mathcal G_2$.  
\item There exists $A \in \GL(m)$, such that $AW_1A^{\eta_{p,q}}=W_2$, $p+q=m$.
\end{itemize}
\end{corollary}
\begin{proof}
The Lie algebras $\mathcal G_i$ are isomorphic to $F_2(p,q)/ W_i^{\perp}$ for $i=1,2$ by Proposition~\ref{structurespace}. The statement of the corollary follows now by using Proposition~\ref{eqg}.
\end{proof}
 
Assume that $\g$ is a $2$-step nilpotent Lie algebra with a $1$-dimensional commutator ideal $[\g,\g]$, and assume that there exist positive integers $p,q$ and a non-degenerate one-dimensional subspace $W$ in $\so(p,q)$, such that $\g$ is isomorphic to $\mathbb{R}^{p,q} \oplus W$ with $m=p+q \geq 2$. Let us define the set $\mathcal{A}_{p,q}=\{ Z \in \so(p,q) \vert \text{ rank Z is maximal} \}$. 

\begin{corollary}
The group $O(m)$ acts transitively by $\eta_{p,q}$-conjugation on $\mathcal{A}_{p,q}$, where $m=p+q$.
\end{corollary}
 
\begin{proof}
We define the set $\mathcal{A}_m=\{ Z \in \so(m) \vert \text{ rank Z is maximal} \}$ which is Zariski open in $\so(m)$. The group $O(m)$ acts transitively on it by conjugation, see~\cite{Eber02}. Notice that $\mathcal{A}_m \eta_{p,q}=\mathcal{A}_{p,q}$. For every $Z,Y\in \mathcal{A}_m$ there exists an $A\in O(m)$, such that $Z=AYA^{-1}=AYA^{\mathbf t}$. Then, 
$$Z \eta_{p,q}=A Y \eta_{p,q}^2A^{-1} \eta_{p,q}=A Y \eta_{p,q}^2A^{t} \eta_{p,q}=A Y \eta_{p,q} A^{\eta_{p,q}}$$ with $Z\eta_{p,q}, Y\eta_{p,q} \in \mathcal{A}_{p,q}$. This finishes the proof.
\end{proof}


\section{Lie triple system as a rational subspace}\label{sec:lts}



\subsection{Lie triple systems}


In the present section we collect some useful facts about the Lie triple system of an arbitrary Lie algebra $\g$. A reader familiar with this notion can skip this section. 

\begin{definition}\label{def:Lie_triple_syst}
A subspace $W$ of $\g$ is called a {\it Lie triple system} if $[W,[W,W]]\subset W$. 
\end{definition}
Define the centre $\mathfrak Z(W)$ of $W$ by
\begin{equation}\label{eq:centre_triple}
\mathfrak Z(W)=\{a\in W\mid\ [a,b]=0\ \ \text{for all}\ \ b\in W\}.
\end{equation}
We say that $\mathfrak Z(W)$ is compact if $\exp(\mathfrak Z(W))$ is a compact subgroup of the group $G$ corresponding to the Lie algebra $\g$. 

\begin{proposition}
The set $\exp(\mathfrak{Z}(W))$ is a connected abelian subgroup of the Lie group corresponding to the Lie algebra~$\g$.
\end{proposition} 

\begin{proof}
We observe that the centre $\mathfrak{Z}(W)$ of the Lie triple system $W$ is commutative and the standard arguments finish the proof.
\end{proof}

\begin{proposition}\label{prop:use1} 
Let $(\mathfrak g,[.\,,.])$ be a Lie algebra, and let $W$ be its Lie triple system. Then, $[W,W]$ and $W+[W,W]$ are subalgebras of $\mathfrak g$.
\end{proposition}

\begin{proof}
To show that $[W,W]$ is a subalgebra, we need to check 
$$
\big[[W,W],[W,W]\big]\subset [W,W].
$$
Let $w_1,w_2,w'_1,w'_2\in W$, then with the notation $[w'_1,w'_2]=u$, we get 
$$
\big[[w_1,w_2],[w'_1,w'_2]\big]=\big[[w_1,w_2],u\big]=-[[w_2,u],w_1]-[[u,w_1],w_2]\in [W,W],
$$
by the Jacobi identity, because $[w_2,u],[u,w_1]\in W$ by the definition of the Lie triple system, $[W,[W,W]]\subset W$.

To prove the second statement we choose arbitrary $a,b,c,x,y,z\in W$ and write
$$
\big[a+[b,c],x+[y,z]\big]
=  [a,x]+[a,[y,z]]+[[b,c],x]+[[b,c],[y,z]]\in W+[W,W]
$$
by the first statement and by the definition of the Lie triple system.
\end{proof}

\begin{remark}
Let us denote the sets in Proposition~\ref{prop:use1} by $\mathfrak p=W$, $\mathfrak t=[W,W]$, and $\mathcal L=W+[W,W]$. Then Proposition~\ref{prop:use1} implies that the Lie algebra $\mathcal L$ admits the decomposition $\mathcal L=\mathfrak t+\mathfrak p$ with the Cartan pair $\mathfrak t,\mathfrak p$ satisfying the following properties
\begin{equation}\label{eq:Cartan}
[\mathfrak t,\mathfrak t]\subseteq \mathfrak t,\quad [\mathfrak t,\mathfrak p]\subseteq \mathfrak p,\quad [\mathfrak p,\mathfrak p]\subseteq \mathfrak t.
\end{equation}
Note that if a Lie algebra $\mathfrak h$ admits a direct sum decomposition $\mathfrak h=\mathfrak t\oplus\mathfrak p$ satisfying~\eqref{eq:Cartan}, then there is an involution $\theta\colon\mathfrak h\to \mathfrak h$ $(\theta^2=\Id_{\mathfrak h})$ possessing the following properties 
$$ \mathfrak t\subset\mathfrak h\quad\text{is such that}\quad \theta(t)=t,\ \ \forall\ t\in\mathfrak t,$$  
$$\mathfrak p\subset\mathfrak h\quad\text{is such that}\quad \theta(p)=-p,\ \ \forall\ p\in\mathfrak p.$$
\end{remark}

Given a Lie algebra $(\mathfrak g,[.\,,.])$, we denote by $\ad_v\colon \mathfrak g\to\mathfrak g$ the linear map defined by $\ad_v(u)=[v,u]$. The map $\ad\colon \mathfrak g\to\End(\mathfrak g)$ is a Lie algebra homomorphism, named the adjoint representation of the Lie algebra $\mathfrak g$. The kernel of the adjoint map $\ad$ is the centre of the Lie algebra. 

\begin{definition}
Let $\mathfrak g$ be a Lie algebra. A scalar product $\langle .\,,.\rangle$ on $\mathfrak g$ is called $\ad$-invariant if 
\begin{equation}\label{def:ad_inv}
\langle \ad_v(u),w\rangle=-\langle u,\ad_v(w)\rangle.
\end{equation}
\end{definition}
Equivalently, it can be stated that the map $\ad_v\colon \mathfrak g\to\mathfrak g$ is skew-symmetric with respect to the scalar product $\langle .\,,.\rangle$.

\begin{proposition}\label{prop:WL}
Let $(\mathfrak g,[.\,,.])$ be a Lie algebra, let $W$ its Lie triple system, and denote $\mathcal L=W+[W,W]$. Let $\mathfrak Z(W)$
be the centre of $W$, see~\eqref{eq:centre_triple}, and let
$(.\,,.)_{\mathcal L}$ an $\ad$-invariant inner product on $\mathcal L$.
Then the following statements hold.
\begin{itemize}
\item[1.] {Denote by $\mathfrak Z(\mathcal L)$ the centre of $\mathcal L$. Then the Lie algebra $\mathcal L$ is decomposed into the direct sum of two ideals
$
\mathcal L=\mathfrak Z(\mathcal L)\oplus_{\bot}[\mathcal L,\mathcal L]$,
where the decomposition is orthogonal with respect to $(.\,,.)_{\mathcal L}$.}
\item[2.] {$\mathfrak Z(W)\subseteq \mathfrak Z(\mathcal L)$.}
\item[3.] {The centre $\mathfrak Z([\mathcal L,\mathcal L])$ of $[\mathcal L,\mathcal L]$ is trivial.}
\item[4.] {If $\mathfrak Z(\mathcal L)\neq 0$, then $\mathfrak Z(W)\neq 0$.}
\item[5.] {If $\mathfrak Z(W)=0$, then $\mathcal L=[\mathcal L,\mathcal L]$.}
\end{itemize}
\end{proposition}

\begin{proof} {\sc Proof of 1.}
Let $\mathfrak Z(\mathcal L)$ be the centre of the algebra $\mathcal L$. We will show that 
\begin{equation}\label{eq:2_1}
\mathfrak Z(\mathcal L)=[\mathcal L,\mathcal L]^{\bot}
\end{equation}
with respect to the inner product $(.\,,.)_{\mathcal L}$. Let $z\in[\mathcal L,\mathcal L]^{\bot} $ and let $u,v\in \mathcal L$ be arbitrarily chosen. Then,
$$
( [u,z],v)_{\mathcal L}=-( z,[u,v])_{\mathcal L}=0,
$$
because the inner product is $\ad$-invariant. It shows that $[u,z]=0$, and therefore, $z\in \mathfrak Z(\mathcal L)$, which implies that $\mathfrak Z(\mathcal L)\supset[\mathcal L,\mathcal L]^{\bot}$. Reversing the arguments we show the inverse inclusion, and conclude that $\mathcal L=\mathfrak Z(\mathcal L)\oplus_{\bot}[\mathcal L,\mathcal L]$ by~\eqref{eq:2_1}.

\medskip

{\sc Proof of 2.} Choose arbitrarily $z\in \mathfrak Z(W)$ and $u,v,w\in W$. Then we obtain
$\big[z,u+[v,w]\big]=[z,u]-[w,[z,v]]-[v,[w,z]]=0$ by the Jacobi identity. Thus, $z\in \mathfrak Z(\mathcal L)$.

\medskip

{\sc Proof of 3.} Let $z\in \mathfrak Z([\mathcal L,\mathcal L])$. Then for any $u\in\mathcal L$ and $a\in[\mathcal L,\mathcal L]$, we have
$$
0=(u,[z,a])_{\mathcal L}=(z,[u,a])_{\mathcal L}.
$$
Therefore, $z\in[\mathcal L,\mathcal L]^{\bot}=\mathfrak Z(\mathcal L)$, and simultaneously, $z\in \mathfrak Z([\mathcal L,\mathcal L])\subset[\mathcal L,\mathcal L]$. We conclude that $z=0$ by item 1.

\medskip

{\sc Proof of 4.} Let $z\in\mathfrak Z(\mathcal L)$ and $z\neq 0$. Then 
$$
[\mathcal L,\mathcal L]\subsetneq\mathcal L
\quad([\mathcal L,\mathcal L]\quad\text{is a proper subset of}\quad \mathcal L\quad\text{by item 1.})
$$ 
Since $[\mathcal L,\mathcal L]=[W,W]+[W,[W,W]]$, we conclude that 
$$[W,[W,W]]\subsetneq W\quad([W,[W,W]]\quad\text{is a proper subset of}\quad W).
$$
Let $[W,[W,W]]^{\bot}$ be the orthogonal complement to $[W,[W,W]]$ in $\mathcal L$ with respect to $(.\,,.)_{\mathcal L}$. Then, $A=W\cap [W,[W,W]]^{\bot}\neq\emptyset$. We claim that $A\subset \mathfrak Z(W)$. Pick arbitrarily $a,b,c\in W$ and $y\in A$, $y\neq 0$. Then,
$
[c,[a,b]]\subset [W,[W,W]],
$ and therefore,
$$
0=(y,[c,[a,b]])_{\mathcal L}=([y,c],[a,b])_{\mathcal L}\quad\Longrightarrow\quad [y,W]\subset[W,W]^{\bot}.
$$
On the other hand, $[y,W]\subset[W,W]$, which implies $[y,W]=0$ and thus $y\in \mathfrak Z(W)$. We conclude that $\mathfrak Z(W)\neq 0$.

\medskip

{\sc Proof of 5.} If $\mathfrak Z(W)= 0$, then we conclude that $\mathfrak Z(\mathcal L)= 0$ by item 4, and $\mathcal L=[\mathcal L,\mathcal L]$ by item 1.
\end{proof}

Next step is the study of irreducible Lie triple systems in $\g$. We recall some definitions and properties. 

\begin{definition}
The Killing form $B_{\mathfrak g}$ on a Lie algebra $\mathfrak g$ is the map
$B_{\mathfrak g}\colon \mathfrak g\times\mathfrak g\to\mathbb R$ defined by
$$
B_{\mathfrak g}(u,v):=\tr(\ad_u\circ \ad_v).
$$
\end{definition}

The kernel of the Killing form $B_{\mathfrak g}$ on a Lie algebra $\mathfrak g$ is defined as
$$
\ker(B_{\mathfrak g})=\{x\in\mathfrak g\mid \ B_{\mathfrak g}(x,u)=0\ \text{for all}\ u\in \mathfrak g\}.
$$
Notice that the kernel of a Killing form is always an ideal of $\mathfrak g$ due to the adjoint invariance of the Killing form. Indeed if $x\in \ker(B_{\mathfrak g})$, then for any $u,v\in \mathfrak g$
$$
B_{\mathfrak g}([x,v],u)=B_{\mathfrak g}(x,[v,u])=0\quad\Longrightarrow\quad [ \ker(B_{\mathfrak g}),\mathfrak g]\subset \ker(B_{\mathfrak g}).
$$ 
According to the Cartan criterion, a Lie algebra $\mathfrak g$ is semisimple if and only if the Killing form $B_{\mathfrak g}$ is non-denenerate on $\mathfrak g$, or equivalently the kernel $\ker(B_{\mathfrak g})$ is trivial. In particular, since the Lie algebra $\mathfrak{so}(p,q)$ is simple, the Killing form $B_{\mathfrak{so}(p,q)}$ is non-degenerate.

\begin{definition}
Let $\mathfrak g$ be a Lie algebra. A Lie triple system $W$ of $\mathfrak g$ is called irreducible if there are no Lie triple systems $W_1$ and $W_2$ of $\mathfrak g$ such that 
$$
W=W_1\oplus W_2,\qquad [W_1,W_2]=\{0\}.
$$
\end{definition}

\begin{proposition}\label{prop:red}
Let $W$ be a non-abelian Lie triple system of $\mathfrak g$,  let $\mathfrak Z(W)$ be its centre, let $\mathcal L=W+[W,W]$, and let $(.\,,.)_{\mathcal L}$ be an $\ad$-invariant inner product on $\mathcal L$. Then the following properties hold.
\begin{itemize}
\item[1.] {If $\mathfrak Z(W)\neq 0$, and if $W_1=\mathfrak Z(W)^{\bot}$ is its orthogonal complement in $W$ with respect to $(.\,,.)_{\mathcal L}$, then $W_1$ is a non-abelian Lie triple system and $W=\mathfrak Z(W)\oplus_{\bot} W_1$.}
\item[2.] {There are non-abelian irreducible Lie triple systems $W_j$ with $[W_i,W_j]=\{0\}$, $i\neq j$, such that $W=\mathfrak Z(W)\oplus\big(\oplus_{j=1}^{N}W_j\big)$.}
\item[3.] {If $\mathcal L=W+[W,W]$ and $W\cap[W,W]\neq\{0\}$, then $W$ is reducible, and $W=W_1\oplus_{\bot} W_2$, $[W_1,W_2]=\{0\}$, where $W_1=W\cap[W,W]$, $W_2$ is the orthogonal complement of $W_1$ in $W$ with respect to $(.\,,.)_{\mathcal L}$.}
\end{itemize}

If moreover, $W$ is an irreducible non-abelian Lie triple system of $\mathfrak g$, then
\begin{itemize}
\item[4.] {$
\mathcal L=W=[W,W]$ or $W\cap[W,W]=\{0\}$, and $\mathcal L=W\oplus[W,W];
$
furthermore, the Lie algebra $\mathcal L$ has trivial centre.}
\item[5.] {If $\mathcal L=W\oplus[W,W]$, then $B_{\mathcal L}(W,[W,W])=0$. Thus, the decomposition into the direct sum is orthogonal with respect to the Killing form $B_{\mathcal L}$.}
\end{itemize}
\end{proposition}

\begin{proof}
{\sc Proof of 1.} Let $a,b,c\in W_1$ and $z\in\mathfrak Z(W)$ be arbitrary elements. Then
$$
([a,[b,c]],z)_{\mathcal L}=-([a,z],[b,c])_{\mathcal L}=0\quad\Longrightarrow\quad[W_1,[W_1,W_1]]\subset\mathfrak Z(W)^{\bot}=W_1,
$$
and we conclude that $W_1$ is a Lie triple system.
\\

{\sc Proof of 2.} Since $W$ is non-abelian, it follows that $\mathfrak Z(W)\neq W$, and we can write $W=\mathfrak Z(W)\oplus W_1$, where $W_1$ is the orthogonal complement of $\mathfrak Z(W)$ in $W$ with respect to $(.\,,.)_{\mathcal L}$. The set $\mathfrak Z(W)$ is obviously a Lie triple system. The set $W_1$ is also a Lie triple system by arguing as in the proof of item~1.

If $W_1$ is irreducible, then we finish the proof. Otherwise, we write $W_1=\oplus_{j=2}^{N}W_j$, where $W_j$ are non-abelian irreducible Lie triple systems such that $[W_i,W_j]=0$, $i\neq j$, that finishes the proof of item~2.
\\

{\sc Proof of 3.} We need to prove that $W_1$ and $W_2$ are Lie triple systems of $\mathfrak g$ such that $[W_1,W_2]=0$.

Claim 1: {\it $W_1$ is an ideal of $\mathcal L$.} Let $a,b,c\in W$ and $x\in W_1=W\cap[W,W]$ be arbitrary. Then $[a+[b,c],x]=[a,x]+[[b,c],x]$, and $[a,x]\in W\cap[W,W]$, since $x\in W\cap[W,W]$. Thus, $[a,W_1]\subset W_1$. Analogously, $[[b,c],x]\in[[W,W],W]\subseteq W$ by $x\in W$ and $[[b,c],x]\in[[W,W],[W,W]]\subset[W,W]$ since $x\in[W,W]$, and therefore, $[[b,c],W_1]\subset W_1$. This shows that $[\mathcal L,W_1]\subset W_1$. 

Claim 2: {\it $W_1$ and $W_2$ are $\ad_{[W,W]}$ invariant.} In particular, Claim 1 implies that
$[[W,W],W_1]\subset W_1$, i.e. $W_1$ is $\ad_{[W,W]}$-invariant. To show the same for $W_2$ we take an arbitrary $v\in W$ and write $v=v_1+v_2$, where $v_1\in W_1$, $v_2\in W_2$. Then
$$
W\supset \ad_{[W,W]}(v)=\ad_{[W,W]}(v_1)+\ad_{[W,W]}(v_2). 
$$
Because of $\ad_{[W,W]}(v_1)\subset W_1$, we conclude that $\ad_{[W,W]}(v_2)\subset W_2$ for any $v_2\in W_2$.

Claim 3: {\it $W_1$ and $W_2$ are Lie triple systems.} Note that $[W_1,W_1]\subset [W,W]$, since $W_1=W\cap[W,W]$ and $[W_2,W_2]\subset [W,W]$ by $W_2\subset W$. Then
$$
[[W_1,W_1],W_1]\subseteq[[W,W],W_1]\subseteq W_1
$$
because $W_1$ is $\ad_{[W,W]}$ invariant. The same reasons work for $W_2$. 

Claim 4: {\it $[W_1,W_2]=0$.} Notice
$$
([W,W],[W_1,W_2])_{\mathcal L}=(\underbrace{[[W,W],W_1]}_{\subset W_1},W_2)_{\mathcal L}=0\quad\text{by}\quad W_1=W_2^{\bot}.
$$ 
Thus, $[W_1,W_2]\subset [W,W]^{\bot}$, and on the other hand $[W_1,W_2]\subset [W,W]$, since both $W_1,W_2$ are subsets of $W$. We conclude that $[W_1,W_2]=0$.
\\

{\sc Proof of 4.} Let $W_1=W\cap[W,W]$, and let $W_2$ be the orthogonal complement to $W_1$ in $W$ with respect to the inner product $(.\,,.)_{\mathcal L}$. Then the consideration is reduced to two cases
$$
(a)\ \ W_1=0\qquad\text{or}\qquad (b)\ \ W_1\neq0.
$$
In the case $(a)$, we get $\mathcal L=W\oplus[W,W]$. In the case $(b)$ we obtain 
$
W=W_1\oplus_{\bot}W_2,
$
and by the assumption of the irreducibility we conclude that $W_2=\{0\}$. Thus,
$$
W_1=W\quad\Longrightarrow\quad W=W_1=W\cap[W,W]\subseteq[W,W].
$$
By taking $\ad_W$ of both sides of the latter equality, we obtain $[W,W]\subseteq[W,[W,W]]\subseteq W$. So we conclude that $W=[W,W]=\mathcal L$.

Let us show that the centre $\mathfrak Z(\mathcal L)$ is trivial. If $\mathfrak Z(\mathcal L)\neq 0$, then $\mathfrak Z(W)\neq 0$ by item~4 of Proposition~\ref{prop:WL}. If $\mathfrak Z(W)\neq 0$, then $W$ is reducible by the proofs of item~1 and~2 of Proposition~\ref{prop:red}. Thus, the centre $\mathfrak Z(\mathcal L)$ is trivial.
\\

{\sc Proof of 5.} Let $\mathcal L=W\oplus[W,W]$ and let $x\in W$, $y\in[W,W]$ arbitrarily chosen. Then
$$
\ad_y\ad_x([W,W])\subset \ad_y([W,[W,W]])\subset\ad_y(W)\subset[[W,W],W]\subset W
$$
and 
$$
\ad_y\ad_x(W)\subset\ad_y([W,W])\subset[[W,W],[W,W]]=[W,W].
$$
Thus, the operator $\ad_y\ad_x$ acts on $\mathcal L=W\oplus[W,W]$ by interchanging the spaces $W$ and $[W,W]$ in the direct sum $W\oplus[W,W]$, i.e.,
$\ad_y\ad_x(W\oplus[W,W])=[W,W]\oplus W$, and therefore, 
$$
0=\tr(\ad_y\ad_x)=B_{\mathcal L}(x,y)\quad\Longrightarrow\quad B_{\mathcal L}(W,[W,W])=0.
$$
\end{proof}

\begin{proposition}\label{prop:main} 
Let $W$ be a Lie triple system of $\mathfrak g$, let $\mathfrak Z(W)$ be the centre of $W$, and let $\mathcal L=W+[W,W]$. Then for any $\ad$-invariant inner product $(.\,,.)_{\mathcal L}$ we have
\begin{itemize}
\item[1)] {$\mathfrak Z(W)=\mathfrak Z(\mathcal L)$,  $\mathcal L=\mathfrak Z(W)\oplus[\mathcal L,\mathcal L]
$, and the direct sum is orthogonal with respect to $(.\,,.)_{\mathcal L}$;}
\item[2)]{ Let $W_1$ denote the orthogonal complement to $\mathfrak Z(W)$ in $W$ with respect to the inner product $(.\,,.)_{\mathcal L}$. Then $W_1$ is a Lie triple system of $\mathfrak g$ and the ideal $[\mathcal L,\mathcal L]$ of $\mathcal L$ can be written as $[\mathcal L,\mathcal L]=W_1+[W_1,W_1]$.}
\end{itemize}
\end{proposition}

\begin{proof} {\sc Proof of 1.}
If $W$ is abelian, then it is nothing to prove. Let $W$ be a non-abelian Lie triple system of $\mathfrak g$. Then we can write
$$
W=\mathfrak Z(W)\oplus\big(\bigoplus_{j=1}^{N}W_j\big),
$$
where $W_j$ are irreducible Lie triple systems such that $[W_i,W_j]=0$, $i\neq j$, by item 2 of Proposition~\ref{prop:red}. We denote by $\mathcal L_j=W_j+[W_j,W_j]$, $j=1,\ldots, N$, Lie subalgebras of $\mathfrak g$. The algebras $\mathcal L_j$ have trivial centres by item 4 of Proposition~\ref{prop:red}. Moreover,
$
[\mathcal L_i,\mathcal L_j]=0
$ for $i\neq j$
by the Jacobi identity and by $[W_i,W_j]=0$. Thus, $\mathcal L=W+[W,W]=\mathfrak Z(W)\oplus\big(\bigoplus_{j=1}^{N}\mathcal L_j\big)$. The Lie algebra $\mathcal L_0=\oplus_{j=1}^{N}\mathcal L_j$ has a trivial centre $\mathfrak Z(\mathcal L_0)$ because each of the Lie algebras $\mathcal L_j$ has a trivial centre, and they mutually commute. Since we have $\mathfrak Z(W)\subseteq\mathfrak Z(\mathcal L)$ by item 2 of Proposition~\ref{prop:WL}, we conclude that $\mathfrak Z(W)=\mathfrak Z(\mathcal L)$. Indeed, if we assume that there is $x\in\mathfrak Z(\mathcal L)$ and $x\notin\mathfrak Z(W)$, then $x\in\mathcal L_0$ due to the decomposition $\mathcal L=\mathfrak Z(W)\oplus\mathcal L_0$. But then $
[x,y]=0
$ for any $y\in\mathcal L$, and in particular,  $[x,y_0]=0$ for any $y_0\in\mathcal L_0\subset\mathcal L$. It follows that $x\in\mathfrak Z(\mathcal L_0)$ and since $\mathfrak Z(\mathcal L_0)=\{0\}$, we conclude that $x=0$.

Now we show that the decomposition $\mathcal L=\mathfrak Z(W)\oplus[\mathcal L,\mathcal L]$ is orthogonal with respect to the inner product $(.\,,.)_{\mathcal L}$. From
$$
\mathcal L=\mathfrak Z(W)\oplus\mathcal L_0=\mathfrak Z(\mathcal L)\oplus\mathcal L_0
$$
we deduce that $[\mathcal L,\mathcal L]=[\mathcal L_0,\mathcal L_0]=\mathcal L_0$, because the Lie algebra $\mathcal L_0$ has a trivial centre. It is also clear that $\mathcal L_0=[\mathcal L,\mathcal L]$ is an ideal of $\mathcal L$. Thus, the decomposition $\mathcal L=\mathfrak Z(W)\oplus[\mathcal L,\mathcal L]$ will be orthogonal with respect to any $\ad$-invariant inner product 
by item 1 of Proposition~\ref{prop:WL}.

{\sc Proof 2.} 
If $\mathfrak Z(W)=0$, then there is nothing to prove. If $\mathfrak Z(W)\neq 0$, then the orthogonal complement $W_1$ to $\mathfrak Z(W)$ in $W$ with respect to the inner product is a Lie triple system by item 1 of Proposition~\ref{prop:red}. We only need to show that $\mathcal L_0=[\mathcal L,\mathcal L]=W_1+[W_1,W_1]$. Denote $\mathcal L_0^*=W_1+[W_1,W_1]$. Since $W=\mathfrak Z(W)\oplus W_1$, we have 
$$
\mathcal L=W+[W,W]=\mathfrak Z(W)+ W_1+[W_1,W_1]=\mathfrak Z(W)+\mathcal L_0^*.
$$

Claim 1: {\it $\mathcal L=\mathfrak Z(W)\oplus_{\bot}\mathcal L_0^*$.} Since $[\mathfrak Z(W),W_1]=0$, we obtain
$$
(\mathfrak Z(W),[W_1,W_1])_{\mathcal L}=(W_1,[W_1,\mathfrak Z(W)])_{\mathcal L}=0.
$$
Together with $(\mathfrak Z(W),W_1)_{\mathcal L}$ the latter equalities imply  claim 1.

Claim 2: {\it $[\mathfrak Z(W),\mathcal L_0^*]=0$.} This follows from the equality $[\mathfrak Z(W),W_1]=0$ and from the Jacobi identity. 

Now the chain of inclusions
$$
\mathcal L_0=[\mathcal L,\mathcal L]=[(\mathfrak Z(W)+\mathcal L_0^*),(\mathfrak Z(W)+\mathcal L_0^*)]=[\mathcal L_0^*,\mathcal L_0^*]\subseteq \mathcal L_0^*\quad\Longrightarrow\quad \mathcal L_0\subseteq \mathcal L_0^*
$$
follows from item 1 of Proposition~\ref{prop:main}, from claim 2, and from the fact that $\mathcal L_0^*$ is a Lie algebra constructed from a Lie triple system, see Proposition~\ref{prop:use1}. Finally, 
we conclude that
$$
\mathfrak Z(W)\oplus_{\bot} \mathcal L_0^*=\mathcal L=\mathfrak Z(W)\oplus_{\bot} \mathcal L_0\subseteq \mathfrak Z(W)\oplus_{\bot} \mathcal L_0^*
$$
by making use of claim 1. This implies $\mathcal L_0^*=\mathcal L_0$, which finishes the proof.
\end{proof}

\begin{definition}
We say that a Lie algebra $\mathfrak g$ is reductive if for each ideal $\mathfrak a$ in $\mathfrak g$, there is an ideal $\mathfrak b$ in $\mathfrak g$ with $\mathfrak g=\mathfrak a\oplus\mathfrak b$.
\end{definition}
Let us recall the following statement: {\it a Lie algebra $\mathfrak g$ is semisimple, if and only if, $\mathfrak g=\mathfrak a_1\oplus\ldots\oplus\mathfrak a_j$ where $\mathfrak a_j$ are ideals,  each of which is a simple Lie algebra. In this case the decomposition is unique, and the only ideals of $\mathfrak g$ are the sum of various
$\mathfrak a_j$}, see~\cite[Theorem 1.54]{Knapp}. So if a Lie algebra $\mathfrak g$ is a direct sum of a semisimple Lie algebra  and an abelian Lie algebra, then $\mathfrak g$ is reductive. The following proposition shows that there are no other reductive Lie algebras.

\begin{proposition}\cite[Corollary 1.56]{Knapp}
If $\mathfrak g$ is reductive, then $\mathfrak g=\mathfrak a\oplus [\mathfrak g,\mathfrak g]$ with $[\mathfrak g,\mathfrak g]$ semisimple and $\mathfrak a$ abelian.
\end{proposition}

An important example of reductive Lie algebras is given in the following statement.

\begin{proposition}\cite[Proposition 1.59]{Knapp}
Let $\mathfrak g$ be a real Lie algebra of matrices over $\mathbb R$, $\mathbb C$ or $\mathbb H$, which is closed under the operation of conjugate transpose, then $\mathfrak g$ is reductive.
\end{proposition}

\begin{corollary}
If $W$ is a Lie triple system of $\so(m)$, then the Lie algebra $\mathcal L=W+[W,W]$ is reductive.
\end{corollary}
\begin{proof}
Since $C^{\mathbf t}=-C$ for any $C\in\so(m)$, we conclude that $C\in\mathcal L$ implies $C^{\mathbf t}=-C\in \mathcal L$, and therefore, the Lie algebra $\mathcal L$ is reductive.
\end{proof}


\subsection{Lie triple system of $\so(p,q)$ and representations of Clifford algebras}


We start from an example of the Lie triple system of $\so(l,l)$ related to the representation of Clifford algebras $\Cl_{r,s}$.
Let us recall Example~\ref{ex:Cl}, where the subspace $W=J(\mathbb R^{r,s})\subset\so(l,l)$ was defined by the Clifford algebra representation $J\colon \Cl_{r,s}\to\End(\mathbb R^{l,l})$. The case $s=0$ was studied in~\cite{Eber03}.

\begin{proposition}\label{prop:tr_Cl}
The space $W$ is a Lie triple system of $\so(l,l)$ with a trivial centre. 
\end{proposition}
\begin{proof}
First we show that the vector space $W$ is a Lie triple system. For any $X_1, X_2, X_3 \in W$, with $X_i= \sum_{j=1}^{r+s}\lambda_{ij} J_{Z_j}$, $\lambda_{ij} \in \mathbb{R}$, where $\{Z_1, \dotso, Z_{r+s} \}$ is an orthonormal basis of $\mathbb R^{r,s}$, it follows that 
\begin{eqnarray}
[X_1, [ X_2, X_3]]= \sum_{j,k,l=1}^{r+s}\lambda_{1j} \lambda_{2k} \lambda_{3l} [J_{Z_j}, [ J_{Z_k} , J_{Z_l} ]].
\end{eqnarray}
	If we prove that $[J_{Z_j}, [ J_{Z_k}, J_{Z_l}]] \in W$ for all $j,k,l \in \{1, \dotso, r+s\}$, then it will follow that $[X_1, [ X_2, X_3]] \in W$. We recall that $J_{Z_j}J_{Z_k}=-J_{Z_k}J_{Z_j}$ for all $j \not = k$. If all indices $j,k,l$ are different, then
\begin{eqnarray*}
[J_{Z_j}, [ J_{Z_k}, J_{Z_l}]] &=& [J_{Z_j}, J_{Z_k}J_{Z_l}] - [J_{Z_j}, J_{Z_l}J_{Z_k}] 
= J_{Z_j}J_{Z_k}J_{Z_l}- J_{Z_k} J_{Z_l} J_{Z_j} 
\\ &-& J_{Z_j} J_{Z_l}J_{Z_k} + J_{Z_l} J_{Z_k} J_{Z_j}
=
0 \in W.
\end{eqnarray*} 
If $j=k$, then
$
[J_{Z_j}, [ J_{Z_j}, J_{Z_l} ]] = -4\langle Z_j, Z_j \rangle_{r,s} J_{Z_l} \in W$.   
If $k=l$ or $j=k=l$, then $[J_{Z_j}, [ J_{Z_k}, J_{Z_k}]]=0 \in W$. So we conclude that $W=J(\mathbb R^{r,s})$ is a Lie triple system. 

Let us show that the centre of $W$ defined by~(\ref{eq:centre_triple})
is trivial.
For any $Z,Z^{'} \in \mathbb R^{r,s}$ we obtain
\begin{equation*}
[ J_{Z}, J_{Z'} ] = J_{Z}J_{Z'} - J_{Z'}J_{Z} = 
\begin{cases} 
2 J_{Z}J_{Z'} & \text{ if } \langle Z, Z' \rangle_{r,s} =0, 
\\
 -2J_{Z'} J_{Z} + \langle Z, Z' \rangle_{r,s} \Id_{V}, & \text{ if  } \langle Z, Z' \rangle_{r,s} \not =0. 
 \end{cases}
\end{equation*} 
Let us assume that the centre $\mathfrak Z(W)$ is non-trivial and that there is  $Z \in \mathbb R^{r,s}$, $Z \not =0$, such that $[J_{Z}, J_{Z'}] =0$ for all $Z' \in \mathbb R^{r,s}$. There are two possible cases $\langle Z, Z \rangle_{r,s} \not = 0$, and $\langle Z, Z \rangle_{r,s} = 0$.

{\it Case $\langle Z, Z \rangle_{r,s} \not = 0$}.  Then $J_{Z}^2=-\langle Z, Z \rangle_{r,s} \Id_V$ implies that $J_{Z}$ is invertible. The orthogonal complement to $\spn\{Z\}$ is a non-degenerate scalar product space, and there is $Z' \in (\spn\{Z\})^{\bot}$, such that $\langle Z',Z' \rangle_{r,s} \not =0 $. Then $J_{Z'}$ is also invertible and so is $J_{Z}J_{Z'}$, that  yields $J_{Z}J_{Z'} \not =0$. It follows that $[J_{Z}, J_{Z'}] = 2 J_{Z}J_{Z'} \not =0$, which is a contradiction to the assumption that $J_{Z} \in \mathfrak Z(W)$ with $Z \not =0$. 

{\it Case $\langle Z,Z \rangle_{r,s} = 0$}. First we notice that $J_{Z}$ can not be invertible since $J_{Z}^2=0$. Let $Z'$ be an element of $\mathbb R^{r,s}$ such that $\langle Z,Z' \rangle_{r,s} \not =0$, which exists because $\langle .\,, . \rangle_{r,s}$ is non-degenerate. Then, since $J_{Z}\in\mathfrak Z(W)$, we obtain
$$[J_{Z}, J_{Z'}]= -2J_{Z'}J_{Z} + \langle Z, Z' \rangle_{r,s} \Id_{V} =0,$$
which is equivalent to $ J_{Z'}J_{Z}=2\langle Z,Z'\rangle_{r,s}\Id_{V}$. But this implies that $J_{Z}$ is invertible with the inverse $(2\langle Z,Z'\rangle_{r,s})^{-1}J_{Z'}$. We again come to the contradiction.
\end{proof}
\begin{proposition}\label{cor:tr_Cl}
Let $W$ be a Lie triple system of $\so(l,l)$ defined by a representation of the Clifford algebra. Then 
$
\mathcal L=W+[W,W]=[\mathcal L,\mathcal L]$.
\end{proposition} 
\begin{proof}
It was shown in Proposition~\ref{prop:tr_Cl} that the centre of $W$ is trivial. Then applying item~1 of Proposition~\ref{prop:main} we finish the proof.
\end{proof}

Working with a subalgebra $\mathcal L$ of $\so(p,q)$ we use the following definition of the transpose: $D^{\mathbf t}=-\eta_{p,q} D\eta_{p,q}$, $\eta_{p,q}=\diag(I_p,-I_q)$. It is not true in general, that  $D\in \mathcal L$ implies $D^{\mathbf t}\in \mathcal L$. 
Any vector subspace $\mathcal C\subset\so(m)$ is closed under  transposition, because  $C\in \mathcal C$ implies $C^{\mathbf t}=-C\in \mathcal C$. It is not generally true for vector subspaces of $\so(p,q)$. They are only closed under the $\eta_{p,q}$-transposition: $\mathcal D^{\eta_{p,q}}=\eta_{p,q}\mathcal D^{\mathbf t}\eta_{p,q}=-\mathcal D$. 

\begin{proposition}\label{prop:D}
Let $\mathcal C\subset\orth(m)$ and $\eta_{p,q}=\diag(I_p,-I_q)$. Define 
$$
\mathcal D_1=\mathcal C\eta_{p,q}=\{C\eta_{p,q}\mid\ C\in \mathcal C\}\subset\so(p,q),
$$
$$
\mathcal D_2=\eta_{p,q}\mathcal C=\{\eta_{p,q} C\mid\ C\in \mathcal C\}\subset\so(p,q).
$$ Then, if the indefinite scalar product $\langle.\,,.\rangle_{\so(p,q)}$ is non-degenerate on $\mathcal D_1$, then it is non-degenerate on $\mathcal D_2$ and on $\mathcal D_1+\mathcal D_2$. Moreover, the space $\mathcal D_1+\mathcal D_2$ is invariant under transposition and involution
$$
\begin{array}{llcccc}
&\theta\colon &\so(p,q)&\to &\so(p,q)
\\
&&X&\mapsto &\eta_{p,q} X\eta_{p,q} 
\end{array}
$$
\end{proposition}

\begin{proof} We can show that the vectors $D_i=\eta_{p,q}C_i\in \mathcal D_2$, are linearly independent if the vectors $C_i\in \mathcal C$ are linearly independent by the same arguments as in Lemma~\ref{lem:lin_indep}.
Observe that the equalities
$$
\theta(\mathcal D_1)=\eta_{p,q}\mathcal D_1\eta_{p,q}=\eta_{p,q}\mathcal C\eta_{p,q}^2=\eta_{p,q}\mathcal C=\mathcal D_2
$$
imply $\mathcal D_1^{\mathbf t}=-\theta(\mathcal D_1)=-\mathcal D_2$. The space $\mathcal D_1+\mathcal D_2$ is invariant under  the transposition and  involution $\theta$, because 
$$
(\mathcal D_1+\mathcal D_2)^{\mathbf t}=-(\mathcal D_1+\mathcal D_2),\qquad
\theta(\mathcal D_1+\mathcal D_2)=\mathcal D_1+\mathcal D_2.
$$

If the metric $\langle .\,,.\rangle_{\so(p,q)}$ is non-degenerate on $\mathcal D_1$, then for any $X\in\mathcal D_1$, there is $Y\in\mathcal D_1$, such that 
\begin{equation*}
\langle X,Y\rangle_{\so(p,q)}=-\tr(XY)\neq 0.
\end{equation*}
Then,
$$
\langle \eta_{p,q} X\eta_{p,q},\eta_{p,q} Y\eta_{p,q}\rangle_{\so(p,q)}=-\tr(\eta_{p,q} XY\eta_{p,q})=-\tr(XY)\neq 0
$$
and $\langle .\,,.\rangle_{\so(p,q)}$ is non-degenerate on $\mathcal D_2$.
\end{proof}

\begin{corollary}
Under the assumptions of Proposition~\ref{prop:D}, the subspaces $\mathcal D_1$ and $\mathcal D_2$ are isometric.
\end{corollary}

\begin{proof} 
Since $\theta(\mathcal D_1)=\mathcal D_2$, then
$$
-\tr(DD')=-\tr(\eta_{p,q} D \eta_{p,q} \eta_{p,q} D'\eta_{p,q})=-\tr(\theta(D)\theta(D')),
$$
 for $D,D'\in\mathcal D_1$.
\end{proof}

The Lie triple systems $W$ associated with a representation of a Clifford algebra form simple or semisimple  subalgebras $\mathcal L=W+[W,W]$ of $\so(l,l)$. Before we formulate a precise statement and a proof we show the following lemma.

\begin{lemma}\label{lem:Lbasis}
Let a Lie triple system $W$ be associated with a representation of a Clifford algebra $J\colon \Cl_{r,s}\to\so(l,l)$. Then the Lie algebra $\mathcal L=W+[W,W]$ is generated by the basis 
$$
\{J_{Z_i}, J_{Z_j}J_{Z_k},\ i,j,k,=1,\ldots,r+s, j<k\},
$$
where $Z_1,\ldots,Z_{r+s}$ is an orthonormal basis of $\mathbb R^{r,s}.$
\end{lemma}
\begin{proof}
Recall that the representations of the Clifford algebra $\Cl_{r,s}$ satisfy the relations
$$
J_{Z}J_{Z'}+J_{Z'}J_{Z}=-2\la Z,Z'\ra_{r,s}\Id_{\mathbb R^{l,l}}, \quad Z,Z'\in\mathbb R^{r,s}.
$$
Let $\{Z_1,\ldots,Z_{n}\}$, $n=r+s$, be an orthonormal basis of $\mathbb R^{r,s}$. Then the following commutation relations 
$$
[J_{Z_i},J_{Z_j}]=2J_{Z_i}J_{Z_j}, \quad [J_{Z_i},[J_{Z_i},J_{Z_j}]]=-4\la Z_i,Z_i\ra_{r,s}J_{Z_j},\quad [J_{Z_i},[J_{Z_j},J_{Z_k}]]=0,
$$
hold.
Thus, the Lie algebra 
$
\mathcal L=W+[W,W]$ is generated by the set $\{J_{Z_k},J_{Z_i}J_{Z_j}\}$, $i,j,k=1,\ldots,n=r+s$.

In order to show that $\{J_{Z_k},J_{Z_i}J_{Z_j}\}$ for $ i,j,k=1,\ldots,n$ with $i<j$ is a basis, we proceed by induction. Recall that $J_{Z_1}, \dotso,J_{Z_n}$ are orthogonal to each other, hence linear independent. The orthogonality we understand in the following sense
$$
\langle J_{Z_i}v,J_{Z_j}v\rangle_{l,l}=\langle v,v\rangle_{l,l}\langle Z_i,Z_j\rangle_{r,s}=0\quad\text{for any}\quad v\in \mathbb R^{l,l},\ \ i,j=1,\ldots,n=r+s,
$$
see discussions in Section~\ref{subseq:pseudoH}.

Let $r+s=2$, then we have  
$$
\langle J_{Z_1}v,J_{Z_1}J_{Z_2}v\rangle_{l,l}=\langle v,J_{Z_2}v\rangle_{l,l}\langle Z_1,Z_1\rangle_{r,s}=0\quad\text{for any}\quad v\in \mathbb R^{l,l}.
$$
Analogously $\langle J_{Z_2}v,J_{Z_1}J_{Z_2}v\rangle_{l,l}=0$ for any $v\in \mathbb R^{l,l}$. We conclude that
$J_{Z_1},J_{Z_2}$ are orthogonal to $J_{Z_1}J_{Z_2}$ and hence $\{J_{Z_1},J_{Z_2},J_{Z_1}J_{Z_2}\}$ is a linear independent system. 

Let $n=r+s \geq 3$. For the induction step we assume that we are given a set of linearly independent operators $\{J_{Z_k},J_{Z_i}J_{Z_j}\}$ for $ i,j,k=1,\ldots,d < n$ with $i<j$. We are adding one operator $J_{Z_{d+1}}\not=0$ with $\la Z_{d+1} \,, Z_{d+1} \ra = \pm 1$ to the set and prove that this is still a set of linearly independent operators. By contradiction, assume that there exist $\lambda_1, \dotso, \lambda_{d}, \mu_{1,2}, \dotso, \mu_{d-1,d} \in \mathbb{R}$ such that
\begin{equation}\label{eq:d1}
J_{Z_{d+1}}=\sum_{k=1}^{d} \lambda_k J_{Z_k} + \sum_{1\leq i<j\leq d} \mu_{i,j}J_{Z_i}J_{Z_j}.
\end{equation}
We calculate
\begin{eqnarray*}
0&=&[J_{Z_{d+1}} \,, J_{Z_{d+1}}]=2 \left (\sum_{k=1}^{d} \lambda_k J_{Z_k} \right ) J_{Z_{d+1}} ,
\end{eqnarray*}
and obtain that $\sum_{k=1}^{d} \lambda_k J_{Z_k}=0$ as $J_{Z_{d+1}}$ is invertible. It follows that $\lambda_1=\dotso=\lambda_{d}=0$ by the induction assumption. Substituting the values of $\lambda_k$ into~\eqref{eq:d1} we obtain 
$$J_{Z_{d+1}}=\sum_{1\leq i<j\leq d} \mu_{i,j} J_{Z_i}J_{Z_j}.$$
We choose now any pair of indices $l,m\in \{1, \dotso,d\}$ such that $l<m$ and calculate
\begin{eqnarray*}
0=[J_{Z_{d+1}} \,, J_{Z_{l}}J_{Z_m}]&=& 2\la Z_l \,, Z_l\ra \sum_{k \in \{1, \dotso, d\} \setminus \{l\}} \alpha_{k} \mu_{l,k} J_{Z_m}J_{Z_k} 
\\&+& 2\la Z_m \,, Z_m \ra \sum_{s=\{1, \dotso, d\} \setminus \{m\}} \beta_{s} \mu_{m,s} J_{Z_l}J_{Z_s},
\end{eqnarray*}
with $\alpha_k, \beta_s =\pm 1$. This implies that $\mu_{l,m}=0$ for all $l,m \in \{1, \dotso,d\}$, $l<m$, since the operators $\{J_{Z_i}J_{Z_j}, i,j \in \{1, \dotso,d\}, i<j\}$ are linearly independent by assumption of the induction. But then $J_{Z_{d+1}}=0$ in~\eqref{eq:d1}, which yields the contradiction. Thus we conclude that the operators $\{J_{Z_k},J_{Z_{d+1}},J_{Z_i}J_{Z_j}\}$ for $ i,j,k=1,\ldots,d$, $i<j$, are linearly independent. By this method we can add any operator of the form $J_{Z_{q}}$, $q=d+1,\ldots, n$, with $\la Z_q \,, Z_q \ra = \pm 1$, and obtain a set of linearly independent operators.

Now we assume that we are given a set 
$$\left \{J_{Z_k},J_{Z_{d+1}},J_{Z_i}J_{Z_j} \mid\  k,i,j \in \{1, \dotso,d\},\ i<j \right \}$$
of linearly independent elements and prove that adding an element of the form $J_{Z_t}J_{Z_{d+1}}$ with a fixed $t \in \{1, \dotso,d\}$, we obtain a new set 
$$\left \{J_{Z_k},J_{Z_{d+1}},J_{Z_i}J_{Z_j},J_{Z_t}J_{Z_{d+1}} \mid\  k,i,j \in \{1, \dotso,d\},\ i<j \right \}$$
of linearly independent operators. 
Assume that there exist $\lambda_1, \dotso, \lambda_{d+1}, \mu_{1,2}, \dotso, \mu_{d-1,d} \in \mathbb{R}$ such that
\begin{equation}\label{eq:dd1}
J_{Z_{t}}J_{Z_{d+1}}=\sum_{k=1}^{d+1} \lambda_k J_{Z_k} + \sum_{1\leq i<j\leq d} \mu_{i,j}J_{Z_i}J_{Z_j}.
\end{equation}
We calculate 
\begin{eqnarray*}
0=[J_{Z_{t}}J_{Z_{d+1}} \,, J_{Z_s}]&=&2\sum_{k \in \{1, \dotso,d+1\} \setminus \{s\}} \lambda_k J_{Z_k}J_{Z_s} - 2\sum_{i \in \{1, \dotso,s-1\}} \mu_{i,s} \la Z_s \,, Z_s \ra J_{Z_i} \\&+& 2\sum_{i \in \{s+1, \dotso,d\}} \mu_{i,s} \la Z_s \,, Z_s \ra J_{Z_i}
\end{eqnarray*}
for any $s \in \{1, \dotso,d\}\setminus\{t\}$ and arrive at
\begin{eqnarray}\label{lambdan+1}
-\lambda_{d+1} J_{Z_{d+1}}J_{Z_s}
&=& 
\sum_{k \in \{1, \dotso,d\} \setminus \{s\}} \lambda_k J_{Z_k}J_{Z_s} - \sum_{i \in \{1, \dotso,s-1\}} \mu_{i,s} \la Z_s, Z_s \ra J_{Z_i}\nonumber
\\
&+& 
\sum_{i \in \{s+1, \dotso,d\}} \mu_{i,s} \la Z_s \,, Z_s \ra J_{Z_i}.
\end{eqnarray}
If $\lambda_{d+1}=0$, then it follows that $\lambda_k=\mu_{i,s}=0$ for all $k \in \{1, \dotso,d\} \setminus \{s\}$, $i \in \{1, \dotso,d\} \setminus \{s\}$ by the induction assumption. Since $s$ was chosen arbitrarily we can continue the proof and assume that $\lambda_{d+1} \not=0$. Then, for any $a \in \{1, \dotso, d\} \setminus \{s\}$
\begin{eqnarray*}
0&=&
-\lambda_{d+1} [J_{Z_{d+1}}J_{Z_s}, J_{Z_{a}}]
=
-2\lambda_a \la Z_a, Z_a \ra J_{Z_s} 
\\
&- & 2\sum_{i \in \{1, \dotso,s-1\} \setminus \{a\}} \mu_{i,s} \la Z_s, Z_s \ra J_{Z_i}J_{Z_a}
+ 
2\sum_{i \in \{s+1, \dotso,d\}\setminus \{a\}} \mu_{i,s} \la Z_s, Z_s \ra J_{Z_i}J_{Z_a}.
\end{eqnarray*}
As $\{J_{Z_s}, J_{Z_i}J_{Z_a} \vert i \in \{1, \dotso,d\}\setminus \{a,s\} \}$ are linearly independent, it follows that $\lambda_a=\mu_{i,s}=0$ for all $s \in \{1, \dotso,d\}\setminus\{t\}$, for any choice of  $a \in \{1, \dotso,d\} \setminus \{s\}$, and for $i \in \{1, \dotso,d\}\setminus \{a,s\}$. Hence, adding the element $J_{Z_t}J_{Z_{d+1}}$, $t=1,\ldots,d+1$, we again obtain a linear independent set.
This implies that
$$\{J_{Z_k},J_{Z_i}J_{Z_j} \vert \quad i,j,k=1,\ldots,n, \quad i<j \}$$ 
is a basis of $\mathcal L$.
\end{proof}

\begin{theorem}\label{th:trans}
Let $J\colon\Cl_{r,s}\to\so(l,l)$ be a representation, and let $W=J(\mathbb R^{r,s})\subset\so(l,l)$. Then the Lie algebra $\mathcal L=W+[W,W]$ is simple if $(r,s) \not \in \{(3,0),(1,2)\}$, and it is semisimple if $(r,s) \in \{(3,0),(1,2)\}$.  
\end{theorem}

\begin{proof}
Let us assume that $\h\subset \mathcal L$ is an ideal: $[\h,\mathcal L]\subset\h$. 
We aim at showing that the only possible ideal is either trivial or the whole $\mathcal L$, unless $(r,s) \not \in \{(3,0),(1,2)\}$. In the last case we show that $\mathcal L$ is the direct sum of two ideals. 
\\

{\sc Case 1}. Let us suppose that $J_Z\in\h$, with $Z\neq 0$ and $\la Z,Z\ra_{r,s}\neq 0$. Then, normalising $Z$, we can assume that there exists an orthonormal basis $\{Z_1,\ldots,Z_{n}\}$ with $Z=Z_1$. So,
\begin{equation*}
\begin{array}{rcll}
\h\ni [J_{Z_1},J_{Z_j}] & = & 2J_{Z_1}J_{Z_j}, &\quad j=2,\ldots,n,
\\
\h\ni [J_{Z_1},[J_{Z_1},J_{Z_j}]]& = & -4 \la Z_1,Z_1\ra_{r,s}J_{Z_j} ,&\quad j=2,\ldots,n,
\\
\h\ni [J_{Z_j},J_{Z_i}] & = & 2J_{Z_j}J_{Z_i}, &\quad i,j=1,\ldots,n, \ i\neq j.
\end{array}
\end{equation*}
We see that all  generators of $\mathcal L$ are contained in $\h$, which implies that $\h=\mathcal L$.
\\
 
Let us assume now that $J_Z\in\h$, with $Z\neq 0$ and $\la Z,Z\ra_{r,s}= 0$. Choose an orthonormal basis $\{Z_1,\ldots,Z_{n}\}$, such that $Z=\sum_{j=1}^{n}\lambda_kZ_k$ with $\lambda_1\neq 0$.  Note that there is at least one more coefficient $\lambda _k\neq 0$. Then 
$$
\h\ni[J_Z,J_{Z_1}]=2\sum_{k=2}^{n}\lambda_kJ_{Z_k}J_{Z_1}=2J_YJ_{Z_1},
$$
where we set $Y=\sum_{k=2}^{n}\lambda_kJ_{Z_k}$. Note that $\la Y,Y \ra_{r,s}\not=0$, $\la Y,Z_1 \ra_{r,s}=0$. Then 
$$
\h\ni[J_{Y}J_{Z_1},J_{Z_1}]=-2\la Z_1, Z_1 \ra_{r,s}J_{Y}.
$$
So, we reduce the problem to the previous case, concluding that $\h=\mathcal L$.
\\

{\sc Case 2}. In this case, we assume that $h=\sum_{i<j}\lambda_{ij}J_{Z_i}J_{Z_j}\in\h$ and $\lambda_{12}\neq 0$, otherwise we can change the numeration of the basis.  Then
$$\h\ni[J_{Z_1},h]=2\langle Z_1,Z_1\rangle_{r,s}\sum_{j=2}^{r+s}\lambda_{1j}J_{Z_j}\neq 0,$$
because $[J_{Z_i},[J_{Z_j},J_{Z_k}]]=0$ for $i\neq j\neq k$. We apply now Case 1.
\\

{\sc Case 3}.
We assume now that $h \in \h$ is a linear combination of $J_{Z_k}$ and $J_{Z_i}J_{Z_j}$ for some $k,i,j=1,\ldots, r+s$. Consider three cases: $r+s= 2$, $r+s=3$, and $r+s\geq 4$.
\\

 Let $\mathbf {r+s=2}$. Let $\{Z_1,Z_2\}$ be an orthonormal basis for $\mathbb R^{r,s}$ such that $h=\lambda_1J_{Z_1}+\lambda_2J_{Z_2}+\lambda_3J_{Z_1}J_{Z_2}$, where at least $\lambda_1$ and $\lambda_3$ are different from zero. Then
 \begin{eqnarray*}
\h\ni [\lambda_1J_{Z_1}+\lambda_2J_{Z_2}+\lambda_3J_{Z_1}J_{Z_2}, J_{Z_1}]= 2\lambda_2 J_{Z_2}J_{Z_1}+2\lambda_3\langle Z_1,Z_1\rangle_{r,s} J_{Z_2}.
\end{eqnarray*}
If $\lambda_2=0$, then we apply the arguments of Case 1. If $\lambda_2\not=0$, then 
\begin{eqnarray*}
\h\ni[\lambda_3\langle Z_1,Z_1\rangle_{r,s}J_{Z_2}+\lambda_2J_{Z_2}J_{Z_1}, J_{Z_2}]=2\lambda_2\langle Z_2,Z_2\rangle_{r,s}J_{Z_1},
\end{eqnarray*}
and we again reduce the proof to the Case 1.
\\

Let $\mathbf {r+s\geq 4}$. Let $\h\ni h=\sum_{k=1}^{r+s}\lambda_kJ_{Z_k}+\sum_{i<j}\mu_{i,j}J_{Z_i}J_{Z_j}$, where the basis $\{Z_1,\ldots,Z_{r+s}\}$ is orthonormal and at least two coefficients do not non-vanish, say $\lambda_2\neq 0$ and $\mu_{1,2}\neq 0$.
Then  
\begin{equation}\label{eq:rs4}
\h\ni h_1=[h, J_{Z_1}] =\sum_{k=2}^{r+s}\lambda_kJ_{Z_k}J_{Z_1}+2\langle Z_1,Z_1\rangle_{r,s}\sum_{j\geq 2}\mu_{1,j}J_{Z_j},
\end{equation}
We have $h_1\neq 0$ since otherwise it contradict the assumption $\lambda_2\neq 0$ and $\mu_{1,2}\neq 0$.
Taking the commutator with $J_{Z_2}$ we obtain
$$
\h  \ni  h_2=[h_1, J_{Z_2}]
= 2\lambda_2\langle Z_2,Z_2\rangle_{r,s}J_{Z_1}+4\langle Z_1,Z_1\rangle_{r,s}\sum_{j\geq 3}\mu_{1,j}J_{Z_j}J_{Z_2}.
$$
The vector $h_2\neq 0$ since $\lambda_2\neq 0$. We take the commutator with $J_{Z_3}$ and obtain
$$
\h \ni h_3=[h_2,J_{Z_3}]
= 4\lambda_2\langle Z_2,Z_2\rangle_{r,s}J_{Z_1}J_{Z_3}+8\langle Z_1,Z_1\rangle_{r,s}\langle Z_3,Z_3\rangle_{r,s}\mu_{1,3}J_{Z_2}.
$$
If we are still not in Case 2, we take the commutator with $J_{Z_4}$ and obtain
$$
\h \ni h_4=[h_3,J_{Z_4}]
= 16\langle Z_1,Z_1\rangle_{r,s}\langle Z_3,Z_3\rangle_{r,s}\mu_{1,3}J_{Z_2}J_{Z_4}.
$$
Thus the proof reduces to Case~2.
\\

Let $\mathbf {r+s=3}$. We start as in the previous case and obtain the vector 
$$
\h \ni h_3=[h_2,J_{Z_3}]
= 4\lambda_2\langle Z_2,Z_2\rangle_{r,s}J_{Z_1}J_{Z_3}+8\langle Z_1,Z_1\rangle_{r,s}\langle Z_3,Z_3\rangle_{r,s}\mu_{1,3}J_{Z_2}.
$$
Since we do not have an element $J_{Z_4}$, we can only take the commutators with $J_{Z_k}$ or $J_{Z_i}J_{Z_j}$, $k,i,j=1,2,3$. Anyway we are able to produce either zero vectors or an element of the same type as $h_3$, namely a linear combination of $J_{Z_k}$ and $J_{Z_i}J_{Z_j}$ for $i\neq j\neq k$, $i,j,k=1,2,3$.

Thus, without loss of generality, we can assume that the ideal $\h$ of $\mathcal L$ contains an element $h=J_{Z_1}+\lambda J_{Z_2}J_{Z_3}$, $\lambda\neq 0$. We calculate
\begin{equation*}
\begin{split}
& [h,J_{Z_1}]=0,
\\
&h_1= [h,J_{Z_2}]=2J_{Z_1}J_{Z_2}+2\lambda \langle Z_2,Z_2\rangle_{r,s}J_{Z_3},
\\
& h_2=[h,J_{Z_3}]=2J_{Z_1}J_{Z_3}-2\lambda \langle Z_3,Z_3\rangle_{r,s}J_{Z_2},
\\
& h_3= [h,J_{Z_1}J_{Z_2}]=-2\langle Z_1,Z_1\rangle_{r,s}J_{Z_2}+2\lambda \langle Z_2,Z_2\rangle_{r,s}J_{Z_1}J_{Z_3},
\\
& h_4= [h,J_{Z_1}J_{Z_3}]=-2\langle Z_1,Z_1\rangle_{r,s}J_{Z_3}-2\lambda \langle Z_3,Z_3\rangle_{r,s}J_{Z_1}J_{Z_2},
\\
& [h,J_{Z_2}J_{Z_3}]=0.
\end{split}
\end{equation*}
If $h_1$ and $h_4$ are linearly independent, then their span in $\h$ contains $J_{Z_3}$ and $J_{Z_1}J_{Z_2}$ and we continue the proof as in Cases 1 or 2. The same arguments are applied  when $h_2$ and $h_3$ are linearly independent. 

We assume that neither $h_1,h_4$ nor $h_2,h_3$ form a linear independent pair of vectors. Since the basis $\{Z_1,Z_2,Z_3\}$ is orthonormal, the vectors $h_1,h_4$ can be linearly dependent only if $\lambda=\pm 1$. To distinguish the values of the vectors, we write the superscript $^+$ for the case $\lambda=1$ and the superscript $^-$ for the case $\lambda=-1$. 

Assume now that $\lambda=1$. We write $h=h^+=J_{Z_1}+J_{Z_2}J_{Z_3}$ and obtain
$$
h_{1}^+=2(\langle Z_2,Z_2\rangle_{r,s}J_{Z_3}+J_{Z_1}J_{Z_2}),
\quad h_4^+=2(-\langle Z_1,Z_1\rangle_{r,s}J_{Z_3}-\langle Z_3,Z_3\rangle_{r,s}J_{Z_1}J_{Z_2})
$$
It suffices to consider the following different cases.
If
\begin{equation}\label{eq:30}
\begin{split}
&\langle Z_1,Z_1\rangle_{r,s} = \langle Z_2,Z_2\rangle_{r,s}=\langle Z_3,Z_3\rangle_{r,s}=1 \quad\text{and}\quad 
\\
&-\langle Z_1,Z_1\rangle_{r,s}  = \langle Z_2,Z_2\rangle_{r,s}=-\langle Z_3,Z_3\rangle_{r,s}=1,
\end{split}
\end{equation} 
then $h_1^+=-h_4^+=2(J_{Z_3}+J_{Z_1}J_{Z_2})$ or $h_1^+=h_4^+=2(J_{Z_3}+J_{Z_1}J_{Z_2})$, respectively,  and 
$$
\h\ni [h^+,h_1^+]=4\Big(2\langle Z_2,Z_2\rangle_{r,s}J_{Z_1}J_{Z_3}-(\langle Z_1,Z_1\rangle_{r,s}+\langle Z_2,Z_2\rangle_{r,s}\langle Z_3,Z_3\rangle_{r,s})J_{Z_2}\Big)
$$
for this choice of signatures of the scalar product $\langle.\,.\rangle_{r,s}$. We see that $\h\ni [h^+,h_1^+]=4h_2^+=4h_3^+$. 
In the cases 
\begin{equation}\label{eq:12}
\begin{split}
&\langle Z_1,Z_1\rangle_{r,s} = -\langle Z_2,Z_2\rangle_{r,s}=-\langle Z_3,Z_3\rangle_{r,s}=1\quad\text{and}\quad 
\\
&-\langle Z_1,Z_1\rangle_{r,s}  =-\langle Z_2,Z_2\rangle_{r,s}=\langle Z_3,Z_3\rangle_{r,s}=1,
\end{split}
\end{equation} 
we have $h_1^+=h_4^+=2(-J_{Z_3}+J_{Z_1}J_{Z_2})$ or $h_1^+=-h_4^+=2(-J_{Z_3}+J_{Z_1}J_{Z_2})$, respectively and $\h\ni [h^+,h_1^+]=-4h_2^+=4h_3^+$. 

Analogously, we consider possibility when $\lambda=-1$. We use the notation $h^-=J_{Z_1}-J_{Z_2}J_{Z_3}$ and obtain
$$
h_1^-=2(-\langle Z_2,Z_2\rangle_{r,s}J_{Z_3}+J_{Z_1}J_{Z_2}),
\quad h_4^-=2(-\langle Z_1,Z_1\rangle_{r,s}J_{Z_3}+\langle Z_3,Z_3\rangle_{r,s}J_{Z_1}J_{Z_2}).
$$
If the signature of the scalar product $\langle.\,.\rangle_{r,s}$ satisfies~\eqref{eq:30},
then $h_1^-=h_4^-=2(-J_{Z_3}+J_{Z_1}J_{Z_2})$ or $h_1^-=-h_4^-=2(-J_{Z_3}+J_{Z_1}J_{Z_2})$, respectively. Since
$$
\h\ni[h^-,h_1^-]=4\Big(-2\langle Z_2,Z_2\rangle_{r,s}J_{Z_1}J_{Z_3}+(\langle Z_1,Z_1\rangle_{r,s}-\langle Z_2,Z_2\rangle_{r,s}\langle Z_3,Z_3\rangle_{r,s})J_{Z_2}\Big),
$$
we obtain $\h\ni [h^+,h_1^-]=-4h_2^-=4h_3^-$. In the case~\eqref{eq:12} we have $h_1^-=-h_4^-=2(J_{Z_3}+J_{Z_1}J_{Z_2})$ or $h_1^-=h_4^-=2(J_{Z_3}+J_{Z_1}J_{Z_2})$, respectively and $\h\ni [h^+,h_1^-]=4h_2^-=4h_3^-$. 

Let us make the following conclusion. The cases of the signatures considered in~\eqref{eq:30} and~\eqref{eq:12} correspond to the cases $\mathbb R^{r,s}=\mathbb R^{3,0}$ and $\mathbb R^{r,s}=\mathbb R^{1,2}$. It is easy to see that in both cases the Lie algebra $\mathcal L=W+[W,W]$ is the direct sum of two ideals $\h_+=\spn\{h_+,h_1^+,h_2^+\}$ and $\h_-=\spn\{h_-,h_1^-,h_2^-\}$ and the Lie algebra $\mathcal L=W+[W,W]$ is semisimple.

\end{proof}

\begin{corollary}\label{cor:reductive}
If $W$ is a Lie triple system defined in Theorem~\ref{th:trans}, then $\mathcal L = W + [W, W]$ is reductive.
\end{corollary}
As a consequence we also obtain a new proof of Proposition~\ref{cor:tr_Cl}.


\subsection{Lie triple system $W$ is a rational subspace of $\mathcal L$}


The main aim of this section is to show that in the case of the trivial centre $\mathfrak Z(W)$ of the Lie triple system $W$ of $\so(p,q)$, the set $W$ is a rational subspace of $\mathcal L$.
Let us assume that $\mathcal L=W+[W,W]$ is reductive. So $\mathcal L=\mathfrak Z(\mathcal L)\oplus [\mathcal L,\mathcal L]$. Since $\mathfrak Z(\mathcal L)=\mathfrak Z(W)$ by Proposition~\ref{prop:main}, we obtain that $\mathcal L=[\mathcal L,\mathcal L]$ is a semisimple Lie algebra in the case  $\mathfrak Z(W)=0$. As we saw in Corollary~\ref{cor:reductive}, it is the case when the Lie triple system $W$ is defined by the Clifford algebra representation. Let us recall some definitions.

\begin{definition}
Let $\mathfrak g$ be a real semisimple Lie algebra, and let $B_{\mathfrak g}$ be its Killing form. An involution $\theta$ $(\theta^2=\Id_{\mathfrak g})$ is called a Cartan involution on $\mathfrak g$ if $C_{\theta}(X,Y):=-B_{\mathfrak g}(X,\theta(Y))$ is a positive definite bilinear form.
\end{definition}

As it was observed before, the bilinear form  on $\so(p,q)$
$$
\la X,Y\ra_{\so(p,q)}=\tr(X^{\eta_{p,q}}Y)=-\tr(XY),\quad X^{\eta_{p,q}}=\eta_{p,q} X^{\mathbf t}\eta_{p,q}=-X,
$$ 
is a (positive) scalar multiple of the Killing form $B_{\so(p,q)}$, because the Lie algebra $\so(p,q)$ is simple. Define the involution $\theta$ on $\so(p,q)$ by 
\begin{equation}\label{eq:Cartan_inv}
X\mapsto \theta(X)=\eta_{p,q} X \eta_{p,q}.
\end{equation} 
We claim that $\theta$ is the Cartan involution on $\so(p,q)$. Indeed, if $X\in \so(p,q)$ and $X\neq 0$, then
\begin{eqnarray*}
C_{\theta}(X,X)&:= &B_{\so(p,q)}(X,\theta(X))=c\,\la X,\theta(X)\ra_{\so(p,q)}=-c\,\tr(X\eta_{p,q} X\eta_{p,q})
\\
&= &-c\,\tr(X\eta_{p,q})^2>0,
\end{eqnarray*}
because if $X\in \so(p,q)$, then
$X\eta_{p,q}\in\so(m)$, $p+q=m$ with $\tr(X\eta_{p,q})^2<0$.

It is a well known fact that any complex semisimple Lie algebra $\g$ admits a basis in which the structure constants are integer. The real basis for the compact real form can be easily recovered, and the structure constants are half integers, see, for instance,~\cite{Chev,Knapp}. Recently an analogous result for an arbitrary real semisimple Lie algebra $\g$ was obtained in~\cite[Theorem 4.1]{HolgerKammeyer}. By making use of the Cartan involution $\theta$ and of the non-degenerate Killing form on $\g$, an explicit form of the real basis was recovered from the complexification of $\g$. We will not use the exact form of this basis, the only important fact for our purpose is that the structure constants are rational, more precisely, they belong to $\frac{1}{2}\mathbb Z$, for more details see~\cite{HolgerKammeyer}. We denote the basis by the letter $\mathcal C_{\mathcal L}$ and called it {\it Chevalley basis} referring to C.~Chevalley, who constructed this basis for real compact forms.

\begin{definition}\label{def:rs_lie}
Let $\mathfrak g$ be a Lie algebra such that  the Lie algebra $\mathfrak g$ has the rational structure constants with respect to a Chevalley basis $\mathcal C_{\mathfrak g}$. Then the set $\spn_{\mathbb Q}\{\mathcal C_{\mathfrak g}\}$ is called the rational structure of the Lie algebra $\mathfrak g$. A subspace $U$ of $\mathfrak g$ is called the rational subspace with respect to the rational structure $\spn_{\mathbb Q}\{\mathcal C_{\mathfrak g}\}$, if there is a basis $B_{U}$ such that $B_U\subset \spn_{\mathbb Q}\{\mathcal C_{\mathfrak g}\}$. 
\end{definition}

\begin{example}\label{prop:dir_sum}
Let $\mathcal L=\mathfrak p\oplus\mathfrak t$ be the Cartan decomposition of a semisimple Lie algebra. Then $\mathfrak p$ and $\mathfrak t$ are rational subspaces of $\mathcal L$ with respect to the rational structure $\spn_{\mathbb Q}\{\mathcal C_{\mathcal L}\}$ by~\cite{HolgerKammeyer}. The subspaces  $\mathfrak p$ and $\mathfrak t$ are contained in $\spn_{\frac{1}{2}\mathbb Z}\{\mathcal C_{\mathcal L}\}\subset\spn_{\mathbb Q}\{\mathcal C_{\mathcal L}\}$.
\end{example}
 
 \begin{proposition}\label{prop1:dir_sum}
Let $\mathcal L$ be a semisimple Lie subalgebra of $\so(p,q)$, such that $\mathcal L=\mathfrak p\oplus\mathfrak t$ is a semisimple Lie algebra, where $\mathfrak p$ and $\mathfrak t$ form the Cartan pair
$$
[\mathfrak t,\mathfrak t]\subseteq \mathfrak t,\quad [\mathfrak t,\mathfrak p]\subseteq \mathfrak p,\quad [\mathfrak p,\mathfrak p]\subseteq \mathfrak t.
$$
Then $\mathfrak p$ and $\mathfrak t$ are rational subspaces of $\mathcal L$ with respect to the rational structure $\spn_{\mathbb Q}\{\mathcal C_{\mathcal L}\}$. 
\end{proposition}

\begin{proof}
We define the involution by the rule 
$\theta (p+k)=-p+k$, for all $p\in \mathfrak p$, $k\in\mathfrak t$. It is an isomorphism for the real semisimple Lie algebra~$\mathcal L$. Construct the complex Chevalley basis on the complexification $\mathcal L\otimes \mathbb C$ of $\mathcal L$ by making use of the unique complex extension $\Theta$ of the involution $\theta$ and the conjugation on $\mathcal L\otimes \mathbb C$ with respect to the real form $\mathcal L$. Then applying the construction of~\cite{HolgerKammeyer} we recover the real basis $\mathcal C_{\mathcal L}$ for the real form $\mathcal L$. Hence $\mathfrak p$ and $\mathfrak t$ have the basis in $\spn_{\frac{1}{2}\mathbb Z}\{\mathcal C_{\mathcal L}\}\subset\spn_{\mathbb Q}\{\mathcal L\}$. 
 \end{proof}
 
\begin{corollary}\label{cor:dir_sum}
If $W$ is a Lie triple system of $\so(p,q)$ with a trivial centre, and if $\mathcal L=W\oplus[W,W]=[\mathcal L,\mathcal L]$ is a semisimple Lie algebra, then $W$ is a rational subspace of $\mathcal L$ with respect to the rational structure $\spn_{\mathbb Q}\{\mathcal C_{\mathcal L}\}$. 
\end{corollary}

\begin{proof}
If $W$ is a Lie triple system of $\so(p,q)$ with a trivial centre $\mathfrak Z(W)$, then $\mathcal L=[\mathcal L,\mathcal L]$ by item 1 of Proposition~\ref{prop:main}. The pair $W\oplus[W,W]=\mathfrak p\oplus\mathfrak t=\mathcal L$ is the Cartan pair because $W$ is a Lie triple system. Then we finish the proof by applying Proposition~\ref{prop1:dir_sum}.
 \end{proof}

Now let us assume that $\mathcal L_1=W\cap[W,W]\neq\{0\}$. We need to show that $W$ has a basis in the rational structure $\spn_{\mathbb Q}\{\mathcal C_{\mathcal L}\}$. 
Note that $\mathcal L_1$ is an ideal of $\mathcal L=[W,W]+W$ because $\mathcal L_1$ is $\ad_W$ and $\ad_{[W,W]}$ invariant, see the proof of item 3 of Proposition~\ref{prop:red}. Let $\mathcal L_2$ be the orthogonal complement to $\mathcal L_1$ with respect to any $\ad$-invariant inner product $(.\,,.)_{\mathcal L}$ on $\mathcal L$. Then $\mathcal L_2$ is also an ideal of $\mathcal L$. Indeed
$$
([X,\mathcal L_2],\mathcal L_1)_{\mathcal L}=-(\mathcal L_2,[X,\mathcal L_1])_{\mathcal L}\subseteq -(\mathcal L_2,\mathcal L_1)_{\mathcal L}=0.
$$
Thus, we have two ideals $\mathcal L_1$, $\mathcal L_2$ of $\mathcal L$, such that $\mathcal L_1\cap\mathcal L_2=\{0\}$. This implies that they are also orthogonal with respect to the Killing form $B_{\mathcal L}$, and therefore, assuming  $\mathcal L=[\mathcal L,\mathcal L]$ to be semisimple, we obtain that $B_{\mathcal L}$ is non-degenerate on both $\mathcal L_1$ and $\mathcal L_2$. (If it were degenerate on one of them, then it would be degenerate on the other one too, and then, it would be degenerate on $\mathcal L$, which is a contradiction.) Moreover, the restrictions of $B_{\mathcal L}$ on the ideals $\mathcal L_1$ and $\mathcal L_2$ define the Killing forms $B_{\mathcal L_1}$ and $B_{\mathcal L_2}$ of them.

\begin{proposition}
The Lie triple system $W$ of $\so(p,q)$ with a trivial centre is a rational subspace of a semisimple Lie algebra $\mathcal L=W+[W,W]=[\mathcal L,\mathcal L]$ with respect to the rational structure $\spn_{\mathbb Q}\{\mathcal C_{\mathcal L_1}\}\oplus \spn_{\mathbb Q}\{\mathcal C_{\mathcal L_2}\}$.
\end{proposition}

\begin{proof}
Let us denote $\mathcal L_1=W_1$, and let $W_2$ be the orthogonal complement to $W_1$ in $W$ with respect to the $\ad$-invariant inner product $(.\,,.)_{\mathcal L}$, which we have used before for the definition of $\mathcal L_2$. Let $V_2$ be the orthogonal complement to $W_1$ in $[W,W]$ with respect to the  same inner product $(.\,,.)_{\mathcal L}$. Then, 
$$
W=W_1\oplus W_2,\qquad [W,W]=W_1\oplus V_2.
$$
Obviously,  $W_2\cap V_2=\{0\}$ and $W_2\oplus V_2\subseteq \mathcal L_2=\mathcal L_1^{\bot}$. Thus, 
$$
\mathcal L_1\oplus\mathcal L_2=\mathcal L=W+[W,W]=(W_1\oplus W_2)+(W_1\oplus V_2)=W_1\oplus(W_2\oplus V_2)\subseteq \mathcal L_1\oplus\mathcal L_2
$$
and we conclude that $\mathcal L_2=W_2\oplus V_2$.

Observe also that $\mathcal L_1=W_1$ and $W_2$ are Lie triple systems of $\so(p,q)$, satisfying $[W_1,W_2]=0$, see the proof of item 3 of Proposition~\ref{prop:red}.
So, 
\begin{equation}\label{eq:Lie}
[V_2,W_2]\subset[[W,W],W_2]\subset W_2,\quad\text{and}\quad [V_2,W_1]\subset[[W,W],W_1]\subset W_1,
\end{equation}
$$
[W_2,W_2]\subset[W,W]=W_1\oplus V_2,\quad\text{and}\quad([W_2,W_2],W_1)_{\mathcal L}=-(W_2,[W_2,W_1])_{\mathcal L}=0,
$$
which implies that $[W_2,W_2]\subset V_2$. Finally, 
$$
[V_2,V_2]\subset[[W,W],[W,W]]\subset[W,W]=W_1\oplus V_2\quad\text{and}
$$
$$
([V_2,V_2],W_1)_{\mathcal L}=-(V_2,[V_2,W_1])_{\mathcal L}\subset-(V_2,W_1)_{\mathcal L}=0
$$ 
by~\eqref{eq:Lie}. Thus,
the semisimple Lie algebra $\mathcal L_2$ admits a decomposition $\mathcal L_2=W_2\oplus V_2$, where $W_2$ and $V_2$ form a Cartan pair.  Moreover, there is a Chevalley basis $\mathcal C_{\mathcal L_2}$ such that $W_2$ is a rational subspace of $\mathcal L_2$ with respect to $\spn_{\mathbb Q}\{\mathcal C_{\mathcal L_2}\}$ by Proposition~\ref{prop1:dir_sum}.

As a semisimple Lie algebra $\mathcal L_1$ admits the Chevalley basis $\mathcal C_{\mathcal L_1}$, the basis $\mathcal C=\mathcal C_{\mathcal L_1}\cup\mathcal C_{\mathcal L_2}$ is a Chevalley basis of the Lie algebra $\mathcal L=\mathcal L_1\oplus \mathcal L_2$. We define the rational structure of $\mathcal L$ by 
$\spn_{\mathbb Q}\{\mathcal C\}=\spn_{\mathbb Q}\{\mathcal C_{\mathcal L_1}\}\oplus \spn_{\mathbb Q}\{\mathcal C_{\mathcal L_2}\}$. Now $W=\mathcal L_1\oplus W_2$ is a rational subspace of $\mathcal L$ with respect to $\spn_{\mathbb Q}\{\mathcal C\}$.
\end{proof}

 

\section{2-step nilpotent Lie algebras admitting rational structure constants}\label{sec:int_str}


In this section we show that the standard pseudo-metric Lie algebra $\mathcal G=V\oplus W$ admits rational structure constants if $W$ is a Lie triple system of $\orth(V)$ being a rational subspace of the Lie algebra $\mathcal L=W+[W,W]\subset\orth(V)$. Another important result is that the Lie algebras $\mathfrak G$ from Definition~\ref{def:pseudo_induced} also admit rational structure constants under a special condition on the map inducing the standard pseudo-metric form on~$\mathfrak G$.  We start from some general properties of subspaces of $\End(V)$.


\subsection{More about rational structures}


Let us formulate a generalisation of Definition~\ref{def:rs_lie}. Let $V$ be an $m$-dimensional vector space, and let $W$ be a $k$-dimensional subspace of $\End(V)$.

\begin{definition}
A subspace $W$ is called the rational subspace of $\End(V)$ if there are bases $B_V=\{v_1,\ldots,v_m\}$ of $V$, and $B_W=\{\zeta_1,\ldots,\zeta_k\}$ of $W$, such that
$$
\zeta_j(\spn_\mathbb Q\{B_V\})\subseteq\spn_\mathbb Q\{B_V\}\quad\text{for all}\quad \zeta_j\in B_{W}.
$$ 
\end{definition}
Thus the basis elements of $W$ leave  the rational combinations of $B_V$ invariant. It is equivalent to say that any $\zeta_j\in B_{W}$ written as a matrix in the basis $B_V$ has rational entries.  

\begin{example}
Let $V=\mathbb R^{p,q}$ and $\la .\,,.\ra_{p,q}$ be the standard scalar product in $\mathbb R^{p,q}$. Let $A_1,\ldots, A_k\in\so(p,q)$ be arbitrary matrices with rational entries and $W=\spn_{\mathbb R}\{A_1,\ldots, A_k\}$. Then $W$ is a rational subspace of $\so(p,q)\subset\End(\mathbb R^m)$, $m=p+q$.
\end{example}

Let $\mathcal Z$ be an $n$-dimensional vector space, and let $J\colon \mathcal Z\to \End(V)$ be a linear map.

\begin{definition}
A map $J\colon \mathcal Z\to \End(V)$ is called rational if there are bases $B_V=\{v_1,\ldots,v_m\}$ of $V$ and $B_{\mathcal Z}=\{z_1,\ldots,z_n\}$ of $\mathcal Z$ such that
$$
J_{z_j}(\spn_\mathbb Q\{B_V\})\subseteq\spn_\mathbb Q\{B_V\}\quad\text{for all}\quad z_j\in B_{\mathcal Z}.
$$ 
\end{definition}
Thus, if the map $J\colon \mathcal Z\to \End(V)$ is rational, then the space $W=J(\mathcal Z)$ is a rational subspace of $\End(V)$ with respect to the bases
$$
B_V=\{v_1,\ldots,v_m\},\quad\text{and}\quad B_W=\{\zeta_1=J_{z_1},\ldots,\zeta_n=J_{z_n}\}.
$$
If moreover, the map $J\colon \mathcal Z\to \End(V)$ is injective and $W=J(\mathcal Z)$ is a rational subspace of $\End(V)$ with respect to the bases $B_V$ and $B_W$, then $J$ is a rational map with respect to the bases $B_V$ and $B_{\mathcal Z}=\{z_1,\ldots,z_n\}$ with $J_{z_i}=\zeta_i\in B_W$.

\begin{example}
Let $W$ be a rational subspace of $\so(p,q)\subset\End(\mathbb R^m)$, $p+q=m$. Then the inclusion $\iota\colon W\to\End(\mathbb R^m)$ is an injective and  rational linear map. Moreover, it is skew-symmetric in the following sense,
$$
\la \iota(\zeta)v,w\ra_{p,q}=\la \zeta(v),w\ra_{p,q}=-\la v,\zeta(w)\ra_{p,q}=-\la v,\iota(\zeta)w\ra_{p,q}.
$$

Let now $A\in\GL(\mathbb R^m)$, and let $W$ be a rational subspace of $\End(\mathbb R^{m})$. Then the inclusion map $\iota\colon AWA^{-1}\to\End(\mathbb R^m)$ is also injective and rational linear. To see that $\iota$ is rational we choose the bases $B_{\mathbb R^{m}}$ and $B_{W}$ such that the space $W$ is a rational subspace of $\End(\mathbb R^{m})$. Then the matrices $\zeta\in W$ written in the basis $B_{\mathbb R^{m}}$ have the same entries as matrices $A\zeta A^{-1}\in AWA^{-1}$ written in the basis $AB_{\mathbb R^{m}}$. So, all matrices from $AB_{W}A^{-1}$ has rational entries written in the basis $AB_{\mathbb R^{m}}$, and therefore, the space $AWA^{-1}$ is a rational subspace of $\End(\mathbb R^m)$ relatively to the bases $B_{AWA^{-1}}=AB_{W}A^{-1}$ and $AB_{\mathbb R^{m}}$. 

Now we want to show that any map from $AWA^{-1}$ is skew-symmetric with respect to some scalar product if $W\subset\so(p,q)$. Let us recall the notation $A^{\eta_{p,q}}=\eta_{p,q}A^{\mathbf t}\eta_{p,q}$. Then,
$$
(A^{-1})^{\eta_{p,q}}=\eta_{p,q}(A^{-1})^{\mathbf t}\eta_{p,q}=\eta_{p,q}(A^{\mathbf t})^{-1}\eta_{p,q}=\big(\eta_{p,q}(A^{\mathbf t})\eta_{p,q}\big)^{-1}=(A^{\eta_{p,q}})^{-1}.
$$ 
We define a matrix $M=(A^{\eta_{p,q}})^{-1}A^{-1}$, and recall the scalar product $\la v,w\ra_M=\la v,Mw\ra_{p,q}$. Then we obtain for any $\zeta\in W\subset\so(p,q)$, $A\in\GL(\mathbb R^m)$, $m=p+q$ and for arbitrary $\hat v,\hat w\in\spn_{\mathbb R}\{AB_{\mathbb R^{m}}\}$, that
$$
\begin{array}{llll}
&\la A\zeta A^{-1}\hat v,\hat w\ra_{M} 
=
\la A\zeta A^{-1} Av,MAw\ra_{p,q}=\la \zeta v,A^{\eta_{p,q}}MAw\ra_{p,q}=-\la v,\zeta w\ra_{p,q}
\\
&=-\la A^{-1}\hat v,A^{\eta_{p,q}}MA\zeta A^{-1}\hat w\ra_{p,q}
=-\la \hat v,(A^{\eta_{p,q}})^{-1}A^{\eta_{p,q}}MA\zeta A^{-1}\hat w\ra_{p,q}
\\
& =-\la \hat v,A\zeta A^{-1}\hat w\ra_{M}.
\end{array}
$$
\end{example}

\begin{definition}
A vector space $V$ is called $W$-irreducible for $W\subseteq\End(V)$ if no proper subspace of $V$ is invariant under all elements of $W$.
\end{definition}

Let $\la.\,,.\ra_V$ be a scalar product (non-degenerate bilinear form) on $V$. We recall that if a vector space $W$ is a subspace of $\orth(V,\la.\,,.\ra_V)\subset\End(V)$, then the scalar product $\la.\,,.\ra_V$ is called $W$-invariant, see Definition~\ref{def:Winv}. Notice an analogue with the invariant scalar product on Lie algebras, where it is equivalent to the fact that the adjoint map $\ad_v$ is skew-symmetric with respect to this scalar product. Having in mind Definition~\ref{def:pseudo_induced}, let us formulate the following statement.

\begin{theorem}\label{th:J_rat}
Let $\mathcal Z$ be a finite dimensional vector space, let $(V,\la.\,,.\ra_V)$ be a scalar product space, and let the map $J\colon \mathcal Z\to \orth(V,\la.\,,.\ra_V)\subset\End(V)$ be injective and rational. Let $V$ be $W$-irreducible for $W=J(\mathcal Z)$. Then the standard pseudo-metric Lie algebra $\mathfrak G=(\tilde V\oplus Z,[.\,,.])$ induced by $J$ admits a basis with rational structure constants. Here $\tilde V=(V,c\la.\,,.\ra_V)$, $c\neq0$. 
\end{theorem}
\begin{proof} We give the proof in several steps.

{\sc Step 1}. Let $(V,\la.\,,.\ra_V)$ be an $m$-dimensional scalar product space with a basis $B_V^{\mathbb Q}=\{v_1,\ldots,v_m\}$ such that $\upsilon_{ij}=\la v_i,v_j\ra_V\in\mathbb Q$, for instance we can take an orthonormal basis. Now we claim {\it if $v\in V$ is such that $\la v,v_i\ra_V\in\mathbb Q$ for all $v_i\in B_V^{\mathbb Q}$, then $v\in\spn_{\mathbb Q}\{B_V^{\mathbb Q}\}$.} Indeed, if we write $v=\sum_kx_k v_k$, then the linear system 
$$
y_i=\la v,v_i\ra_V=\sum_{k}x_k\la v_k,v_i\ra_V
$$
has rational coefficients $\la v_k,v_i\ra_V$ and $y_i\in\mathbb Q$. It is clear that the solutions $x_i$ are rational numbers. 
\\

{\sc Step 2}.
Consider now an arbitrary $k$-dimensional rational subspace $W\subset\End(V)$ with respect to bases
$$
B_V=\{v_1,\ldots,v_m\} \quad \text{and}\quad B_{W}=\{\zeta_1,\ldots,\zeta_k\}.
$$
Let us also assume that there is a scalar product $\la.\,,.\ra_V$, such that
$$
W\subseteq\orth(V,\la.\,,.\ra_V)\subset\End(V)\quad\text{and }\quad \la v_i,v_j\ra_V\in\mathbb Q, \ \ i,j=1,\ldots,m.
$$ 
Consider a Lie algebra $\mathcal G=(V\oplus W,[.\,,.]_{\mathcal G})$ with the Lie bracket defined by
$$
\la \zeta,[v,v']_{\mathcal G}\ra_{\orth(V,\la.\,,.\ra_V)}=\la\zeta(v),v'\ra_V.
$$
We claim: {\it the Lie algebra $\mathcal G$ has rational structure constants with respect to the basis $\{v_1,\ldots,v_m,\zeta_1,\ldots,\zeta_k\}$.} Indeed, $W\subset\End(V)$ is a rational subspace,  the matrices of all $\zeta_i\in B_{W}$ written in the basis $B_V$ have rational entries, and therefore, $\la \zeta_i,\zeta_j\ra_{\orth(V,\la.\,,.\ra_V)}=-\tr(\zeta_i\zeta_j)\in\mathbb Q$, $i,j=1,\ldots,k$. Moreover, since $\la v_{\alpha},v_{\beta}\ra_V\in\mathbb Q$, $\alpha,\beta=1,\ldots,m$,
 we also have $\la\zeta_i(v_{\alpha}),v_{\beta}\ra_V\in\mathbb Q$ for all $i=1,\ldots,k$, $\alpha,\beta=1,\ldots,m$. Thus,
$$
 \la \zeta_i,[v_{\alpha},v_{\beta}]_{\mathcal G}\ra_{\orth(V,\la.\,,.\ra_V)}=\la\zeta_i(v_{\alpha}),v_{\beta}\ra_V\in \mathbb Q,
$$
which implies $[v_{\alpha},v_{\beta}]_{\mathcal G}\in\spn_{\mathbb Q}\{B_{W}\}$ by Step~1. 

So it is left to show that we can modify a given scalar product $\la.\,,.\ra_V$ to a new one $\la.\,,.\ra_V^*$ such that all hypothesis of Theorem~\ref{th:J_rat} are still satisfied, and moreover, $\la v_{\alpha},v_{\beta}\ra_V^*\in \mathbb Q$. Then by Lemma~\ref{lem:2} we conclude that the Lie algebras $\mathcal G$ and $\mathfrak G$ are isomorphic, and therefore, $\mathfrak G$ admits rational structure constants. We still need some auxiliary results.
\\

{\sc Step 3}. {\it If there are two $W$-invariant scalar products $\la.\,,.\ra_V^1$ and $\la.\,,.\ra_V^2$ on $V$, 
and moreover, $V$ is $W$-irreducible, then $\la.\,,.\ra_V^1=c\la.\,,.\ra_V^2$ for some $c\neq 0$.} Indeed we define a map $S\colon V\to V$ by $\la v,w\ra_V^2=\la Sv,w\ra_V^1$, and the transformation $S$ is symmetric with respect to both scalar products and commutes with all elements of $W$ as it was shown in~\eqref{eq:symmetric} and~\eqref{eq:commute}. Thus, the elements of $W$ leave invariant the eigenspaces of $S$, and the irreducibility of $V$ implies that $S=c\,\Id_{V}$, $c\neq 0$.
\\

{\sc Step 4}.
Let us set our considerations in a more general perspective. Let us denote any bilinear symmetric form on a space $V$ by $b$, and the space of all possible bilinear symmetric forms on a space $V$ by $\mathcal B$. So $\mathcal B$ is a real vector space. We define the action of $\End(V)$ on $\mathcal B$ by 
$$
Xb(v,w)=b(Xv,w)+b(v,Xw),\quad X\in\End(V),\quad v,w\in V,\quad\text{and}\quad b\in\mathcal B.
$$
Thus, if $X\in\End(V)$ is skew-symmetric with respect to $b$, then $Xb=0$, and we say that $b$ is $X$-invariant. A symmetric bilinear $b$ is $W$-invariant if $Xb=0$ for all $X\in W\subset\End(V)$.

 Let $b_1=\la.\,,.\ra_V^1$ be a non-degenerate $W$-invariant bilinear symmetric form on $V$, and let $V$ be $W$-irreducible. Let $
\mathcal K=\{b\in \mathcal B,\mid\  b\ \ \text{is}\ \ W\text{-invariant}\}$. 
Then, $\mathcal K\neq\emptyset$, since $b_1\in \mathcal K$. The space $\spn_{\mathbb R}\{b_1\}$ belongs to $ \mathcal K$, because if $b_1$ is $W$-invariant, then $c b_1$ is $W$-invariant for any $c$. If we assume that there is $\tilde b\in \mathcal B$ linearly independent of $b_1$, then $\tilde b$ is not $W$-invariant by Step~3. Thus, we conclude that $\dim{\mathcal K}=1$. For the future purpose we only need the fact $\dim{\mathcal K}\geq 1$.
\\

{\sc Step 5}. Let $V$ be $W$-irreducible relatively to $W\subset\orth(V,\la.\,,.\ra_V)$, and let $W$ be a rational subspace of $\End(V)$ with respect to the bases
$$
B_V=\{v_1,\ldots,v_m\},\qquad B_W=\{\zeta_1,\ldots,\zeta_n\}.
$$
We claim: {\it there exists a constant $c\neq 0$ such that for $\la.\,,.\ra_V^*=c\la.\,,.\ra_V$ the inclusions
$$
W\in\orth(V,\la.\,,.\ra_V^*)\quad\text{and}\quad \la v_i,v_j\ra_V^*\in\mathbb Q,\ \ \text{for all}\ \ v_i\in B_{V},
$$
hold.
} To show this, we start by taking the dual $V^*$ of $V$ and choosing the basis $B_{V^*}=\{v_1^*,\ldots,v_m^*\}$, such that $v_i^*(v_j)=\delta_{ij}$. It allows us to choose the basis $\{b_{ij}\}_{1\leq i\leq j\leq m}$ of $\mathcal B$ by setting
$
b_{ij}=\frac{1}{2}(v_i^*\otimes v_j^*+v_j^*\otimes v_i^*)$.
Then 
$$
b_{ij}(v_{\alpha},v_{\beta})=\begin{cases}
1\quad&\text{if}\quad i=j=\alpha=\beta,
\\
\frac{1}{2}\quad&\text{if}\quad i=\alpha,\ j=\beta,\ i\neq j,
\\
0.\quad&\text{othewise}
\end{cases}
$$
We observe that since the action of $\End(V)$ on $\mathcal B$ is linear, we obtain
\begin{equation}\label{eq:basis_form}
\zeta_k\big(\spn_{\mathbb Q}\big\{ \{b_{ij}\}_{1\leq i\leq j\leq m}\big\}\big)\subseteq\spn_{\mathbb Q}\big\{\{b_{ij}\}_{1\leq i\leq j\leq m}\big\}\ \ \text{for all}\ \ \zeta_k\in B_W.
\end{equation}

Now we define the map 
$$
\begin{array}{lllllll}
\Xi\colon &\mathcal B &\to &\mathcal B^n
\\
&b&\mapsto&\Xi(b)=(\zeta_1(b),\ldots,\zeta_n(b))
\end{array}
$$
for $\zeta_k\in B_W$. Then, it is clear that 
$$
\Xi\big(\spn_{\mathbb Q}\big\{ \{b_{ij}\}_{1\leq i\leq j\leq m}\big\}\big)\subseteq\spn_{\mathbb Q}\big\{\{b_{ij}\}_{1\leq i\leq j\leq m}\big\}
$$
by~\eqref{eq:basis_form} and $\ker(\Xi)=\mathcal K$. We need only to find a non-zero form $b\in \mathcal P=\ker(\Xi)\cap\spn_{\mathbb Q}\big\{ \{b_{ij}\}_{1\leq i\leq j\leq m}\big\}$. Let us assume that $\tilde b\in \mathcal P$. Then we can write $\tilde b=\sum_{i\leq j}q_{ij}b_{ij}$ with $q_{ij}\in\mathbb Q$ and $\tilde b(v_{\alpha},v_{\beta})\in\mathbb Q$. Then, $\tilde b=cb$, where $b(.\,,.)=\la.\,,.\ra_{V}$, $c\neq 0$ by Step~3. If $c>0$, then the form $b^*$ has the same index $(p,q)$ as the original scalar product, and if $c<0$, then the index is $(q,p)$.

Denote $N=\dim(\mathcal B)$. The map $\Xi\colon\mathcal B\to\mathcal B^n$ and the basis $\{b_{ij}\}_{i\leq j}$ define the basis on $\mathcal B^n$ in a natural way. Then the $(nN\times N)$-matrix $A$ for the map $\Xi$ has rational entries by~\eqref{eq:basis_form}, and therefore, the determinant of any $(k\times k)$ sub-matrix belongs to $\mathbb Q$. Hence, $\rank_{\mathbb Q}(A)=\rank_{\mathbb R}(A)$ and $\ker_{\mathbb Q}(A)=\ker_{\mathbb R}(A)$. Because of $\dim(\mathcal K)=\dim(\ker(\Xi))=\ker_{\mathbb R}(A)=1$, we can find a non-zero element in $\ker (\Xi)\cap\spn_{\mathbb Q}\big\{ \{b_{ij}\}_{1\leq i\leq j\leq m}\big\}$ by Step~4. 
This proves the theorem.
\end{proof}

Applying the Mal'cev criterion we obtain the following corollary.

\begin{corollary}
If $\mathbb G$ is a simply connected 2-step nilpotent Lie group with the Lie algebra $\mathfrak G$, then the group $\mathbb G$ admits a lattice. 
\end{corollary}


Let $W\subset\End(V_j)$ for  $j=1,\ldots, n$, where each $V_j$ is $W$-irreducible finite dimensional space, which admits a $W$-invariant scalar product $\la.\,,.\ra_{V_j}$. Define the vector space $V=\bigoplus_{j=1}^{n} V_j$ and the scalar product $\la.\,,.\ra_{V}=\bigoplus_{j=1}^{n}\la.\,,.\ra_{V_j}$. Then the direct sum $V=\bigoplus_{j=1}^{n} V_j$ is orthogonal with respect to $\la.\,,.\ra_{V}$. Let $W$ be a rational subspace of a semisimple Lie algebra $\mathcal L$ having the following property:
\begin{itemize}
\item[($\mathfrak P$)] each vector space $V_j$ admits a basis $B_{V_j}$ and the Lie algebra admits a basis $\mathcal C_{\mathcal L}$ (for instance the Chevalley basis) such that  $\mathcal C_{\mathcal L}$ leaves the rational structure $\spn_{\mathbb Q}\{B_{V_j}\}$ invariant, $$\zeta_k(\spn_{\mathbb Q}\{B_{V_j}\})\subset\spn_{\mathbb Q}\{B_{V_j}\}$$ for all $\zeta_k\in \mathcal C_{\mathcal L}$.
\end{itemize} 
Let us remark, that if $\mathcal L$ is a semisimple Lie algebra of a compact subgroup $\mathbb G$ of $\GL(V_j)$ (which is the case when the Killing form on $\mathcal L$ is positive-definite), then the representation $\rho\colon \mathbb G\to GL(V_j)$ has the property that the vector space $V_j$ admits the basis $B_{V_j}$, such that $d\rho(\mathcal C_{\mathcal L})$ leaves  the integer structure $\spn_{\mathbb Z}\{B_{V_j}\}$  invariant, and as a consequence, leaves  the rational structure $\spn_{\mathbb Q}\{B_{V_j}\}$ invariant. Thus, for the Lie algebras of compact Lie  groups
the property (1) always holds. 

Now we are ready to prove the following theorem. 

\begin{theorem}\label{th:7}
Let $\mathcal L\subset\orth(V,\la.\,,.\ra_V)$ be a subalgebra that admits a Chevalley basis $\mathcal C_{\mathcal L}$ such that the structure constants with respect to this basis are rational, let $W$ be a rational subspace of $\mathcal L$ relatively to the rational structure $\spn_{\mathbb Q}\{\mathcal C_{\mathcal L}\}$. Assume also that $V=\oplus V_j$, where each $V_j$ is $W$-irreducible and admits $W$-invariant scalar product $\la.\,,.\ra_{V_j}$. Here $W\subset\End(V_j)$ for each $j=1,\ldots, n$. Moreover, the basis $\mathcal C_{\mathcal L}$ is such that for any $V_j$ there is a basis $B_{V_j}$ such that $\zeta_k(\spn_{\mathbb Q}\{B_{V_j}\})\subset\spn_{\mathbb Q}\{B_{V_j}\}$ for all $\zeta_k\in \mathcal C_{\mathcal L}$.
Then there exists a scalar product $\la.\,,.\ra_{\mathcal G}$ on $V\oplus W$, such that the standard pseudo-metric Lie algebra $\mathcal G=(V\oplus W,[.\,,.],\la.\,,.\ra_{\mathcal G})$ admits rational structure constants. 
\end{theorem}

\begin{proof}
Let $B_W=\{\zeta_1,\ldots,\zeta_k\}\subset \spn_{\mathbb Q}\{\mathcal C_{\mathcal L}\}$. By the hypothesis of the theorem $W\subset \orth(V_j,\la.\,,.\ra_{V_j})\subset\End(V_j)$ and for any $V_j$ we can find a basis $B_{V_j}$ such that matrices for all $\zeta_j\in \mathcal C_{\mathcal L}$ have rational entries when they are written in the bases $B_{V_j}$ for all $j$. Since $B_{W}\subset\spn_Q\{\mathcal C_{\mathcal L}\}$, the subspace $W\subset\End(V_j)$ is the rational  relatively to the bases $B_W$ and $B_{V_j}$ for each $j$. Then, for any $V_j$ we can modify the scalar products $\la.\,,.\ra_{V_j}$ such that $\la v_{\alpha},v_{\beta}\ra_{V_j}\in\mathbb Q$ for any two elements $v_{\alpha},v_{\beta}\in B_{V_j}$. Let $B_V=\{v_1,\ldots,v_m\}$ be a union of bases $\{B_{V_j}\}$, and let the scalar product $\la.\,,.\ra_V^*$ be the sum of modified scalar products, that makes the direct sum $\oplus V_j$ orthogonal. Then the bases $B_V$ and $B_W$ satisfy the conditions of Step 2 of the previous theorem, and therefore, the Lie algebra $V\oplus W$ has rational structure constants.
\end{proof}

\begin{theorem}\label{th:8}
Let $(V,\la.\,,.\ra_V)$ be a finite dimensional scalar product space. Let $W$ be a Lie triple system in $\orth(V,\la.\,,.\ra_V)$ that has a trivial centre. If the vector space $V$ and the Lie algebra $\mathcal L=W+[W,W]$ satisfy condition ($\mathfrak P$) described before Theorem~\ref{th:7}, then the standard pseudo-metric 2-step nilpotent Lie algebra $\mathcal G=V\oplus W$ admits rational structure constants. 
\end{theorem}
\begin{proof}
The Lie algebra $\mathcal L=W+[W,W]$ is a subalgebra of $\orth(V,\la.\,,.\ra_V)$. We have shown in Section 6 that the Lie algebra $\mathcal L$ has a basis $B_{\mathcal L}$ such that the structure constants of $\mathcal L$ are rational, and moreover, $W$ is a rational subspace of $\mathcal L$ with respect to the rational structure $\spn_{\mathbb Q}\{B_{\mathcal L}\}$.  Then we apply Theorem~\ref{th:7} and finish the proof.
\end{proof}
 \begin{corollary}
If $G$ is a simply connected 2-step nilpotent Lie group with the Lie algebra $\mathcal G$ described in Theorem~\ref{th:8}, then $G$ admits a lattice.
 \end{corollary}
 
Now we are ready to show an important consequence of the theory developed above. 
Let us make the following observation. It was shown that any pseudo $H$-type Lie algebra $\mathfrak n_{r,s}$ arises from representation of the Clifford algebra $\Cl_{r,s}$. Thus, if there are $2l\times 2l$-matrices $J_j$, $j=1,\ldots,r+s$ satisfying the condition 
\begin{itemize}
\item {$J_j^2=-\Id_{\mathbb R^{2l}}$, $j=1,\ldots,r$; }
\item {$J_j^2=\Id_{\mathbb R^{2l}}$, $j=r+1,\ldots,r+s$; }
\item {$J_jJ_i=-J_iJ_j$, $j\neq i$; }
\end{itemize}  
then the corresponding pseudo $H$-type algebra $\mathfrak n_{r,s}$ exists. It is known~\cite{Wolfe} that the matrices satisfying the above conditions exist and moreover, they can be chosen with integer entries. Thus, the space $W$ has the basis $(J_1,\ldots, J_{r+s})$ and the space $[W,W]$ is spanned by $J_iJ_j$, $i,j=1,\ldots,r+s$. Thereby, we see that the Lie algebra $\mathcal L=W+[W,W]$ admits a basis of $(2l\times 2l)$-matrices having integer entries, and moreover, the Lie algebra in the basis $\{J_j\}_{j=1}^{r+s}$ admits integer structure constants, and the space $W$ is a rational subspace of the Lie algebra $\mathcal L$, see Theorem~\ref{th:trans}. This basis leaves invariant the rational span of the standard Euclidean basis of $\mathbb R^{2l}$, and therefore, satisfies condition ($\mathfrak P$) before Theorem~\ref{th:7}. Here we substitute the Chevalley basis by the basis directly related to the representation of the Clifford algebras and the representation space is considered to be $\mathbb R^{2l}$. Since all pseudo $H$-type algebras are isomorphic to pseudo $H$-type algebras $\mathfrak n_{r,s}$ arising from the representation of the Clifford algebras $\Cl_{r,s}$, we only need to prove the following theorem.

\begin{theorem}\label{th:int_str_const}
Let $\n_{r,s}$ be a pseudo $H$-type Lie algebra, and let $\mathfrak N_{r,s}$ be the corresponding pseudo $H$-type Lie group. Then $\mathfrak N_{r,s}$ admits a lattice.
\end{theorem}
\begin{proof}
Let $\n_{r,s}=(\mathbb R^{l,l}\oplus\mathbb R^{r,s},[.\,,.],\la.\,,.\ra_{\n}=\la.\,,.\ra_{l,l}+\la.\,,.\ra_{r,s})$ be a pseudo $H$-type Lie algebra, let $\Cl_{r,s}$ be the Clifford Lie algebra, and let $J\colon \Cl_{r,s}\to\End(\mathbb R^{l,l})$ be a representation defining the commutators in $\n_{r,s}$: $\la Z,[v,v']\ra_{r,s}=\la J_Zv,v'\ra_{l,l}$. Then $W=J(\mathbb R^{r,s})\subseteq\so(l,l)\subset\End(\mathbb R^{l,l})$ is the Lie triple system of $\so(l,l)$ having  a trivial centre. 
Let now $\mathcal G=(\mathbb R^{l,l}\oplus W,[.\,,.]^*,\la.\,,.\ra_{\mathcal G})$ be a standard pseudo-metric 2-step nilpotent Lie algebra
with $\la.\,,.\ra_{\mathcal G}=\la.\,,.\ra_{l,l}+\la.\,,.\ra_{\so(l,l)}$, and $\la \zeta,[v,v']^*\ra_{\so(l,l)}=\la \zeta(v),v'\ra_{l,l}$ for any $\zeta\in W$. The Lie algebra $\mathcal G$ admits rational structure constants, see Theorem~\ref{th:8}. We need to show that the Lie algebras $\n_{r,s}=(\mathbb R^{l,l}\oplus\mathbb R^{r,s},[.\,,.])$ and $\mathcal G=(\mathbb R^{l,l}\oplus W,[.\,,.]^*)$ are isomorphic. To achieve the goal we will construct auxiliary Lie algebra $\mathfrak G$ that will be isomorphic to both Lie algebras $\n_{r,s}$ and $\mathcal G$.

In order to construct the Lie algebra $\mathfrak G$ we use the injectivity property of the Clifford representations $J\colon \mathbb R^{r,s}\to\End(\mathbb R^{l,l})$. We disregard the standard scalar product $\la.\,,.\ra_{r,s}$ on $\mathbb R^{r,s}$ and simply write $\mathbb R^{r+s}$. Pullback the metric $\la .\,,.\ra_{\so(l,l)}$ to the space $\mathbb R^{r+s}$ by defining the scalar product $2l\la Z,Z'\ra^{'}_{\mathbb R^{r+s}}=\la J_Z,J_{Z'}\ra_{\so(l,l)}$. Let $\mathfrak G=\mathbb R^{2l}\oplus \mathbb R^{r+s}$ as a vector space and the commutator $[.\,,.]'$ defined by $\la Z,[v,w]'\ra'_{\mathbb R^{r+s}}=\la J_Zv,w\ra_{l,l}$. Let $\phi\colon \mathfrak G\to \mathcal G$ be the map constructed by $\phi(v+Z)=v+J_Z$ for all $v\in\mathbb R^{2l}$, $Z\in \mathbb R^{r+s}$. The map $\phi$ is the Lie algebra isomorphism $\mathfrak G=(\mathbb R^{l,l}\oplus \mathbb R^{r+s},[.\,,.]')$ and $\mathcal G=(\mathbb R^{l,l}\oplus W,[.\,,.]^*)$. Indeed for any $\xi\in W$ and any $v,w\in\mathbb R^{l,l}$ we obtain
\begin{eqnarray*}
\la \xi,[v,w]^*\ra_{\so(l,l)} & = & \la\xi(v),w\ra_{l,l}=\la J_Z(v),w\ra_{l,l}=\la Z,[v,w]'\ra'_{\mathbb R^{r+s}}
\\
&=&\la J_Z,J_{[v,w]'}\ra_{\so(l,l)}
= \la \xi,\phi([v,w]')\ra_{\so(l,l)}.
\end{eqnarray*}

Now we show that the Lie algebras $\mathfrak G$ and $\mathfrak n_{r,s}$ are isomorphic.
Observe that since
$$\la J_{Z_i},J_{Z_i}\ra_{\so(l,l)}=-\tr(J_{Z_i}^2)=-2l\,\nu_i(r,s),$$
where $\{Z_1,\ldots, Z_{r+s}\}$ is an orthonormal basis of $\mathbb R^{r,s}$ with respect to $\la .\,,.\ra_{r,s}$, the set $\{J_{Z_1},\ldots,J_{Z_{r+s}}\}$ forms an orthogonal basis of $W$ with respect to the restriction of the trace metric. Thus, the index of the spaces $W$ and $\mathbb R^{r,s}$ coincides. This also shows that the collection $Z_1,\ldots, Z_{r+s}$ is also orthogonal in $\mathbb R^{r+s}$ with respect to the metric $\la .\,,.\ra^{'}_{\mathbb R^{r+s}}$. Therefore, the scalar products $\la .\,,.\ra^{'}_{\mathbb R^{r+s}}$ and $\la .\,,.\ra_{r,s}$ differs by the positive multiple $2l$. Now the Lie algebra $\mathfrak G$ with the metric $\la .\,,.\ra_{l,l}+\la .\,,.\ra^{'}_{\mathbb R^{r+s}}=\la .\,,.\ra_{l,l}+2l\la .\,,.\ra_{r,s}$ has the same Lie brackets as the Lie algebra $\mathfrak G$ with the scalar product $(2l)^{-1}\la .\,,.\ra_{l,l}+\la .\,,.\ra_{r,s}$ by Proposition~\ref{prop:uniq1}. The Lie brackets of $\mathfrak G$ and  $\n_{r,s}$ are defined by the scalar product $\la .\,,.\ra_{l,l}+\la .\,,.\ra_{r,s}$ and $(2l)^{-1}\la .\,,.\ra_{l,l}+\la .\,,.\ra_{r,s}$ respectively. Thus, $\n_{r,s}=(\mathbb R^{l,l}\oplus \mathbb R^{r,s},[.\,,.])$ and $\mathfrak G=(\mathbb R^{l,l}\oplus \mathbb R^{r+s},[.\,,.]')$ are isomorphic by Lemma~\ref{lem:uniq1}. 

Finally we conclude that the Lie algebras $\n_{r,s}=(\mathbb R^{l,l}\oplus \mathbb R^{r,s},[.\,,.])$ and $\mathcal G=(\mathbb R^{l,l}\oplus W,[.\,,.]^*)$ are isomorphic, and therefore, the Lie algebra $\n_{r,s}$ has rational structure constants. Applying the Mal'cev criterium we finish the proof. 
\end{proof}
Other proofs of Theorem~\ref{th:int_str_const} can be found in~\cite{FM,FMV}

Let us make the last observation. Let $\g$ be a $2$-step nilpotent Lie algebra such that $\dim([\g, \g ])=n$, and the complement $V$ to $[\g, \g]$ has dimension $m$. As we showed in Theorem~\ref{th:isom}, there exist $p,q \in \mathbb{N}$, $p+q=m$, and an $n$-dimensional subspace $\mathcal{D}$ of $\mathfrak{so}(p,q)$, such that $\g$ is isomorphic as a Lie algebra to the standard metric $2$-step nilpotent Lie algebra $\g^*=\mathbb{R}^{p,q}\oplus_{\bot}\mathcal{D}$. Now we state the following theorem.
 
 \begin{theorem}\label{th:last}
If $\g$ admits a basis with rational structure constants, then we may choose $\mathcal{D}$ having a basis whose matrices only have entries in $\mathbb{Z}$ relatively to the standard basis $e_1, \dotso,e_m$ of $\mathbb{R}^{p,q}$.
\end{theorem}
\begin{proof}
We assume that there exists a basis $\mathcal{B}=\{v_1, \dotso, v_m,z_1, \dotso,z_n\}$ of $\g=V \oplus_{\perp}[\g \,, \g]$, with $v_1, \dotso, v_m$ being a basis of $V$, and $z_1, \dotso,z_n$ being a basis of $[\g \,, \g]$ such that the structur constants $C_{ij}^k$ with respect to $\mathcal{B}$ are in $\mathbb{Q}$. We write $C_{ij}^k=\frac{a_{ij}^k}{b_{ij}^k}$ with $a_{ij}^k \in \mathbb{Z}$ and $b_{ij}^k \in \mathbb{N} \setminus \{0\}$. We define a natural number $d$ as the least common multiple of the collection $\{b_{ij}^k \vert i,j=1, \dotso,m, \quad k=1, \dotso,n\}$, and define the basis $\mathcal{B}_{d}=\{\sqrt{d}v_1, \dotso,\sqrt{d}v_m,z_1, \dotso,z_n\}$. It follows that the structure constants $\tilde C _{ij}^k$ with respect of $\mathcal B_d$ are given by $d C_{ij}^k$ as
$$ \sum_{k=1}^n \tilde C_{ij}^k z_k= [\sqrt{d}v_i \,, \sqrt{d} v_j]=d[v_i \,, v_j]=d\sum_{k=1}^n C_{ij}^k z_k=\sum_{k=1}^n dC_{ij}^k z_k.$$ 
Hence, $\tilde C _{ij}^k$ are  natural numbers such that the matrix $\tilde C^k=d C^k$ only has entries in $\mathbb{Z}$. As we know from Theorem~\ref{th:isom}, there exist $p,q \in \mathbb{N}$, $p+q=m$ such that the $n$-dimensional subspace $\mathcal{D}=\spn\{C^1 \eta_{p,q}, \dotso,C^k \eta_{p,q}\}$ is a non-degenerate subspace of $\mathfrak{so}(p,q)$ such that $\g \cong \mathbb{R}^{p,q} \oplus \mathcal{D}$. As $\tilde C^k \eta_{p,q}=d C^k \eta_{p,q} \in \mathcal{D}$, and the entries of $\tilde C^k \eta_{p,q}$ lie obviously in $\mathbb{Z}$, it follows that there exists a basis of $\mathcal D$ whose matrices only have entries in $\mathbb{Z}$, relatively to the standard basis $e_1, \dotso,e_m$ of $\mathbb{R}^{p,q}$.
\end{proof}



\begin{thebibliography}{9}

\bibitem{AFM}
C.~Autenried, K.~Furutani, I.~Markina,  
{\it Classification of pseudo $H$-type algebras},  arXiv 1410.3244, 2014.

\bibitem{Atiyah} 
M.~F.~Atiyah, R.~Bott,  A.~Shapiro, 
{\it Clifford modules}, 
Topology {\bf 3} (1964) suppl. 1, 3--38.

\bibitem{Bott}
R.~Bott, 
{\it The stable homotopy of the classical groups}, 
Ann.  Math. (2),  {\bf 70} (1959), no. 2, 313--337.

\bibitem{CCFI}
O.~Calin, D.~C.~Chang, K.~Furutani, C.~Iwasaki,  
{\it Heat kernels for elliptic and sub-elliptic operators. 
Methods and techniques}, 
Applied and Numerical Harmonic Analysis. Birkh\"auser/Springer, New York, 2011. 433 pp.

\bibitem{CDPT}
L.~Capogna, D.~Danielli, S.~D.~Pauls, J.~T.~Tyson,  
{\it An introduction to the Heisenberg group and the sub-Riemannian isoperimetric problem}, 
Progress in Mathematics, 259. Birkh\"auser Verlag, Basel, 2007. 223 pp.

\bibitem{CM} 
D.~C.~Chang, I.~Markina, 
{\it Geometric analysis on quaternion $\mathbb H$-type groups}, 
J. Geom. Anal. {\bf 16} (2006), no. 2, 265--294.

\bibitem{Chern}
S.~-S.~Chern, C.~Chevalley, 
{\it \'Elie Cartan and his mathematical work}, 
Bull. Amer. Math. Soc. {\bf 2} (1952), 217--250.

\bibitem{Chev}
C.~Chevalley, 
{\it Sur certains groupes simples},
 (French) Tohoku Math. J. {\bf 7 (2)} (1955), 14--66.

\bibitem{Ciatti}
P.~Ciatti, 
{\it Scalar products on Clifford modules and pseudo-H-type Lie algebras},  
Ann. Mat. Pura Appl. (4) {\bf 178} (2000), 1--31.

\bibitem{Ciatti1}
P.~Ciatti, 
{\it  Spherical distributions on harmonic extensions of pseudo-H-type groups},
J. Lie Theory {\bf 7} (1997), no. 1, 1--28.

\bibitem{Cliff}
W.~K.~Clifford, 
{\it Applications of Grassmann's extensive algebra}, 
Amer. J. Math. {\bf 1} (1878), 350--358.

\bibitem{CP}
L.~A.~Cordero, P.~E.~Parker, 
{\it Isometry groups of pseudoriemannian 2-step nilpotent Lie groups}, 
Houston J. Math. {\bf 35} (2009), no. 1, 49--72.

\bibitem{CDKR}
M.~Cowling, A.~H.~Dooley, A.~Kor\'{a}nyi, F.~Ricci, 
{\it $H$-type groups and Iwasawa decompositions},  
Adv. Math.  \textbf 87 (1991), no. 1, 1--41.

\bibitem{Eber04}
P.~Eberlein,  
{\it Geometry of 2-step nilpotent Lie groups}, 
Modern dynamical systems and applications, Cambridge Univ. Press, Cambridge (2004), 67--101.

\bibitem{Eber03}
P.~Eberlein,  {\it Riemannian submersion and lattices in 2-step nilpotent Lie groups}, 
Comm. Anal. Geom. {\bf 11} (2003), no.~3, 441--488.

\bibitem{Eber02}
P.~Eberlein,  
{\it The moduli space of 2-step nilpotent Lie algebras of type $(p,q)$}, 
Exploration in Complex and Riemannian Geometry, 37--72.

\bibitem{FM}
K.~Furutani, I.~Markina,
{\it Existence of the lattice on general $H$-type groups},
J. Lie theory, {\bf 24}, (2014), 979--1011.

\bibitem{FMV}
K.~Furutani, I.~Markina, A.~Vasiliev,
{\it Free nilpotent and H-type Lie algebras. Combinatorial and orthogonal designs.} 
J. Pure and Appl. Algebra,  doi:10.1016/j.jpaa.2015.05.027,  arXiv 1410.3767.

\bibitem{Garling}
 D.~J.~H.~Garling, 
 {\it Clifford algebras. An introduction}, 
 London Mathematical Society Student Texts, 78, Cambridge: Cambridge University Press, 2011.
 
\bibitem{Gaveau} 
B.~Gaveau, 
{\it Principe de moindre action, propagation de la chaleur et estim\'ees sous elliptiques sur certains groupes nilpotents}, 
Acta Math. {\bf 139} (1977), no. 1--2, 95--153.

\bibitem{GKM}
M.~Godoy Molina, A.~Korolko, I.~Markina,
{\it Sub-semi-Riemannian geometry of general $H$-type groups},  
Bull. Sci. Math. {\bf 137} (2013), no. 6, 805--833.

\bibitem{JPP} 
C.~Changrim. P.~E.~Parker, K.~Park, 
{\it Pseudo $H$-type 2-step nilpotent Lie groups}, 
Houston J. Math. {\bf 31} (2005), no. 3, 765--786 (electronic). 

\bibitem{HolgerKammeyer}
H.~Kammeyer, 
{\it An explicit rational structure for real semisimple Lie algebras}, 
J. Lie Th., {\bf 24} (2014), no. 2, 307--319.

\bibitem{Kaplan}
A.~Kaplan, 
{\it Fundamental solutions for a class of hypoelliptic PDE generated by composition of quadratic forms}, 
Trans. Amer. Math. Soc. {\bf 258} (1980), 147--153.

\bibitem{Kap2}
A.~Kaplan, 
{\it On the geometry of groups of Heisenberg type},
Bull. London Math. Soc. {\bf 15} (1983), no. 1, 35--42.

\bibitem{Kor1}
A.~Kor\'{a}nyi, 
{\it Geometric properties of Heisenberg-type groups}, 
Adv. in Math. {\bf 56} (1985),  no. 1, 28--38.

\bibitem{Knapp}
A.~W.~Knapp, 
{\it Lie groups. Beyond an introduction}, 
2-nd ed. Progress in Mathematics, v.~140, Birkh\"auser, 2002, 812 pp..

\bibitem{Lam}
T.~Y.~Lam, 
{\it The algebraic theory of quadratic forms},
Mathematics Lecture Note Series. W. A. Benjamin, Inc., Reading, Mass., 1973.

\bibitem{Malc}
A.~I.~Mal'cev, 
{\it On a class of homogeneous spaces},  
Amer. Math. Soc. Translation (1951). no. 39, 33 pp.; Izvestiya Akad. Nauk. SSSR. Ser. Mat. {\bf 13}, (1949), 9--32.

\bibitem{ONeill}
B.~O'Neill, 
{\it Semi-Riemannian geometry}, 
Academic Press, Elsevier 1983.

\bibitem{Riemann}
H.~M.~Riemann, 
{\it $H$-type groups and Clifford modules}, 
Adv. Appl. Clifford Alg. {\bf 11} (2001), no. 2, 277--288.

\bibitem{Riemann1}
H.~M.~Riemann, 
{\it  Rigidity of H-type groups}, 
Math. Z. {\bf 237} (2001), no. 4, 697--725. 

\bibitem{Wolfe}
W.~Wolfe, 
{\it Amicable orthogonal designs--existence}, 
Canad. J. Math. {\bf 28} (1976), no. 5, 1006--1020.

\end{thebibliography}
\end{document}